\documentclass[aihp]{imsart}

\startlocaldefs
\endlocaldefs

\usepackage[utf8]{inputenc}
\usepackage[LGR,T1]{fontenc}
\usepackage{color}
\usepackage{xcolor}
\usepackage{amsmath}
\usepackage{amsthm}
\usepackage{amssymb}
\usepackage{graphicx}
\usepackage{wasysym}
\usepackage{bbm}
\usepackage[unicode=true,
 bookmarks=true,bookmarksnumbered=true,bookmarksopen=true,bookmarksopenlevel=3,
 breaklinks=false,pdfborder={0 0 1},backref=false,colorlinks=true]
 {hyperref}
\hypersetup{
 pdfauthor={Lei Yu},
 pdfborderstyle=,pdfborderstyle={},pdfborderstyle={},pdfborderstyle={},pdfborderstyle={},pdfborderstyle={},pdfborderstyle={},pdfborderstyle={},pdfborderstyle={},pdfborderstyle={},pdfborderstyle={},pdfborderstyle={},pdfborderstyle={},pdfborderstyle={},pdfborderstyle={},pdfborderstyle={},pdfborderstyle={},pdfborderstyle={},pdfpagelayout=OneColumn,pdfnewwindow=true,pdfstartview=XYZ,plainpages=false,linkcolor=blue,urlcolor=blue,citecolor=red,anchorcolor=blue,linkcolor=blue,urlcolor=blue,citecolor=red,anchorcolor=blue}

\usepackage{mathtools}  
\mathtoolsset{showonlyrefs} 

\makeatletter
\theoremstyle{plain}
\newtheorem{thm}{\protect\theoremname}
\theoremstyle{plain}
\newtheorem{defn}{\protect\definitionname}
\theoremstyle{plain}
\newtheorem{lem}{\protect\lemmaname}
\theoremstyle{plain}
\newtheorem{cor}{\protect\corollaryname}
\theoremstyle{remark}

\theoremstyle{remark}
\newtheorem{rem}{\protect\remarkname}

\usepackage{color}
\usepackage{cite}
\usepackage{float}
\usepackage{tikz}
\usetikzlibrary{patterns,arrows}

\newtheorem{property}{Property}\newtheorem{assumption}{Assumption}
 
\allowdisplaybreaks[1]
\newcommand{\bone}{\mathbbm{1}}

\providecommand{\claimname}{Claim}
\providecommand{\corollaryname}{Corollary}
\providecommand{\definitionname}{Definition}
\providecommand{\lemmaname}{Lemma}

\providecommand{\remarkname}{Remark}
\providecommand{\theoremname}{Theorem}

\@ifundefined{showcaptionsetup}{}{%
 \PassOptionsToPackage{caption=false}{subfig}}
\usepackage{subfig}
\makeatother

\providecommand{\claimname}{Claim}
\providecommand{\corollaryname}{Corollary}
\providecommand{\definitionname}{Definition}
\providecommand{\lemmaname}{Lemma}
\providecommand{\remarkname}{Remark}
\providecommand{\theoremname}{Theorem}

\begin{document}
\begin{frontmatter} 
\title{Asymptotics for Strassen's Optimal Transport Problem}

\maketitle
\runtitle{Asymptotics for Strassen's OT Problem}  

\begin{aug} 
\author{\snm{Lei} \fnms{Yu}\thanks{The author was supported by the NSFC grant 62101286 and the Fundamental
Research Funds for the Central Universities of China (Nankai University).} \ead[label=e1]{leiyu@nankai.edu.cn}}

\affiliation{Nankai University}

\address{School of Statistics and Data Science, LPMC \& KLMDASR, 
Nankai University, Tianjin 300071,  
China, 
\printead{e1}}
 
\end{aug}

\begin{abstract}
In this paper, we consider Strassen's version of optimal transport
(OT) problem, which concerns minimizing the excess-cost probability
(i.e., the probability that the cost is larger than a given value)
over all couplings of two given distributions. We derive large deviation,
moderate deviation, and central limit theorems for this problem. Our
proof is based on Strassen's dual formulation of the OT problem, Sanov's
theorem on the large deviation principle (LDP) of empirical measures,
as well as the moderate deviation principle (MDP) and central limit
theorems (CLT) of empirical measures. In order to apply the LDP, MDP,
and CLT to Strassen's OT problem, nested formulas for Strassen's OT
problem are derived. Based on these nested formulas and using a splitting
technique, we construct asymptotically optimal solutions to Strassen's
OT problem and its dual formulation. 
\end{abstract}

\begin{abstract}
Dans cet article, nous considérons la version de Strassen du problème
de transport optimal (OT), qui concerne la minimisation de la probabilité
de surcoût (c'est-à-dire la probabilité que le coût soit supérieur
à une valeur donnée) sur tous les couplages de deux distributions
données. Nous obtenons  des théorèmes de grande déviation, de déviation
modérée et de limite centrale pour ce problème. Notre preuve est basée
sur la formulation duale de Strassen 
du problème OT, le théorème
de Sanov sur le principe de grande déviation (LDP) des mesures empiriques,
ainsi que le principe de déviation modérée (MDP) et les théorèmes
centraux limites (CLT) des mesures empiriques. Afin d'appliquer les
LDP, MDP et CLT au problème OT de Strassen, des formules imbriquées
pour le problème OT de Strassen sont établies. 
Sur la base de ces
formules imbriquées et en utilisant une technique de division, nous
construisons des solutions asymptotiquement optimales au problème
OT de Strassen et à sa formulation duale.  
\end{abstract}

\begin{keyword}[class=MSC] \kwd[Primary ]{60F10} \kwd{60F05}
\kwd[; secondary ]{49Q22} \end{keyword}

\begin{keyword} Optimal Transport, Large Deviation, Moderate Deviation,
Central Limit Theorem \end{keyword}

\end{frontmatter} \tableofcontents{}

\section{Introduction}

\label{sec:Introduction}

The theory of optimal transport (OT) has been studied for a long history
due to its importance to related problems in physics, mathematics,
economics and other areas; see e.g. \cite{monge1781memoire,kantorovich1942translocation,rachev1998mass,villani2003topics,villani2008optimal}.
Recently, OT theory has been applied increasingly in computer science,
mathematical imaging, machine learning, and information theory \cite{rubner2000earth,sommerfeld2018inference,yu2018asymptotic,schmitz2018wasserstein,del1999central,lopez2018generalization}.
The OT problem was introduced by Monge \cite{monge1781memoire} and
Kantorovich \cite{kantorovich1942translocation}, who defined it as
the problem of minimizing the \emph{expectation of a cost function}
over all couplings of two given distributions. Let $(\mathcal{X},\tau_{1})$
and $(\mathcal{Y},\tau_{2})$ be Polish spaces. Let $\Sigma(\mathcal{X})$
and $\Sigma(\mathcal{Y})$ be respectively the Borel $\sigma$-algebras
on $\mathcal{X}$ and $\mathcal{Y}$ that are generated by the topologies
$\tau_{1}$ and $\tau_{2}$. Let $\mathcal{P}(\mathcal{X})$ and $\mathcal{P}(\mathcal{Y})$
denote the sets of probability measures (or \emph{distributions})
on $\mathcal{X}$ and $\mathcal{Y}$ respectively. Let $P_{X}\in\mathcal{P}(\mathcal{X})$
and $P_{Y}\in\mathcal{P}(\mathcal{Y})$ be two distributions where
the subscripts $X$ and $Y$ indicate which spaces the distributions
are defined on. 
The \emph{coupling set} of $(P_{X},P_{Y})$ is defined as 
\[
\Pi(P_{X},P_{Y}):=\Biggl\{\begin{array}{l}
P_{XY}\in\mathcal{P}(\mathcal{X}\times\mathcal{Y}):\,P_{XY}(A\times\mathcal{Y})=P_{X}(A),\forall A\in\Sigma(\mathcal{X}),\\
\qquad\qquad\qquad\qquad\,P_{XY}(\mathcal{X}\times B)=P_{Y}(B),\forall B\in\Sigma(\mathcal{Y})\:
\end{array}\Biggr\}.
\]
Distributions in $\Pi(P_{X},P_{Y})$ are termed \emph{couplings} of
$(P_{X},P_{Y})$. Let $c:\mathcal{X}\times\mathcal{Y}\to[0,+\infty]$
be lower semi-continuous, which is called \emph{cost function}.

\begin{defn}
The \emph{OT cost }between $P_{X}$ and $P_{Y}$ is defined as\footnote{The existence of the minimizers are well-known; see, e.g., \cite[Theorem 1.3]{villani2003topics}.
Furthermore, when the (joint) distribution of the random variables
involved in an expectation is clear from context, we will omit the
subscript ``$(X,Y)\sim P_{XY}$''. } 
\begin{equation}
\mathcal{E}(P_{X},P_{Y}):=\min_{P_{XY}\in\Pi(P_{X},P_{Y})}\mathbb{E}_{(X,Y)\sim P_{XY}}[c(X,Y)].\label{eq:OT}
\end{equation}
Any $P_{XY}\in\Pi(P_{X},P_{Y})$ attaining $\mathcal{E}(P_{X},P_{Y})$
is called an \emph{OT plan}. 
\end{defn}
The minimization problem in \eqref{eq:OT} is called \emph{Monge--Kantorovich's
OT problem }\cite{villani2003topics}. The functional $(P_{X},P_{Y})\in\mathcal{P}(\mathcal{X})\times\mathcal{P}(\mathcal{Y})\mapsto\mathcal{E}(P_{X},P_{Y})\in[0,+\infty]$
is called the \emph{OT (cost) functional. }If $(\mathcal{X},\tau_{1})=(\mathcal{Y},\tau_{2})$
and $c=d^{p}$ where $p\ge1$ and $d$ is a metric on the Polish space
$(\mathcal{X},\tau_{1})$, then $W_{p}(P_{X},P_{Y}):=(\mathcal{E}(P_{X},P_{Y}))^{1/p}$
is the\emph{ }so-called\emph{ $p$-th Wasserstein distance} between
$P_{X}$ and $P_{Y}$. In \cite{kantorovich1942translocation}, Kantorovich
provided a dual formulation for Monge--Kantorovich's OT problem,
which is known as the \emph{Kantorovich duality theorem} in the literature.
Define the $c$-transform of a function $\phi:\mathcal{X}\to\mathbb{R}$
as $\phi^{c}(y)=\inf_{x\in\mathcal{X}}\phi(x)+c(x,y)$ for all $y\in\mathcal{Y}$.
\begin{thm}[Kantorovich Duality]
\cite[Theorem 5.10]{villani2008optimal} \label{thm:Kantorovich}
It holds that 
\begin{align}
\mathcal{E}(P_{X},P_{Y}) & =\sup_{(\phi,\psi)\in C_{\mathrm{b}}(\mathcal{X})\times C_{\mathrm{b}}(\mathcal{Y}):\phi+\psi\leq c}\int_{\mathcal{X}}\phi\ \mathrm{d}P_{X}+\int_{\mathcal{Y}}\psi\ \mathrm{d}P_{Y}\label{eq:-111}\\
 & =\sup_{\phi\in L^{1}(P_{X})}\int_{\mathcal{Y}}\phi^{c}\ \mathrm{d}P_{Y}-\int_{\mathcal{X}}\phi\ \mathrm{d}P_{X}\label{eq:dual}
\end{align}
where $C_{\mathrm{b}}(\mathcal{X})$ denotes the collection of bounded
continuous functions $\phi:\mathcal{X}\to\mathbb{R}$, and $L^{1}(P_{X})$
denotes the collection of integrable functions $\phi:\mathcal{X}\to\mathbb{R}$
with respect to the distribution $P_{X}$. 
\end{thm}
In 1965, Strassen \cite{strassen1965existence} considered an excess-cost
probability version of the OT problem, in which the \emph{excess-cost
probability}, instead of the expectation, is to be minimized, as shown
in the following definition. Here the excess-cost probability refers
to the probability that the cost function is larger than a given value. 
\begin{defn}
For $\alpha\ge0$, the \emph{optimal excess-cost probability (ECP)
}between $P_{X}$ and $P_{Y}$ with respect to the cost function $c$
is defined as\footnote{Similarly to, and also as a specific case of the last footnote, for
brevity, when the (joint) distribution of the random variables involved
in a probability is clear from context, we will also omit the subscript
``$(X,Y)\sim P_{XY}$''.} 
\begin{equation}
\mathcal{G}_{\alpha}(P_{X},P_{Y}):=\min_{P_{XY}\in\Pi(P_{X},P_{Y})}\mathbb{P}_{(X,Y)\sim P_{XY}}\{c(X,Y)>\alpha\}.\label{eq:strassen}
\end{equation}
Any $P_{XY}\in\Pi(P_{X},P_{Y})$ attaining $\mathcal{G}_{\alpha}(P_{X},P_{Y})$
is called an \emph{optimal ECP plan}. 
\end{defn}
We term the minimization problem in \eqref{eq:strassen} as \emph{Strassen's
OT problem}. In fact, Strassen's OT problem is a $\{0,1\}$-valued
cost version of Monge--Kantorovich's OT problem in which the cost
function is set to the indicator function $(x,y)\mapsto\bone_{c(x,y)>\alpha}$
(rather than $c$ itself). Moreover, $(x,y)\mapsto\bone_{c(x,y)>\alpha}$
is lower semi-continuous since $c$ is lower semi-continuous. Hence,
we write ``minimization'', instead of ``infimization'', in \eqref{eq:strassen}.
Furthermore, in Strassen's OT problem, the optimal ECP reduces to
the total variation (TV) distance if we set $(\mathcal{X},\tau_{1})=(\mathcal{Y},\tau_{2})$,
$\alpha=0$, and set the cost function $c$ to the discrete metric
$(x,y)\mapsto\bone_{x\neq y}$ \cite{villani2003topics}. That is,
for this case, 
\begin{equation}
\mathcal{G}_{0}(P_{X},P_{Y})=\min_{P_{XY}\in\Pi(P_{X},P_{Y})}\mathbb{P}\{X\neq Y\}=\Vert P_{X}-P_{Y}\Vert_{\mathrm{TV}}.\label{eqn:maxcoupling}
\end{equation}
Here $\Vert P-Q\Vert_{\mathrm{TV}}:=\sup_{A}P(A)-Q(A)$ denotes the
TV distance between two distributions $P$ and $Q$ defined on the
same measurable space, where the supremum is taken with respect to
all possible measurable sets $A$. Similarly to Monge--Kantorovich's
OT problem, Strassen's OT problem also admits a dual representation,
which was first given by Strassen \cite{strassen1965existence}. Such
a dual representation can be seen as a particular form (i.e., the
$\{0,1\}$-valued cost version) of the Kantorovich duality theorem.
If $\{(x,y):c(x,y)>\alpha\}=\emptyset$ for given $c$ and $\alpha$,
then, obviously $\mathcal{G}_{\alpha}(P_{X},P_{Y})=0$. To exclude
this trivial case, we make the following assumption.

\assumption[Nonempty] \label{ass:Nonempty-Condition-We} { \emph{We
assume that $\alpha$ is a number such that $\{(x,y):c(x,y)>\alpha\}$
is nonempty.}} 
\begin{thm}[Strassen Duality]
\cite[Theorem 5.4.1]{rachev1998mass}\cite[Corollary 1.28]{villani2003topics}\label{thm:strassen}
Under Assumption \ref{ass:Nonempty-Condition-We}, it holds that 
\begin{align}
\mathcal{G}_{\alpha}(P_{X},P_{Y}) & =\sup_{\textrm{closed }E,F:c(x,y)>\alpha,\forall x\in E,y\in F}P_{X}(E)+P_{Y}(F)-1\label{eq:-19}\\
 & =\sup_{\textrm{compact }E}P_{X}(E)-P_{Y}(\Gamma_{c\le\alpha}(E)),\label{eq:-20}
\end{align}
where for any set $A\subseteq\mathcal{X}$, denote the $\alpha$-enlargement
of $A$ under the cost function $c$ as 
\begin{equation}
\Gamma_{c\le\alpha}(A):=\bigcup_{x\in A}\{y\in\mathcal{Y}:c(x,y)\leq\alpha\}.\label{eq:Gamma}
\end{equation}
\end{thm}
In \eqref{eq:-19}, ``closed $E,F$'' can be replaced by ``measurable
$E,F$'', ``compact $E,F$'' (by inner regularity of probability
measures on Polish spaces), or ``open $E,F$'' (by the  tube lemma
\cite[Lemma 26.8]{munkres2000topology}: if $\mathcal{X}$ is any
topological space and $\mathcal{Y}$ a compact space, then the projection
map ${\displaystyle \mathcal{X}\times\mathcal{Y}\to\mathcal{X}}$
is closed). By the tube lemma, $\Gamma_{c\le\alpha}(E)$ is closed
in $\mathcal{Y}$ for any compact $E\subseteq\mathcal{X}$.  From
\eqref{eq:-20}, given $(P_{X},P_{Y})$, $\mathcal{G}_{\alpha}(P_{X},P_{Y})$
is right-continuous (and obviously non-increasing) in $\alpha$. Furthermore,
by symmetry, it also holds that $\mathcal{G}_{\alpha}(P_{X},P_{Y})=\sup_{\textrm{compact }F}P_{Y}(F)-P_{X}(\Gamma_{c\le\alpha}(F)).$

\subsection{\label{subsec:Main-Result-1}Main Result 1: Large Deviations Principle}

Studying the asymptotic behavior of a sequence of random variables
or probability distributions are central topics in probability theory.
Although the OT theory has been widely studied in the literature,
the asymptotic behaviors of OT problems have been rarely investigated.
This is the major motivation for us to write this paper. In this paper,
we investigate the asymptotic behavior of Strassen's OT problem. To
this end, we need first define the $n$-dimensional Strassen's OT
problem. Denote $\mathcal{X}^{n}$ as the $n$-fold product space
of $\mathcal{X}$. For the product space $\mathcal{X}^{n}\times\mathcal{Y}^{n}$,
we consider an additive cost function $c_{n}$ on $\mathcal{X}^{n}\times\mathcal{Y}^{n}$,
which is given by 
\[
c_{n}(x^{n},y^{n}):=\frac{1}{n}\sum_{i=1}^{n}c(x_{i},y_{i})\quad\mbox{for }(x^{n},y^{n})\in\mathcal{X}^{n}\times\mathcal{Y}^{n},
\]
where $c$ is the cost function given above which is independent of
$n$. 
Obviously, $c_{n}$ is lower semi-continuous since $c$ is lower semi-continuous.
For $\alpha\ge0$, the optimal ECP between the $n$-fold products
of $P_{X}$ and $P_{Y}$ with respect to the cost function $c_{n}$
is 
\begin{equation}
\mathcal{G}_{\alpha}^{(n)}(P_{X},P_{Y}):=\min_{P_{X^{n}Y^{n}}\in\Pi(P_{X}^{\otimes n},P_{Y}^{\otimes n})}\mathbb{P}_{(X^{n},Y^{n})\sim P_{X^{n}Y^{n}}}\{c_{n}(X^{n},Y^{n})>\alpha\}\label{eq:nStrassen}
\end{equation}
where $P_{X}^{\otimes n}$ and $P_{Y}^{\otimes n}$ denote the $n$-fold
products of $P_{X}$ and $P_{Y}$ respectively. The minimization problem
in \eqref{eq:nStrassen} is termed the \emph{$n$-dimensional Strassen's
OT problem}. 
It is easily seen that when $n=1$, $\mathcal{G}_{\alpha}^{(1)}(P_{X},P_{Y})$
reduces to the one-dimensional version $\mathcal{G}_{\alpha}(P_{X},P_{Y})$
in \eqref{eq:strassen}. In this paper, we aim at characterizing the
convergence rate of $\mathcal{G}_{\alpha}^{(n)}(P_{X},P_{Y})$ as
the dimension $n\to\infty$ for given $P_{X},P_{Y}$, $c$, and $\alpha$.
To analyze the asymptotic behavior of Strassen's OT problem, we plan
to leverage existing limit theorems in probability theory. However,
obviously an optimization is involved in Strassen's OT problem and
solving this optimization 
is very difficult in general \cite{villani2003topics}. Hence, it
seems unfeasible to apply limit theorems \emph{directly} to Strassen's
OT problem. To overcome this obstacle, we establish a formula, termed
the \emph{nested formula}, which forms a bridge between Strassen's
OT problem and existing limit theorems. 
Specifically, we observe that the minimization problem in \eqref{eq:nStrassen}
can be decoupled into two nested subproblems: an outer subproblem
and an inner subproblem.

Given $n\geq1$, the \emph{empirical measure} (also known as
\emph{type} for the finite alphabet case) for a sequence $x^{n}\in\mathcal{X}^{n}$
is 
\[
T_{x^{n}}:=\frac{1}{n}\sum_{i=1}^{n}\delta_{x_{i}}
\]
where $\delta_{x}$ is Dirac mass at the point $x\in\mathcal{X}$.
The \emph{empirical joint measure}, denoted by $T_{x^{n},y^{n}}$,
for a pair of sequences $(x^{n},y^{n})\in\mathcal{X}^{n}\times\mathcal{Y}^{n}$
is defined similarly. 
Obviously, empirical measures (or empirical joint measures) for $n$-length
sequences are discrete distributions whose probability values are
multiples of $1/n$. Denote $\mu_{n},\nu_{n}$ as the laws\footnote{Note that $\mu_{n},\nu_{n}$ are the \emph{laws} of empirical
measures, rather than empirical measures themselves.} (distributions) of the empirical measures of $X^{n}\sim P_{X}^{\otimes n}$
and $Y^{n}\sim P_{Y}^{\otimes n}$ respectively. Denote $\Pi(\mu_{n},\nu_{n})$
as the set of couplings of $\mu_{n},\nu_{n}$. Here $\mu_{n},\nu_{n}$
and their couplings are respectively defined on Borel measurable spaces
$\mathcal{P}(\mathcal{X}),\mathcal{P}(\mathcal{Y}),\mathcal{P}(\mathcal{X}\times\mathcal{Y})$
induced by the weak topologies. 
\begin{thm}[Nested Formula for Strassen's OT]
\label{thm:nestedOT} Given $P_{X},P_{Y}$, $c$, and $\alpha$,
under Assumption \ref{ass:Nonempty-Condition-We}, we have 
\begin{align}
\mathcal{G}_{\alpha}^{(n)}(P_{X},P_{Y}) & =\min_{\pi\in\Pi(\mu_{n},\nu_{n})}\pi\{(Q_{X},Q_{Y})\in\mathcal{P}(\mathcal{X})\times\mathcal{P}(\mathcal{Y}):\mathcal{E}(Q_{X},Q_{Y})>\alpha\}\label{eq:nestedOT}\\
 & =\sup_{\textrm{closed }A\subseteq\mathcal{P}(\mathcal{X}),B\subseteq\mathcal{P}(\mathcal{Y}):\mathcal{E}(Q_{X},Q_{Y})>\alpha,\forall Q_{X}\in A,Q_{Y}\in B}\mu_{n}(A)+\nu_{n}(B)-1\label{eq:nesteddual}\\
 & =\sup_{\textrm{compact }A\subseteq\mathcal{P}(\mathcal{X})}\mu_{n}(A)-\nu_{n}(\Gamma_{\mathcal{E}\le\alpha}(A)),\label{eq:nesteddual2}
\end{align}
where $(Q_{X},Q_{Y})\mapsto\mathcal{E}(Q_{X},Q_{Y})$ is the OT functional
given in \eqref{eq:OT} and 
\[
\Gamma_{\mathcal{E}\le\alpha}(A):=\bigcup_{Q_{X}\in A}\{Q_{Y}\in\mathcal{P}(\mathcal{Y}):\mathcal{E}(Q_{X},Q_{Y})\leq\alpha\}.
\]
\end{thm}
The inner subproblem in \eqref{eq:nestedOT} (i.e., the optimization
in the definition of $\mathcal{E}(Q_{X},Q_{Y})$) is nothing but (one-dimensional)
Monge--Kantorovich's OT problem defined in \eqref{eq:OT}, while
the outer subproblem corresponds to a new Strassen's OT problem in
which the marginal distributions $\mu_{n},\nu_{n}$ (respectively
defined on Borel measurable spaces $\mathcal{P}(\mathcal{X}),\mathcal{P}(\mathcal{Y})$
induced by the weak topologies) are the laws of the empirical measures
and the cost function is the OT functional $(Q_{X},Q_{Y})\in\mathcal{P}(\mathcal{X})\times\mathcal{P}(\mathcal{Y})\mapsto\mathcal{E}(Q_{X},Q_{Y})$.

Since $\mu_{n},\nu_{n}$ in the nested formula in \eqref{eq:nestedOT}
are the laws of the empirical measures, given the dimension $n$,
they are concentrated on the set of the empirical measures of $n$-length
sequences. This in turn implies that the cost function in the nested
formula, i.e., the OT functional $(Q_{X},Q_{Y})\mapsto\mathcal{E}(Q_{X},Q_{Y})$,
can be restricted to the set of the empirical joint distributions
of pairs of $n$-length sequences. The \emph{set of empirical measures}
of sequences in $\mathcal{X}^{n}$ is denoted as $\mathcal{P}_{n}(\mathcal{X}):=\{T_{x^{n}}:x^{n}\in\mathcal{X}^{n}\}$
and the \emph{set of empirical joint measures} of pairs of sequences
in $\mathcal{X}^{n}\times\mathcal{Y}^{n}$ is denoted as $\mathcal{P}_{n}(\mathcal{X}\times\mathcal{Y}):=\{T_{x^{n},y^{n}}:(x^{n},y^{n})\in\mathcal{X}^{n}\times\mathcal{Y}^{n}\}$.
We denote $\ell_{1}:x^{n}\in\mathcal{X}^{n}\mapsto T_{x^{n}}$
and $\ell_{2}:y^{n}\in\mathcal{Y}^{n}\mapsto T_{y^{n}}$ as the empirical
measure  functions, and  denote $\ell :(x^{n},y^n)\in\mathcal{X}^{n} \times \mathcal{Y}^{n} \mapsto T_{x^{n},y^{n}}$
 as the joint empirical
measure  function. 
By definition, it is easily verified that  $\ell_{1},\ell$ are continuous with respect to   weak topologies, and hence measurable functions with respect to the
Borel $\sigma$-algebras (induced by weak topologies). Throughout
this paper, we use $T_{X},T_{Y},T_{XY}$ to respectively denote elements in $\mathcal{P}_{n}(\mathcal{X}),\mathcal{P}_{n}(\mathcal{X}),\mathcal{P}_{n}(\mathcal{X}\times \mathcal{Y})$, i.e., empirical
measures on $\mathcal{X},\mathcal{Y},\mathcal{X}\times\mathcal{Y}$
of $n$-length sequences. 
Based on the notations   above, the nested formula in \eqref{eq:nestedOT}
can be rewritten as 
\begin{align}
\mathcal{G}_{\alpha}^{(n)}(P_{X},P_{Y}) & =\min_{\pi\in\Pi(\mu_{n},\nu_{n})}\pi\{(T_{X},T_{Y})\in\mathcal{P}_{n}(\mathcal{X})\times\mathcal{P}_{n}(\mathcal{Y}):\mathcal{E}(T_{X},T_{Y})>\alpha\}.\label{eq:nestedOT2}
\end{align}

The intuition behind Theorem \ref{thm:nestedOT} is as follows. 
 Here we assume $c_{n}$ to be continuous. (Any lower semi-continuous
function can be approximated by a nondecreasing sequence of continuous
functions.) On one hand, the cost function $c_{n}(x^{n},y^{n})$ is
permutation-invariant in the sense that it remains the same if the
coordinate pairs of $(x^{n},y^{n})$ are arbitrarily rearranged. In
other words, $c_{n}(x^{n},y^{n})$ depends on $(x^{n},y^{n})$ via
their empirical joint measure $T_{x^{n},y^{n}}$. On the other hand,
any product distribution $P_{X}^{\otimes n}$ can be also rewritten
as a mixture in the following form\footnote{It can be shown that for any measurable $B$ in $\mathcal{X}^{n}$,
$T_{X}\in\mathcal{P}_{n}(\mathcal{X})\mapsto\mathrm{Unif}(\ell_{1}^{-1}(T_{X}))(B)\in [0,1]$
is measurable, which implies that $(T_{X},B)\mapsto\mathrm{Unif}(\ell_{1}^{-1}(T_{X}))(B)$
is a Markov kernel (or transition probability). This Markov kernel is a   regular conditional distribution of  $P_{X}^{\otimes n}$   since $P_{X}^{\otimes n}$ is exchangeable (or permutation-invariant) and if the regular conditional distribution is not this Markov kernel then taking average of the permutation versions of this regular conditional distribution, we will get this Markov kernel (they are not equal up to a $\mu_n$-null set).}: 
\begin{equation}
P_{X}^{\otimes n}=\varint\mathrm{Unif}(\ell_{1}^{-1}(T_{X}))\mathrm{d}\mu_{n}(T_{X}),
\end{equation}
where $\mathrm{Unif}(\ell_{1}^{-1}(T_{X}))$ denotes the discrete uniform  
distribution on $\ell_{1}^{-1}(T_{X})$, the set of sequences
$x^{n}$ having empirical measure $T_{X}$. These two properties imply
that the minimization in \eqref{eq:nStrassen} can be decomposed into
two sub-minimizations: The inner one is over the couplings (empirical
joint measures) of two marginal empirical measures $T_{X^{n}},T_{Y^{n}}$,
and the outer one is over the couplings of the laws $\mu_{n},\nu_{n}$.
The optimal coupling attaining the minimum in \eqref{eq:nStrassen}
can be expressed as 
\begin{equation}
P_{X^{n}Y^{n}}=\varint\mathrm{Unif}(\ell^{-1}(T_{XY}^{*}(T_{X},T_{Y})))\mathrm{d}\pi^{*}(T_{X},T_{Y}),
\end{equation}
where $T_{XY}^{*}(T_{X},T_{Y})$ is an optimal empirical joint measure
attaining $\mathcal{E}(T_{X},T_{Y})$ such that\footnote{The existence of such a map for a continuous cost function follows
from the measurable selection of optimal plans \cite[Corollary 5.22]{villani2008optimal}.
} $(T_{X},T_{Y})\mapsto T_{XY}^{*}(T_{X},T_{Y})$ is measurable under
the $\sigma$-algebra induced by the weak topology, and $\pi^{*}$
is an optimal coupling attaining the minimum in \eqref{eq:nestedOT}
or \eqref{eq:nestedOT2}. In fact, $\mathrm{Unif}(\ell^{-1}(T_{XY}))$
automatically forms a coupling of $\mathrm{Unif}(\ell^{-1}_1(T_{X}))$
and $\mathrm{Unif}(\ell^{-1}_2(T_{Y}))$ if the empirical joint
measure $T_{XY}$ is a coupling of $T_{X},T_{Y}$.


Combining the nested formulas above with the large deviations principle
(LDP) on empirical distributions (specifically, Sanov's theorem \cite[Theorem 6.2.10]{Dembo}),
we show that $\alpha=\mathcal{E}(P_{X},P_{Y})$ is a phase transition
point: For $\alpha<\mathcal{E}(P_{X},P_{Y})$, we have that $\mathcal{G}_{\alpha}^{(n)}(P_{X},P_{Y})$
converges to one exponentially fast as $n\to\infty$, and for $\alpha>\mathcal{E}(P_{X},P_{Y})$,
$\mathcal{G}_{\alpha}^{(n)}(P_{X},P_{Y})$ converges to zero exponentially
fast. The exponents of these convergences, called \emph{large deviations
(LD) exponents}, are characterized by us in terms of variational formulas
(similarly to the large deviations theory for the empirical mean of
i.i.d. random variables \cite{Dembo}). In order to derive our results,
we require an assumption stronger than Assumption \ref{ass:Nonempty-Condition-We}.
\assumption[Interior-Point] \label{ass:Interior} { \emph{We assume
that $c$ is non-constant and $\alpha$ satisfies that $c_{\inf}<\alpha<c_{\sup}$,
where $c_{\inf}:=\inf_{x,y}c(x,y)$ and $c_{\sup}:=\sup_{x,y}c(x,y)$.}}

Besides Assumption \ref{ass:Interior}, we also need another mild
assumption on the uniform continuity of the OT functional. Since $\mathcal{X}$
is a Polish space, $\mathcal{P}(\mathcal{X})$ equipped with the weak
topology is also Polish \cite[Theorem 6.2 and Theorem 6.5]{parthasarathy2005probability}.
Let $\mathtt{L}_{1}$ be the Lévy--Prokhorov metric on $\mathcal{P}(\mathcal{X})$
given by $\mathtt{L}_{1}(Q_{X}',Q_{X})=\inf\{\delta:Q_{X}'(A)\le Q_{X}(A_{\delta})+\delta,\forall\textrm{ closed }A\subseteq\mathcal{X}\}$,
where given a metric $d$ on $\mathcal{X}$, 
\begin{equation}
A_{\delta}:=\bigcup_{x\in A}\{x'\in\mathcal{X}:d(x,x')<\delta\}\label{eq:enlarge}
\end{equation}
denotes the $\delta$-enlargement of $A$ under the metric $d$. Here
$A_{\delta}$ corresponds to a variant of $\Gamma_{c\le\delta}(A)$
defined in \eqref{eq:Gamma}, in which the inequality sign ``$\le$''
is replaced by the strict one ``$<$'' and the cost function $c$
is set to the metric $d$. It is well known that the Lévy--Prokhorov
metric is compatible with the weak topology. Similarly, let $\mathtt{L}_{2}$
be the Lévy--Prokhorov metric on $\mathcal{P}(\mathcal{Y})$. We
additionally assume that the OT functional is uniformly continuous.

\assumption[Uniform Continuity of OT Functional (UCOTF)] \label{ass:Uniform-Continuity-of}
{ \emph{We assume that the optimal transport functional $(Q_{X},Q_{Y})\in\mathcal{P}(\mathcal{X})\times\mathcal{P}(\mathcal{Y})\mapsto\mathcal{E}(Q_{X},Q_{Y})\in[0,+\infty]$
is uniformly continuous, i.e., 
\[
\lim_{\epsilon\downarrow0}\sup_{Q_{X}',Q_{Y}',Q_{X},Q_{Y}:\mathtt{L}_{1}(Q_{X}',Q_{X}),\mathtt{L}_{2}(Q_{Y}',Q_{Y})\le\epsilon}|\mathcal{E}(Q_{X}',Q_{Y}')-\mathcal{E}(Q_{X},Q_{Y})|=0.
\]
}} 
Given $\mathcal{X}$ and $\mathcal{Y}$, the uniform continuity of
$\mathcal{E}$ is only determined by $c$. The UCOTF is not necessarily
satisfied in general, however, it indeed is satisfied for the following
two cases.\footnote{By Lemma \ref{lem:marginalbound} in Appendix \ref{sec:Basic-Lemmas},
it is easy to show that UCOTF is satisfied for the first case. By
\cite[Corollary 6.13]{villani2008optimal}, UCOTF is satisfied for
the second case. A special instance  of Case 2 is $(\mathbb{R}^{k},d)$
with $d(x,y)=\min\{\Vert x-y\Vert_{q},C\}$ for $q\ge1$ and a constant
$C>0$.} 
\begin{enumerate}
\item (Countable Alphabet and Bounded Cost) $\mathcal{X}$ and $\mathcal{Y}$
are countable sets and $c$ is bounded (i.e., $\sup_{x,y}c(x,y)<\infty$). 
\item (Wasserstein Distance Induced by a Bounded Metric) $\mathcal{X}=\mathcal{Y}$
is a Polish space equipped with a bounded metric $d$, i.e., $\sup_{x,y}d(x,y)<\infty$.
The cost function $c=d^{p}$ for $p\ge1$. For this case, $\mathcal{E}=W_{p}^{p}$. 
\end{enumerate}
For two distributions $P,Q$ defined on the same space, we denote\footnote{Throughout this paper, the base of $\log$ is $e$. }
$D(Q\|P):=\int\log(\frac{\mathrm{d}Q}{\mathrm{d}P})\mathrm{d}Q$ as
the Kullback-Leibler (KL) divergence or relative entropy of $Q$ from
$P$. We now state one of our main results in this paper, namely a
\emph{LDP for Strassen's OT problem}. 
\begin{thm}[LDP for Strassen's OT]
\label{thm:LDP} Under Assumptions \ref{ass:Interior} and \ref{ass:Uniform-Continuity-of},
the following hold. 
\begin{enumerate}
\item For $\alpha<\mathcal{E}(P_{X},P_{Y})$, we have 
\begin{equation}
\lim_{n\to\infty}-\frac{1}{n}\log(1-\mathcal{G}_{\alpha}^{(n)}(P_{X},P_{Y}))=f(\alpha),\label{eq:LD1}
\end{equation}
where 
\begin{align}
f(\alpha) & :=\inf_{Q_{X}\in\mathcal{P}(\mathcal{X}),Q_{Y}\in\mathcal{P}(\mathcal{Y}):\mathcal{E}(Q_{X},Q_{Y})\le\alpha}\max\{D(Q_{X}\|P_{X}),D(Q_{Y}\|P_{Y})\}.\label{eq:f}
\end{align}
\item For $\alpha>\mathcal{E}(P_{X},P_{Y})$, we have 
\begin{equation}
\lim_{\alpha'\uparrow\alpha}g(\alpha')\le\liminf_{n\to\infty}-\frac{1}{n}\log\mathcal{G}_{\alpha}^{(n)}(P_{X},P_{Y})\le\limsup_{n\to\infty}-\frac{1}{n}\log\mathcal{G}_{\alpha}^{(n)}(P_{X},P_{Y})\le g(\alpha),\label{eq:LD2}
\end{equation}
where $g(\alpha):=\min\{g_{P_{X},P_{Y}}(\alpha),g_{P_{Y},P_{X}}(\alpha)\}$
with $g_{P_{X},P_{Y}}(\alpha)$ defined as the infimum of $D(Q_{X}\|P_{X})$
over all $Q_{X}\in\mathcal{P}(\mathcal{X})$ such that 
\begin{equation}
\inf_{Q_{Y}\in\mathcal{P}(\mathcal{Y}):D(Q_{Y}\|P_{Y})\leq D(Q_{X}\|P_{X})}\mathcal{E}(Q_{X},Q_{Y})>\alpha,\label{eq:LD_condition}
\end{equation}
and $g_{P_{Y},P_{X}}(\alpha)$ defined similarly. 
\end{enumerate}
\end{thm}

It is easily verified that 
 the function $g$ in Theorem \ref{thm:LDP}  is right-continuous. 
Furthermore, it is well known that the set of discontinuous points
for a right-continuous function has Lebesgue measure zero. Hence $g$
is continuous almost everywhere, which means that 
\[
\lim_{n\to\infty}-\frac{1}{n}\log\mathcal{G}_{\alpha}^{(n)}(P_{X},P_{Y})=g(\alpha)\textrm{ for almost every }\alpha\in(\mathcal{E}(P_{X},P_{Y}),c_{\sup}).
\]

Theorem \ref{thm:LDP} bears a semblance to the classic LDP for the
empirical mean of i.i.d. random variables; for the latter, see \cite{Dembo}.
This is the reason why we call Theorem \ref{thm:LDP} as a theorem
on the \emph{LDP for Strassen's OT}. Nevertheless, the exponents in
our setting are different from, and more complicated than, those for
the empirical mean of i.i.d. random variables, which is due to the
additional minimization in \eqref{eq:nStrassen} or the one in \eqref{eq:nestedOT}.
Our proof is based on the nested formula in Theorem \ref{thm:nestedOT},
Strassen's dual formulation, and Sanov's theorem. Besides these, some
other specific techniques are also needed in our proof, for example,
the splitting technique \cite{nummelin1978uniform,athreya1978new}.
Furthermore, the relative entropy $D(\cdot\|P)$ is the rate function
for the LDP on the empirical measure, as stated in Sanov's theorem
\cite[Theorem 6.2.10]{Dembo}, which leads to the fact that relative
entropies are involved in the expressions in Theorem \ref{thm:LDP}.

Theorem \ref{thm:LDP} generalizes a result in \cite{yu2018asymptotic}.
In \cite{yu2018asymptotic}, the present author together with Tan
only considered the finite alphabet case. For this case, by using
the method of types, they derived the same expression for the LD exponent
for the case of $\alpha<\mathcal{E}(P_{X},P_{Y})$, but they only
provided a bound for the case of $\alpha>\mathcal{E}(P_{X},P_{Y})$.

\subsection{Intuition of Main Result 1}

In the following, we reveal some insights into the expressions in
Theorem \ref{thm:LDP}, from the perspective of the \emph{primal}
problem for the case of $\alpha<\mathcal{E}(P_{X},P_{Y})$ and from
the perspective of the \emph{dual} problem for the case of $\alpha>\mathcal{E}(P_{X},P_{Y})$.
For brevity, we focus on the case of finite alphabets. 

We first explain Theorem \ref{thm:LDP} for the case of $\alpha<\mathcal{E}(P_{X},P_{Y})$.
For a finite alphabet $\mathcal{X}$, the number of possible empirical
measures of sequences in $\mathcal{X}^{n}$ is polynomial in $n$
(more precisely, which is no larger than $(n+1)^{|\mathcal{X}|}$)
\cite{Csiszar}. This implies that every set $A\subseteq\mathcal{P}(\mathcal{X})$
which contains at least one empirical measure, has a dominant empirical
measure $T_{X}\in A$ for sufficiently large $n$ in the sense that
\begin{equation}
(n+1)^{-|\mathcal{X}|}\mu_{n}(A)\leq\mu_{n}(T_{X})\leq\mu_{n}(A).\label{eq:-62}
\end{equation}
Furthermore, by Sanov's theorem \cite[Theorem 6.2.10]{Dembo}, the
law $\mu_{n}$ of the empirical measure $T_{X^{n}}$ of the i.i.d.
sequence $X^{n}\sim P_{X}^{\otimes n}$ (or $Y^{n}\sim P_{Y}^{\otimes n}$)
satisfies a LDP with the relative entropy $D(\cdot\|P_{X})$ as the
rate function. Hence for a fixed set $A$ not containing $P_{X}$,
the polynomial term $(n+1)^{-|\mathcal{X}|}$ in the left-hand side
of \eqref{eq:-62} is dominated by the term $\mu_{n}(A)$, since $\mu_{n}(A)$
vanishes exponentially fast.

Observe that 
\[
1-\mathcal{G}_{\alpha}^{(n)}(P_{X},P_{Y})=\sup_{\pi\in\Pi(\mu_{n},\nu_{n})}\pi\{(T_{X},T_{Y}):\mathcal{E}(T_{X},T_{Y})\leq\alpha\}
\]
and for any 
$\pi\in\Pi(\mu_{n},\nu_{n})$, 
\begin{align}
\pi\{(T_{X},T_{Y}):\mathcal{E}(T_{X},T_{Y})\leq\alpha\} & =\sum_{T_{X},T_{Y}:\mathcal{E}(T_{X},T_{Y})\leq\alpha}\pi\{(T_{X},T_{Y})\}\label{eq:-66}\\
 & \leq\sum_{T_{X},T_{Y}:\mathcal{E}(T_{X},T_{Y})\leq\alpha}\min\{\mu_{n}(T_{X}),\nu_{n}(T_{Y})\}\label{eq:-63}\\
 & =e^{no(1)}\max_{T_{X},T_{Y}:\mathcal{E}(T_{X},T_{Y})\leq\alpha}\min\{\mu_{n}(T_{X}),\nu_{n}(T_{Y})\},\label{eq:-65}
\end{align}
where \eqref{eq:-63} follows since $\mu_{n}(T_{X})=\sum_{T_{Y}}\pi\{(T_{X},T_{Y})\}$
and $\nu_{n}(T_{Y})=\sum_{T_{X}}\pi\{(T_{X},T_{Y})\}$. Finally, expressing
the exponents of $\mu_{n}(T_{X})$ and $\nu_{n}(T_{Y})$ by relative
entropies $D(\cdot\|P_{X})$ and $D(\cdot\|P_{Y})$, we obtain $f(\alpha)$.
Hence the exponent in \eqref{eq:LD1} is lower bounded by $f(\alpha)$.
Moreover, the exponent of the upper bound \eqref{eq:-65} is attained
by some coupling $\pi\in\Pi(\mu_{n},\nu_{n})$. Let $(T_{X}^{*},T_{Y}^{*})$
be an optimal pair that attains the maximum in \eqref{eq:-65}. We
construct $\pi\in\Pi(\mu_{n},\nu_{n})$ such that 
\begin{equation}
\pi\{(T_{X}^{*},T_{Y}^{*})\}=\min\{\mu_{n}(T_{X}^{*}),\nu_{n}(T_{Y}^{*})\},\label{eq:}
\end{equation}
which ensures that the exponent of the upper bound in \eqref{eq:-65}
is asymptotically attained by such a coupling $\pi$. Hence, the exponent
in \eqref{eq:LD1} is also upper bounded by $f(\alpha)$. See the
illustration for this case in Fig. \ref{fig:}.

For the case of $\alpha>\mathcal{E}(P_{X},P_{Y})$, the intuition
behind the expression in \eqref{eq:LD2} is less obvious, because
it is difficult to construct an explicit coupling to asymptotically
attain $\mathcal{G}_{\alpha}^{(n)}(P_{X},P_{Y})$. However, since
$\mathcal{G}_{\alpha}^{(n)}(P_{X},P_{Y})$ can be rewritten in the
form of Strassen's duality (given in \eqref{eq:nesteddual2}), it
suffices to construct an explicit (asymptotically) optimal solution
to Strassen's dual problem. 
As mentioned in the above case, the exponent of a set $A\subseteq\mathcal{P}_{n}(\mathcal{X})$
is asymptotically dominated by only one empirical measure $T_{X}$
in it. Hence, it suffices to consider a singleton $A=\{T_{X}\}$ for
the optimization problem in \eqref{eq:nesteddual2}. For such a singleton,
the exponent of $\mu_{n}(A)$ is $D(T_{X}\|P_{X})$, and the exponent
of $\nu_{n}(\Gamma_{\mathcal{E}\le\alpha}(A))=\nu_{n}\{T_{Y}:\mathcal{E}(T_{X},T_{Y})\leq\alpha\}$
is $\min_{T_{Y}:\mathcal{E}(T_{X},T_{Y})\le\alpha}D(T_{Y}\|P_{Y})$.
To maximize $\mu_{n}(A)-\nu_{n}(\Gamma_{\mathcal{E}\le\alpha}(A))$,
it suffices to consider $A$ such that $\mu_{n}(A)>\nu_{n}(\Gamma_{\mathcal{E}\le\alpha}(A))$,
which, roughly speaking, is equivalent to consider $T_{X}$ such that
\begin{equation}
\min_{T_{Y}:\mathcal{E}(T_{X},T_{Y})\le\alpha}D(T_{Y}\|P_{Y})>D(T_{X}\|P_{X}).\label{eq:-69}
\end{equation}
On the other hand, $\mu_{n}(A)$ will be exponentially larger than
$\nu_{n}(\Gamma_{\mathcal{E}\le\alpha}(A))$, if the left-hand side
in \eqref{eq:-69} is upper bounded away from the right-hand side
in \eqref{eq:-69}. Hence, roughly speaking, in this case the exponent
of $\mathcal{G}_{\alpha}^{(n)}(P_{X},P_{Y})$ is the minimum of $D(T_{X}\|P_{X})$
over all $T_{X}$ satisfying \eqref{eq:-69}. Observe that the condition
in \eqref{eq:-69} is equivalent to the condition in \eqref{eq:LD_condition},
which implies that the exponent of $\mathcal{G}_{\alpha}^{(n)}(P_{X},P_{Y})$
is sandwiched between $\lim_{\alpha'\uparrow\alpha}g_{P_{X},P_{Y}}(\alpha')$
and $g_{P_{X},P_{Y}}(\alpha)$. However, in some cases, $\mu_{n}(A)-\nu_{n}(\Gamma_{\mathcal{E}\le\alpha}(A))$
is maximized by a set $A$ such that both $\mu_{n}(A)$ and $\nu_{n}(\Gamma_{\mathcal{E}\le\alpha}(A))$
approach one. For this case, the quantities $\lim_{\alpha'\uparrow\alpha}g_{P_{X},P_{Y}}(\alpha')$
and $g_{P_{X},P_{Y}}(\alpha)$ do not correspond to the exponent of
the difference $\mu_{n}(A)-\nu_{n}(\Gamma_{\mathcal{E}\le\alpha}(A))$
any more. In fact, the exponent for this case is sandwiched between
the counterparts $\lim_{\alpha'\uparrow\alpha}g_{P_{Y},P_{X}}(\alpha')$
and $g_{P_{Y},P_{X}}(\alpha)$. Hence, in a word, the exponent of
$\mathcal{G}_{\alpha}^{(n)}(P_{X},P_{Y})$ is indeed sandwiched between
$\lim_{\alpha'\uparrow\alpha}g(\alpha')$ and $g(\alpha)$. See the
illustration for this case in Fig. \ref{fig:-1}.

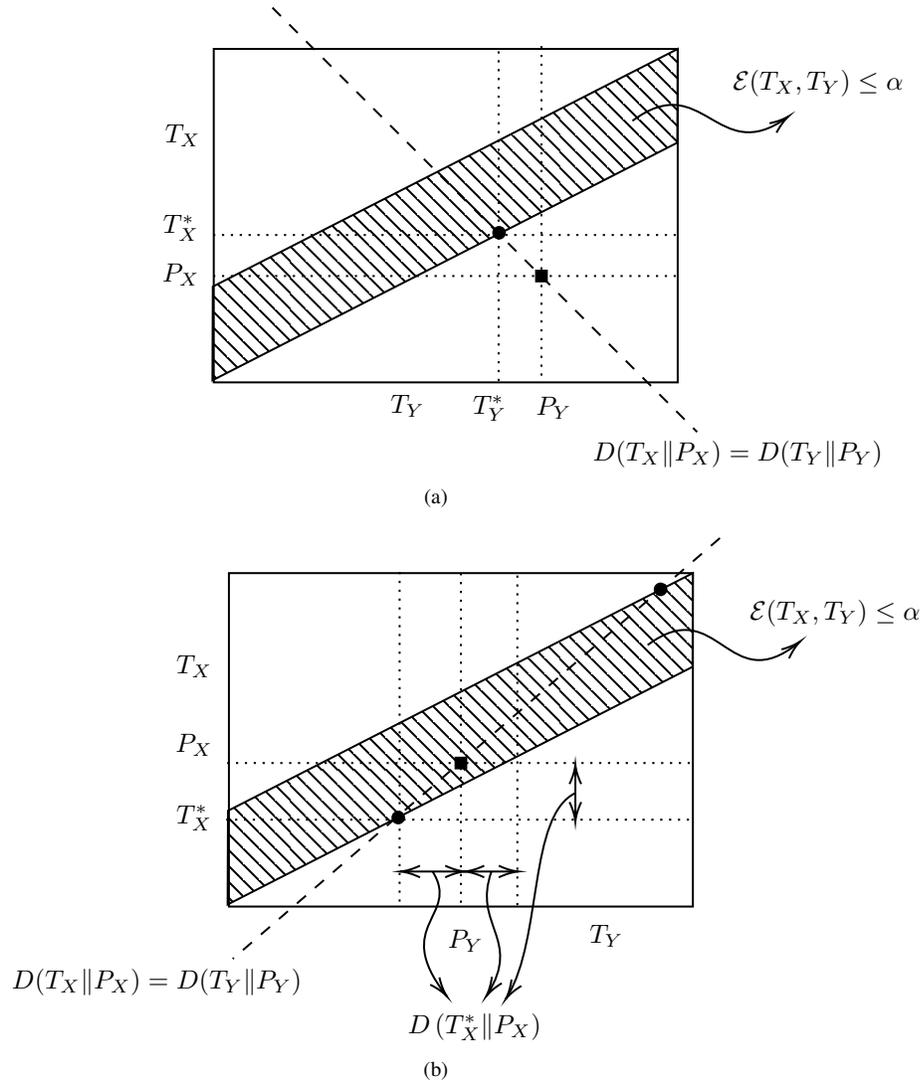
\begin{figure}
\centering

\tikzset{ pattern size/.store in=\mcSize, pattern size = 5pt, pattern
thickness/.store in=\mcThickness, pattern thickness = 0.3pt, pattern
radius/.store in=\mcRadius, pattern radius = 1pt} \makeatletter
\pgfutil@ifundefined{pgf@pattern@name@u1xc0z9lu}{ \pgfdeclarepatternformonly[\mcThickness,\mcSize]{u1xc0z9lu}
{\pgfqpoint{0pt}{-\mcThickness}} {\pgfpoint{\mcSize}{\mcSize}}
{\pgfpoint{\mcSize}{\mcSize}} { \pgfsetcolor{\tikz@pattern@color}
\pgfsetlinewidth{\mcThickness} \pgfpathmoveto{\pgfqpoint{0pt}{\mcSize}}
\pgfpathlineto{\pgfpoint{\mcSize+\mcThickness}{-\mcThickness}}
\pgfusepath{stroke} }} \makeatother \tikzset{every picture/.style={line
width=0.75pt}} 

\subfloat[\label{fig:}]{\begin{tikzpicture}[x=0.75pt,y=0.75pt,yscale=-1,xscale=1,scale=0.9] 
\path (0,100); 
\draw   (147.36,21.65) -- (407.36,21.65) -- (407.36,208.65) -- (147.36,208.65) -- cycle ; 
\draw  [dash pattern={on 0.84pt off 2.51pt}]  (331,20) -- (331,207) ;
\draw  [dash pattern={on 0.84pt off 2.51pt}]  (147,149) -- (409,149) ;
\draw    (382,62) .. controls (421.6,32.3) and (429.35,88.28) .. (468.32,60.29) ; 
\draw [shift={(469.51,59.42)}, rotate = 503.13] [color={rgb, 255:red, 0; green, 0; blue, 0 }  ][line width=0.75]    (10.93,-3.29) .. controls (6.95,-1.4) and (3.31,-0.3) .. (0,0) .. controls (3.31,0.3) and (6.95,1.4) .. (10.93,3.29)   ;
\draw  [pattern=u1xc0z9lu,pattern size=6pt,pattern thickness=0.75pt,pattern radius=0pt, pattern color={rgb, 255:red, 0; green, 0; blue, 0}] (146.84,154.93) -- (146.64,207.5) -- (407.16,74.22) -- (407.36,21.65) -- cycle ; 
\draw  [dash pattern={on 0.84pt off 2.51pt}]  (147,126) -- (409,126) ;
\draw  [dash pattern={on 0.84pt off 2.51pt}]  (307,23) -- (307,210) ;
\draw  [fill={rgb, 255:red, 0; green, 0; blue, 0 }  ,fill opacity=1 ] (304,124.75) .. controls (304,122.96) and (305.46,121.5) .. (307.25,121.5) .. controls (309.04,121.5) and (310.5,122.96) .. (310.5,124.75) .. controls (310.5,126.54) and (309.04,128) .. (307.25,128) .. controls (305.46,128) and (304,126.54) .. (304,124.75) -- cycle ; 
\draw  [dash pattern={on 4.5pt off 4.5pt}]  (180.25,-1.5) -- (418.25,236.5) ;
\draw  [fill={rgb, 255:red, 0; green, 0; blue, 0 }  ,fill opacity=1 ] (328,146) -- (334,146) -- (334,152) -- (328,152) -- cycle ;
\draw (130,70) node   {$T_{X}$}; 
\draw (256,222) node   {$T_{Y}$}; 
\draw (338,223) node   {$P_{Y}$}; 
\draw (129,148) node   {$P_{X}$}; 
\draw (486,41) node   {$\mathcal{E}( T_{X} ,T_{Y}) \leq \alpha $}; 
\draw (302,223) node   {$T^{*}_{Y}$}; 
\draw (129,122) node   {$T^{*}_{X}$}; 
\draw (441,249) node   {$D( T_{X} \| P_{X}) =D( T_{Y} \| P_{Y})$};
\end{tikzpicture}

}\centering

\tikzset{ pattern size/.store in=\mcSize, pattern size = 5pt, pattern
thickness/.store in=\mcThickness, pattern thickness = 0.3pt, pattern
radius/.store in=\mcRadius, pattern radius = 1pt} \makeatletter
\pgfutil@ifundefined{pgf@pattern@name@rdqhqxxav}{ \pgfdeclarepatternformonly[\mcThickness,\mcSize]{rdqhqxxav}
{\pgfqpoint{0pt}{-\mcThickness}} {\pgfpoint{\mcSize}{\mcSize}}
{\pgfpoint{\mcSize}{\mcSize}} { \pgfsetcolor{\tikz@pattern@color}
\pgfsetlinewidth{\mcThickness} \pgfpathmoveto{\pgfqpoint{0pt}{\mcSize}}
\pgfpathlineto{\pgfpoint{\mcSize+\mcThickness}{-\mcThickness}}
\pgfusepath{stroke} }} \makeatother \tikzset{every picture/.style={line
width=0.75pt}} 

\subfloat[\label{fig:-1}]{\begin{tikzpicture}[x=0.75pt,y=0.75pt,yscale=-1,xscale=1,scale=0.9] 
\path (0,225.0170440673828); 
\draw   (165.36,109.65) -- (425.36,109.65) -- (425.36,296.65) -- (165.36,296.65) -- cycle ; 
\draw  [dash pattern={on 0.84pt off 2.51pt}]  (327,111) -- (327,298) ;
\draw  [dash pattern={on 0.84pt off 2.51pt}]  (163.84,247.93) -- (425.84,247.93) ;
\draw    (400,150) .. controls (439.6,120.3) and (444.91,180.64) .. (484.3,148.92) ; 
\draw [shift={(485.51,147.93)}, rotate = 500.04] [color={rgb, 255:red, 0; green, 0; blue, 0 }  ][line width=0.75]    (10.93,-3.29) .. controls (6.95,-1.4) and (3.31,-0.3) .. (0,0) .. controls (3.31,0.3) and (6.95,1.4) .. (10.93,3.29)   ;
\draw  [pattern=rdqhqxxav,pattern size=6pt,pattern thickness=0.75pt,pattern radius=0pt, pattern color={rgb, 255:red, 0; green, 0; blue, 0}] (164.84,242.93) -- (164.64,295.5) -- (425.16,162.22) -- (425.36,109.65) -- cycle ; 
\draw  [dash pattern={on 0.84pt off 2.51pt}]  (164.36,216.15) -- (426.36,216.15) ;
\draw  [dash pattern={on 0.84pt off 2.51pt}]  (295.36,109.65) -- (295.36,296.65) ;
\draw  [fill={rgb, 255:red, 0; green, 0; blue, 0 }  ,fill opacity=1 ] (257,246.7) .. controls (257.03,244.91) and (258.5,243.47) .. (260.3,243.5) .. controls (262.09,243.53) and (263.53,245) .. (263.5,246.8) .. controls (263.47,248.59) and (262,250.03) .. (260.2,250) .. controls (258.41,249.97) and (256.97,248.5) .. (257,246.7) -- cycle ; 
\draw  [dash pattern={on 0.84pt off 2.51pt}]  (261,109) -- (261,296) ;
\draw    (263.5,277) -- (295.5,277) ; 
\draw [shift={(297.5,277)}, rotate = 180] [color={rgb, 255:red, 0; green, 0; blue, 0 }  ][line width=0.75]    (10.93,-3.29) .. controls (6.95,-1.4) and (3.31,-0.3) .. (0,0) .. controls (3.31,0.3) and (6.95,1.4) .. (10.93,3.29)   ; 
\draw [shift={(261.5,277)}, rotate = 0] [color={rgb, 255:red, 0; green, 0; blue, 0 }  ][line width=0.75]    (10.93,-3.29) .. controls (6.95,-1.4) and (3.31,-0.3) .. (0,0) .. controls (3.31,0.3) and (6.95,1.4) .. (10.93,3.29)   ; 
\draw    (299.5,277) -- (325.5,277) ; 
\draw [shift={(327.5,277)}, rotate = 180] [color={rgb, 255:red, 0; green, 0; blue, 0 }  ][line width=0.75]    (10.93,-3.29) .. controls (6.95,-1.4) and (3.31,-0.3) .. (0,0) .. controls (3.31,0.3) and (6.95,1.4) .. (10.93,3.29)   ; 
\draw [shift={(297.5,277)}, rotate = 0] [color={rgb, 255:red, 0; green, 0; blue, 0 }  ][line width=0.75]    (10.93,-3.29) .. controls (6.95,-1.4) and (3.31,-0.3) .. (0,0) .. controls (3.31,0.3) and (6.95,1.4) .. (10.93,3.29)   ; 
\draw    (279.5,277) .. controls (295.34,302.28) and (256.29,306.39) .. (286.56,345.79) ; 
\draw [shift={(287.5,347)}, rotate = 231.71] [color={rgb, 255:red, 0; green, 0; blue, 0 }  ][line width=0.75]    (10.93,-3.29) .. controls (6.95,-1.4) and (3.31,-0.3) .. (0,0) .. controls (3.31,0.3) and (6.95,1.4) .. (10.93,3.29)   ;
\draw    (312.5,277) .. controls (301.66,307.08) and (329.64,314.77) .. (310.41,345.58) ; \draw [shift={(309.5,347)}, rotate = 303.27] [color={rgb, 255:red, 0; green, 0; blue, 0 }  ][line width=0.75]    (10.93,-3.29) .. controls (6.95,-1.4) and (3.31,-0.3) .. (0,0) .. controls (3.31,0.3) and (6.95,1.4) .. (10.93,3.29)   ;
\draw    (359.5,219) -- (359.5,247) ; \draw [shift={(359.5,249)}, rotate = 270] [color={rgb, 255:red, 0; green, 0; blue, 0 }  ][line width=0.75]    (10.93,-3.29) .. controls (6.95,-1.4) and (3.31,-0.3) .. (0,0) .. controls (3.31,0.3) and (6.95,1.4) .. (10.93,3.29)   ; \draw [shift={(359.5,217)}, rotate = 90] [color={rgb, 255:red, 0; green, 0; blue, 0 }  ][line width=0.75]    (10.93,-3.29) .. controls (6.95,-1.4) and (3.31,-0.3) .. (0,0) .. controls (3.31,0.3) and (6.95,1.4) .. (10.93,3.29)   ; 
\draw    (359.5,233) .. controls (332.9,245.8) and (341.24,314.4) .. (321.42,347.03) ; \draw [shift={(320.5,348.5)}, rotate = 303.27] [color={rgb, 255:red, 0; green, 0; blue, 0 }  ][line width=0.75]    (10.93,-3.29) .. controls (6.95,-1.4) and (3.31,-0.3) .. (0,0) .. controls (3.31,0.3) and (6.95,1.4) .. (10.93,3.29)   ;
\draw  [dash pattern={on 4.5pt off 4.5pt}]  (440.5,90) -- (167.5,326) ;
\draw  [fill={rgb, 255:red, 0; green, 0; blue, 0 }  ,fill opacity=1 ] (404,118.7) .. controls (404.03,116.91) and (405.5,115.47) .. (407.3,115.5) .. controls (409.09,115.53) and (410.53,117) .. (410.5,118.8) .. controls (410.47,120.59) and (409,122.03) .. (407.2,122) .. controls (405.41,121.97) and (403.97,120.5) .. (404,118.7) -- cycle ; 
\draw  [fill={rgb, 255:red, 0; green, 0; blue, 0 }  ,fill opacity=1 ] (292.36,213.15) -- (298.36,213.15) -- (298.36,219.15) -- (292.36,219.15) -- cycle ;
\draw (146,162) node   {$T_{X}$}; 
\draw (377,313) node   {$T_{Y}$}; 
\draw (298,316) node   {$P_{Y}$}; 
\draw (146,206) node   {$P_{X}$}; 
\draw (505,132) node   {$\mathcal{E}( T_{X} ,T_{Y}) \leq \alpha $}; 
\draw (146,247) node   {$T^{*}_{X}$}; 
\draw (303,364) node   {$D( T^{*}_{X} \| P_{X})$}; 
\draw (125,339) node   {$D( T_{X} \| P_{X}) =D( T_{Y} \| P_{Y})$};
\end{tikzpicture}

}\caption{Illustrations of the LDP for Strassen's OT problem given in Theorem
\ref{thm:LDP}.}
\end{figure}

To further illustrate our results, the binary example is given in
Section \ref{sec:int_exam}.

\subsection{Main Result 2: Moderate Deviations Principle}

In addition to the large deviations regime, we also consider the moderate
deviations regime and central limit regime in the Strassen's OT problem,
in both of which the parameter $\alpha$ is allowed to vary with $n$
as $n$ going to infinity. For simplicity, in these two regimes, we
only consider the case in which $\mathcal{X}$ and $\mathcal{Y}$
are finite sets. Without loss of generality, we assume $\mathcal{X}=\{1,2,...,M\}$
and $\mathcal{Y}=\{1,2,...,N\}$ for some positive integers $M,N$,
and also assume that $\mathcal{X}$ and $\mathcal{Y}$ are respectively
the supports of $P_{X}$ and $P_{Y}$. We now introduce a \emph{moderate
deviations principle (MDP) for Strassen's OT problem}, in which we
set $\alpha$ to $\alpha_{n}=\mathcal{E}(P_{X},P_{Y})+{\Delta}/{\sqrt{na_{n}}}$
for a positive sequence $\{a_{n}\}$ satisfying $a_{n}\to0$ and $na_{n}\to\infty$
as $n\to\infty$. We characterize the limit of $-a_{n}\log(1-\mathcal{G}_{\alpha_{n}}^{(n)}(P_{X},P_{Y}))$
for $\Delta<0$, and the limit of $-a_{n}\log\mathcal{G}_{\alpha_{n}}^{(n)}(P_{X},P_{Y}))$
for $\Delta>0$. These two limits are called \emph{moderate deviations
(MD) exponents}.

Assume $c$ is finite on the finite set $\mathcal{X}\times\mathcal{Y}$,
i.e., $\max_{x,y}c(x,y)<\infty$. For this case, $\mathcal{P}(\mathcal{X})$
and $\mathcal{P}(\mathcal{Y})$ are probability simplices. 
Let $P_{XY}^{*}$ be an optimal coupling that attain $\mathcal{E}(P_{X},P_{Y})$.
Denote $\mathcal{S}$ as the support of $P_{XY}^{*}$. Define the
hyperplane 
\begin{equation}
\mathbb{S}_{X}:=\Bigl\{\beta_{X}\in\mathbb{R}^{|\mathcal{X}|}:\sum_{x}\beta_{X}(x)=0\Bigr\},\label{eq:-2}
\end{equation}
which corresponds to the set of signed measures\footnote{Here we do not distinguish the signed measure $\beta_{X}$ and the
function $x\mapsto\beta_{X}(\{x\})$, since $\beta_{X}$ is uniquely
determined by the restriction $x\mapsto\beta_{X}(\{x\})$. We also
denote $\beta_{X}(\{x\})$ as $\beta_{X}(x)$ for brevity. } $\beta_{X}$ on $\mathcal{X}$ with total measure zero, i.e., $\beta_{X}(\mathcal{X})=0$.
For $\mathcal{Y}$, define $\mathbb{S}_{Y}$ similarly. For signed
measures $\beta_{X}\in\mathbb{S}_{X},\beta_{Y}\in\mathbb{S}_{Y}$,
we define a functional 
\begin{align}
\theta(\beta_{X},\beta_{Y}) & :=\min_{\begin{subarray}{c}
\beta_{XY}\in\overline{\Pi}(\beta_{X},\beta_{Y}):\\
\{(x,y):\beta_{XY}(x,y)<0\}\subseteq\mathcal{S}
\end{subarray}}\sum_{x,y}\beta_{XY}(x,y)c(x,y),\label{eqn:theta}
\end{align}
where $\overline{\Pi}(\beta_{X},\beta_{Y})$ is the set of all signed
(joint) measures on $\mathcal{X}\times\mathcal{Y}$ such that its
$X$- and $Y$-marginals equal to $\beta_{X}$ and $\beta_{Y}$ respectively.
By the strong duality in linear programming, 
\begin{align}
\theta(\beta_{X},\beta_{Y})=\max_{(\phi,\psi)\in\mathcal{D}}\sum_{x}\phi(x)\beta_{X}(x)+\sum_{y}\psi(y)\beta_{Y}(y),\label{eq:theta_dual}
\end{align}
where 
\begin{align}
\mathcal{D}:= & \{(\phi,\psi)\in\mathbb{R}^{|\mathcal{X}|}\times\mathbb{R}^{|\mathcal{Y}|}:\phi(x)+\psi(y)=c(x,y),\forall(x,y)\in\mathcal{S},\\
 & \qquad\qquad\phi(x)+\psi(y)\leq c(x,y),\forall(x,y)\in\mathcal{S}^{c}\}
\end{align}
is the set of optimal solutions to \eqref{eq:-111}. In fact, $\mathcal{D}$
is independent of the choice of $\mathcal{S}$, as long as $\mathcal{S}$
is the support of an optimal solution to the primal problem \eqref{eq:OT}.
This observation follows from the fact that if the strong duality
holds (as in our case), then any pair of primal optimal solution and
dual optimal solution forms a saddle point of the Lagrangian of the
primal problem \eqref{eq:OT}. Conversely, any saddle point of the
Lagrangian must consist of an primal optimal solution and a dual optimal
solution. In other words, the set of saddle points is the Cartesian
product of the set of primal optimal solutions and the set of dual
optimal solutions. See details on the page 239 of \cite{boyd2004convex}.
The formula in \eqref{eq:theta_dual} coincides with the directional
derivative given on the page 2771 of \cite{tameling2019empirical}.
\begin{thm}[MDP for Strassen's OT]
\label{thm:MDP} Assume $\mathcal{X}$ and $\mathcal{Y}$ are finite,
and $c$ is finite. Assume $\alpha_{0}:=\mathcal{E}(P_{X},P_{Y})\in(c_{\inf},c_{\sup})$,
where $c_{\inf},c_{\sup}$ are defined in Assumption \ref{ass:Interior}.
Let $\{a_{n}\}$ be a positive sequence such that $a_{n}\to0$ and
$na_{n}\to\infty$ as $n\to\infty$. Then, the following hold. 
\begin{enumerate}
\item If $\Delta<0$, we have 
\[
\lim_{n\to\infty}-a_{n}\log(1-\mathcal{G}_{\alpha_{0}+{\Delta}/{\sqrt{na_{n}}}}^{(n)}(P_{X},P_{Y}))=\tilde{f}(\Delta),
\]
where 
\begin{align*}
\tilde{f}(\Delta) & :=\min_{\beta_{X}\in\mathbb{S}_{X},\beta_{Y}\in\mathbb{S}_{Y}:\theta(\beta_{X},\beta_{Y})\le\Delta}\max\Bigl\{\frac{1}{2}\sum_{x}\frac{\beta_{X}(x)^{2}}{P_{X}(x)},\frac{1}{2}\sum_{y}\frac{\beta_{Y}(y)^{2}}{P_{Y}(y)}\Bigr\}.
\end{align*}
\item If $\Delta>0$, we have 
\[
\lim_{\Delta'\uparrow\Delta}\tilde{g}(\Delta')\le\liminf_{n\to\infty}-a_{n}\log\mathcal{G}_{\alpha_{0}+{\Delta}/{\sqrt{na_{n}}}}^{(n)}(P_{X},P_{Y})\le\limsup_{n\to\infty}-a_{n}\log\mathcal{G}_{\alpha_{0}+{\Delta}/{\sqrt{na_{n}}}}^{(n)}(P_{X},P_{Y})\le\tilde{g}(\Delta),
\]
where $\tilde{g}(\Delta):=\min\{\tilde{g}_{P_{X},P_{Y}}(\Delta),\tilde{g}_{P_{Y},P_{X}}(\Delta)\}$
with $\tilde{g}_{P_{X},P_{Y}}(\Delta)$ defined as the minimum of
$\frac{1}{2}\sum_{x}\frac{\beta_{X}(x)^{2}}{P_{X}(x)}$ over all $\beta_{X}\in\mathbb{S}_{X}$
such that 
\[
\min_{\beta_{Y}\in\mathbb{S}_{Y}:\sum_{x}\frac{\beta_{X}(x)^{2}}{P_{X}(x)}\leq\sum_{y}\frac{\beta_{Y}(y)^{2}}{P_{Y}(y)}}\theta(\beta_{X},\beta_{Y})>\Delta,
\]
and $\tilde{g}_{P_{Y},P_{X}}(\Delta)$ defined similarly. 
\end{enumerate}
\end{thm}
Our proof relies on the MDP for the empirical measure, in which the
rate function is $\beta\mapsto\frac{1}{2}\sum_{x}\frac{\beta(x)^{2}}{P(x)}$
(for the finite alphabet case). The characterizations of the MD exponents
for Strassen's OT problem are similar to the ones of the LD exponents,
except that the relative entropies $D(Q_{X}\|P_{X})$ and $D(Q_{Y}\|P_{Y})$
are respectively replaced by $\frac{1}{2}\sum_{x}\frac{\beta_{X}(x)^{2}}{P_{X}(x)}$
and $\frac{1}{2}\sum_{y}\frac{\beta_{Y}(y)^{2}}{P_{Y}(y)}$, and the
OT functional $\mathcal{E}(Q_{X},Q_{Y})$ is replaced by the functional
$\theta(\beta_{X},\beta_{Y})$.

\subsection{Main Result 3: Central Limit Theorem }

\label{subsec:Main-Result-3} For the central limit regime, we set
$\alpha$ to $\alpha_{0}+{\Delta}/{\sqrt{n}}$ with $\alpha_{0}=\mathcal{E}(P_{X},P_{Y})$,
and study the asymptotic behavior of $\mathcal{G}_{\alpha_{0}+{\Delta}/{\sqrt{n}}}^{(n)}(P_{X},P_{Y})$.
Similarly to the moderate deviations regime, $\mathcal{X}$ and $\mathcal{Y}$
are assumed to be finite (i.e., $\mathcal{X}=\{1,2,...,M\}$ and $\mathcal{Y}=\{1,2,...,N\}$)
and also assumed to be respectively the supports of $P_{X}$ and $P_{Y}$.
Assume $c$ is finite. Denote random variables $U_{x}=\bone_{X=x},x\in\mathcal{X}$
with $X\sim P_{X}$. Denote $\mathbf{U}=(U_{x},x\in\mathcal{X})$
as a random vector. The mean and covariance of $\mathbf{U}$ are respectively
\[
\mathbb{E}[\mathbf{U}]=(P_{X}(x))_{x\in\mathcal{X}}\quad\textrm{ and }\quad\mathrm{Cov}(\mathbf{U})=[P_{X}(x)\bone_{x=x'}-P_{X}(x)P_{X}(x')]_{(x,x')\in\mathcal{X}^{2}}.
\]
Define $\Phi_{P_{X}}$ as the Gaussian measure\footnote{In fact, both $\Phi_{P_{X}}$ and $\Phi_{P_{Y}}$ are degenerate,
since the covariance matrices are not invertible. However, this does
not affect our results.} on $\mathbb{R}^{|\mathcal{X}|}$ with zero mean and covariance $\mathrm{Cov}(\mathbf{U})$.
For $P_{Y}$, define $\Phi_{P_{Y}}$ similarly. Define a new Strassen's
OT problem as follows: 
\begin{align}
\Lambda_{\Delta}(P_{X},P_{Y}) & :=\min_{\Psi\in\Pi(\Phi_{P_{X}},\Phi_{P_{Y}})}\Psi\{(\beta_{X},\beta_{Y}):\theta(\beta_{X},\beta_{Y})>\Delta\}\label{eq:-34}\\
 & =\sup_{\substack{\textrm{closed }A\subseteq\mathbb{S}_{X},B\subseteq\mathbb{S}_{Y}:\\
\theta(\beta_{X},\beta_{Y})>\Delta,\forall\beta_{X}\in A,\beta_{Y}\in B
}
}\Phi_{P_{X}}(A)+\Phi_{P_{Y}}(B)-1\label{eq:-19-6}\\
 & =\sup_{\textrm{compact }A\subseteq\mathbb{S}_{X}}\Phi_{P_{X}}(A)-\Phi_{P_{Y}}(\Gamma_{\theta\le\Delta}(A)),\label{eq:-20-6}
\end{align}
where for $A\subseteq\mathbb{S}_{X}$, 
\[
\Gamma_{\theta\le\Delta}(A):=\bigcup_{\beta_{X}\in A}\{\beta_{Y}\in\mathbb{S}_{Y}:\theta(\beta_{X},\beta_{Y})\le\Delta\}.
\]
In Strassen's OT problem in \eqref{eq:-34}, the marginals are two
Gaussian distributions and the cost function is the functional $(\beta_{X},\beta_{Y})\mapsto\theta(\beta_{X},\beta_{Y})$.
Equations \eqref{eq:-19-6} and \eqref{eq:-20-6} follow by Strassen's
duality in Theorem \ref{thm:strassen}. Recall the definition of $\mathbb{S}_{X}$
in \eqref{eq:-2}. In \eqref{eq:-19-6} and \eqref{eq:-20-6}, ``closed
$A$, $B$'' means that $A$ is closed in the space $\mathbb{S}_{X}$
(i.e., under the weak topology, or equivalently, the relative topology)
and $B$ is closed in $\mathbb{S}_{Y}$. We bound the asymptotics
of $\mathcal{G}_{\alpha_{0}+{\Delta}/{\sqrt{n}}}^{(n)}(P_{X},P_{Y})$
in the following theorem which is called the \emph{central limit theorem
(CLT) for Strassen's OT problem}. 
\begin{thm}[CLT for Strassen's OT]
\label{thm:CLT} Assume $\mathcal{X}$ and $\mathcal{Y}$ are finite,
and $c$ is finite. Assume $\alpha_{0}:=\mathcal{E}(P_{X},P_{Y})\in(c_{\inf},c_{\sup})$.
Then, we have 
\begin{align}
\Lambda_{\Delta}(P_{X},P_{Y}) & \leq\liminf_{n\to\infty}\mathcal{G}_{\alpha_{0}+{\Delta}/{\sqrt{n}}}^{(n)}(P_{X},P_{Y})\leq\limsup_{n\to\infty}\mathcal{G}_{\alpha_{0}+{\Delta}/{\sqrt{n}}}^{(n)}(P_{X},P_{Y})\leq\lim_{\Delta'\uparrow\Delta}\Lambda_{\Delta'}(P_{X},P_{Y}).\label{eq:CLTOT}
\end{align}
\end{thm}
Given $(P_{X},P_{Y})$, $\Lambda_{\Delta}(P_{X},P_{Y})$ is right-continuous
in $\Delta$, which means that 
\[
\lim_{n\to\infty}\mathcal{G}_{\alpha_{0}+{\Delta}/{\sqrt{n}}}^{(n)}(P_{X},P_{Y})=\Lambda_{\Delta}(P_{X},P_{Y})\textrm{ for almost every }\Delta\in\mathbb{R}.
\]
Furthermore, different from the CLT for empirical measures \cite[Theorem 14.3]{billingsley2013convergence},
the CLT for Strassen's OT here involves additional OT optimizations
in every term in \eqref{eq:CLTOT}. These optimizations are taken
over couplings of two empirical measures in the definition of $\mathcal{G}_{\alpha_{0}+{\Delta}/{\sqrt{n}}}^{(n)}(P_{X},P_{Y})$,
and over couplings of two Gaussian measures in the definition of $\Lambda_{\Delta}(P_{X},P_{Y})$.



\subsection{Connection to Empirical Optimal Transport }

\label{subsec:connection} It is well known that for a pair of empirical
measures $(T_{X},T_{Y})$, $\mathcal{E}(T_{X},T_{Y})$ is always attained
by an empirical joint measure. In other words, if we define the \emph{empirical
coupling set} for a pair of empirical measures $(T_{X},T_{Y})\in\mathcal{P}_{n}(\mathcal{X})\times\mathcal{P}_{n}(\mathcal{Y})$
as 
\[
\Pi_{n}(T_{X},T_{Y}):=\mathcal{P}_{n}(\mathcal{X}\times\mathcal{Y})\cap\Pi(T_{X},T_{Y})
\]
(i.e., the set of couplings of $(T_{X},T_{Y})$ which is discrete
and whose probability mass at each atom is a multiple of $1/n$),
and define the\emph{ empirical OT cost} for $(T_{X},T_{Y})\in\mathcal{P}_{n}(\mathcal{X})\times\mathcal{P}_{n}(\mathcal{Y})$
as 
\[
\mathcal{E}_{n}(T_{X},T_{Y}):=\min_{T_{XY}\in\Pi_{n}(T_{X},T_{Y})}\mathbb{E}_{(X,Y)\sim T_{XY}}[c(X,Y)],
\]
then $\mathcal{E}_{n}(T_{X},T_{Y})$ remains the same as $\mathcal{E}(T_{X},T_{Y})$. 
\begin{lem}[Empirical OT]
\cite[Page 5]{villani2003topics}\label{lem:empiricalOT} For a pair
of empirical measures $(T_{X},T_{Y})\in\mathcal{P}_{n}(\mathcal{X})\times\mathcal{P}_{n}(\mathcal{Y})$
and for all $n\ge1$, we have 
\begin{align}
\mathcal{E}_{n}(T_{X},T_{Y}) & =\mathcal{E}(T_{X},T_{Y}).\label{eq:empiricalOT}
\end{align}
\end{lem}
Such a result is a consequence of Birkhoff's theorem \cite[Page 5]{villani2003topics}.
Combining this lemma and \eqref{eq:nestedOT2} yields the fact that
characterizing the asymptotics of $\mathcal{G}_{\alpha}^{(n)}(P_{X},P_{Y})$,
as done in Sections \ref{subsec:Main-Result-1}-\ref{subsec:Main-Result-3},
is equivalent to characterizing the asymptotic behavior of the (random)
empirical OT cost $\mathcal{E}_{n}(T_{X^{n}},T_{Y^{n}})$ where $T_{X^{n}},T_{Y^{n}}$
are respectively the empirical measures of a pair of random vectors
$(X^{n},Y^{n})$ that follows the optimal coupling of $(P_{X}^{\otimes n},P_{Y}^{\otimes n})$
attaining the minimum in \eqref{eq:nStrassen}.

By definition, $c_{n}(x^{n},y^{n})=\mathbb{E}_{(X,Y)\sim T_{(x^{n},y^{n})}}[c(X,Y)]$
for all $(x^{n},y^{n})\in\mathcal{X}^{n}\times\mathcal{Y}^{n}$. Hence,
the empirical OT cost can be rewritten in the form of optimization
over sequences, i.e., for two given sequences $x^{n}$ and $y^{n}$
whose empirical measures are respectively $T_{X}$ and $T_{Y}$, we
have 
\begin{equation}
\mathcal{E}_{n}(T_{X},T_{Y})=\min_{\sigma}c_{n}(x_{\sigma}^{n},y^{n}),\label{eq:matching}
\end{equation}
where the minimization is taken over all permutations $\sigma$ on
$\{1,2,...,n\}$, and $x_{\sigma}^{n}$ is the resultant sequence
by permuting $x^{n}$ according to $\sigma$.

The minimization problem at the right-hand side of \eqref{eq:matching}
is known as the \emph{optimal matching problem }\cite{ajtai1984optimal,talagrand1992matching},
the optimal value of which, as shown in \eqref{eq:matching}, coincides
with the empirical OT cost. If $T_{X},T_{Y}$ are set to the empirical
measures of two \emph{independent} random vectors $X^{n},Y^{n}$,
each of which consists of i.i.d. components with a given distribution,
i.e., $(X^{n},Y^{n})\sim P_{X}^{\otimes n}\otimes P_{Y}^{\otimes n}$
for some $P_{X}$ and $P_{Y}$, then the induced empirical OT cost
$\mathcal{E}_{n}(T_{X^{n}},T_{Y^{n}})$ (or $\min_{\sigma}c_{n}(X_{\sigma}^{n},Y^{n})$)
is random as well. The asymptotic behavior of such $\mathcal{E}_{n}(T_{X^{n}},T_{Y^{n}})$
was widely studied in the literature; see for example \cite{ajtai1984optimal,talagrand1992matching,talagrand1993integrability,talagrand1994transportation,dobric1995asymptotics,del1999central,ambrosio2019pde,sommerfeld2018inference,tameling2019empirical,del2019central,dudley1969speed,boissard2014mean,fournier2015rate,weed2019sharp}.
In contrast, in our setting, specifically in \eqref{eq:nStrassen}
or \eqref{eq:nestedOT2}, the random vectors $X^{n},Y^{n}$ are not
necessarily independent. More precisely, their joint distribution
is implicitly specified by the minimization in \eqref{eq:nStrassen}
or \eqref{eq:nestedOT2} which is rather difficult to solve. Hence,
our setting is more complicated.

The LDP and MDP of the empirical OT cost 
were investigated in \cite{ganesh2007large}. In \cite[Theorem 3.1]{ganesh2007large},
Ganesh and O'Connell showed that in the large deviation regime, the
rate function of the empirical Wasserstein distance $W_{1}(T_{X^{n}},T_{Y^{n}})$
with $(X^{n},Y^{n})\sim P_{X}^{\otimes n}\otimes P_{Y}^{\otimes n}$
is $I:t\in\mathbb{R}\mapsto\inf_{Q_{X},Q_{Y}:W_{1}(Q_{X},Q_{Y})=t}D(Q_{X}\|P_{X})+D(Q_{Y}\|P_{Y}).$
This result is intuitive from Sanov's theorem, since $(X^{n},Y^{n})\sim P_{X}^{\otimes n}\otimes P_{Y}^{\otimes n}$
and hence the rate function of the empirical joint measure is the
sum of the one of the empirical measure induced by $P_{X}^{\otimes n}$
and the one induced by $P_{Y}^{\otimes n}$. A similar rate function
for the MDP was also derived in \cite{ganesh2007large}, but with
the relative entropy replaced by the half of the $\chi^{2}$-divergence.
As for the central limit regime, Tameling, Sommerfeld, and Munk \cite{sommerfeld2018inference,tameling2019empirical}
derived the limit law for the $\sqrt{n}$-scaled version of the empirical
Wasserstein distance for $P_{X},P_{Y}$ defined on the same \emph{countable}
metric space $(\mathcal{X},d)$. They showed that the limit law is
not Gaussian in general, but it indeed is if the optimal solution
(also known as the Kantorovich potential) to the Kantorovich dual
problem in \eqref{eq:dual} (or \eqref{eq:-111}) is unique. del Barrio
and Loubes \cite{del2019central} derived a similar central limit
theorem for the quadratic empirical Wasserstein distance which shows
that the limit law is Gaussian as well, if the distributions $P_{X},P_{Y}$
are distinct, absolutely continuous (with respect to the Lebesgue
measure in the Euclidean space), and have moments of order $4+\delta$
for some $\delta>0$ and positive densities on their convex supports.
These assumptions ensure that the optimal solution to the dual problem
in \eqref{eq:dual} is unique, which in turn implies that del Barrio
and Loubes's results are consistent with Tameling, Sommerfeld, and
Munk's. The case when $P_{X}=P_{Y}$ is the uniform distribution on
the unit hypercube was investigated widely in the literature; see
\cite{ajtai1984optimal,talagrand1992matching,talagrand1993integrability,talagrand1994transportation,ambrosio2019pde}.
The case of $P_{X}=P_{Y}$ was extended to other atomless measures
on Euclidean spaces in \cite{dudley1969speed,boissard2014mean,fournier2015rate,weed2019sharp}.
For these cases, the order of the empirical Wasserstein distance is
strictly larger than $\sqrt{n}$, which is hence different from the
countable case in \cite{sommerfeld2018inference,tameling2019empirical}
and the $P_{X}\neq P_{Y}$ case in \cite{del2019central}. In other
words, the asymptotic behavior of the empirical Wasserstein distance
in the central limit regime is sensitive to the factors whether $P_{X}$
and $P_{Y}$ are identical and whether $P_{X}$ and $P_{Y}$ are countably
supported. See relevant discussions in the introduction parts of \cite{del2019central,tameling2019empirical}.
All the LDP, MDP, and CLT results mentioned above for the empirical
Wasserstein distance are different from our results, since in these
results, the (random) empirical measures are independent, while in
our results, they are not. Even so, the convergence orders in our
results remain the same as the ones in these results under the same
settings. 
It is worth noting that in all of the related works mentioned above,
the Kantorovich duality plays a crucial role. In addition, instead
of studying the asymptotic behavior of the OT cost, Gozlan and Léonard
\cite{gozlan2007large} regarded the theory of large deviation as
a tool, and applied it to derive new transportation cost inequalities.

Under the product distribution $P_{X}^{\otimes n}\otimes P_{Y}^{\otimes n}$,
the (random) empirical measures $T_{X^{n}}$ of $X^{n}$ and $T_{Y^{n}}$
of $Y^{n}$ are independent, which means that the joint law of $(T_{X^{n}},T_{Y^{n}})$
for this case is $\mu_{n}\otimes\nu_{n}$. Obviously, $\mu_{n}\otimes\nu_{n}$
is a coupling of $\mu_{n}$ and $\nu_{n}$. On the other hand, in
the nested formula in \eqref{eq:nestedOT}, we minimize the probability
of the event $\{(Q_{X},Q_{Y}):\mathcal{E}(Q_{X},Q_{Y})>\alpha\}$
over all couplings of $\mu_{n}$ and $\nu_{n}$. Hence, 
\begin{align}
\mathcal{G}_{\alpha}^{(n)}(P_{X},P_{Y}) & \leq(\mu_{n}\otimes\nu_{n})\{(Q_{X},Q_{Y})\in\mathcal{P}(\mathcal{X})\times\mathcal{P}(\mathcal{Y}):\mathcal{E}(Q_{X},Q_{Y})>\alpha\}\nonumber \\
 & =\mathbb{P}_{(X^{n},Y^{n})\sim P_{X}^{\otimes n}\otimes P_{Y}^{\otimes n}}\{\mathcal{E}(T_{X^{n}},T_{Y^{n}})>\alpha\}.\label{eq:-16}
\end{align}
Determining the asymptotics of the probability in \eqref{eq:-16}
is just the empirical OT problem mentioned above, which involves only
one OT problem in Monge--Kantorovich's sense. In contrast, besides
Monge--Kantorovich's OT problem which acts as the inner subproblem,
our nested formula also involves Strassen's OT problem which acts
the outer subproblem. By \eqref{eq:-16}, our results in this paper
form lower bounds for the empirical OT problem.

\subsection{\label{subsec:Applications}Applications}

Beyond the theoretical interest of the problem, we would like to emphasize
the potential impact on information-theoretic applications of our
results. Here we provide an application to the covert reconstruction
problem in information-theoretic security \cite{bash2012limits,bash2015quantum,wang2016fundamental,bloch2016covert,yu2018asymptotic}.
Consider two terminals: a (legitimate) user and an eavesdropper. The
user observes a stationary memoryless stochastic process (also known
as a \emph{source}) $\{X_{i}\}$ with each $X_{i}~\sim P_{X}$, and
he/she wants to produce a reconstruction process $\{\hat{X}_{i}\}$,
i.e., a distorted version of the source. However, the reconstruction
device is being overheard by an eavesdropper all the time, no matter
whether the source is being reconstructed or not. When nothing is
being reconstructed, the process overhead by the eavesdropper is assumed
to be another stationary memoryless stochastic process (white noise
or a meaningless signal used to confuse the eavesdropper) $\{Y_{i}\}$
with each $Y_{i}~\sim P_{Y}$. The eavesdropper aims at detecting
whether there is a source being reconstructed at the current time
according to the distribution of the process he/she is observing.
Specifically, if the process that he/she is observing follows a distribution
distinct from the one of $\{Y_{i}\}$, then he/she will claim that
the source is being reconstructed; otherwise, he/she will claim that
the source is not being reconstructed. To avoid the eavesdropper to
detect the reconstruction successfully, 
the reconstruction of $\{X_{i}\}$ produced by the user must follow
the distribution same as $\{Y_{i}\}$. If we consider the cost function
$c$ as a measure of distortion, then the excess-cost probability
is a measure of distortion as well. In fact, the excess-cost probability
is also known as the \emph{excess-distortion probability}, which is
an important measure of distortion in information theory. For the
convert reconstruction problem above, what is the minimum excess-distortion
probability? It is easily checked that the minimum excess-distortion
probability for the first $n$ random variables of the source is $\mathcal{G}_{\alpha}^{(n)}(P_{X},P_{Y})$.
Hence, our results characterize the asymptotic behavior of the excess-distortion
probability for this problem. Furthermore, other applications of the
optimization problems over couplings to information theory can be
found in \cite{yu2018asymptotic}.

\subsection{\label{subsec:Notations}Notations and Organization}

As mentioned at the beginning of the introduction, $(\mathcal{X},\tau_{1})$
and $(\mathcal{Y},\tau_{2})$ are Polish spaces, and $P_{X}$ and
$P_{Y}$ are two probability measures (or distributions) defined respectively
on $\mathcal{X}$ and $\mathcal{Y}$. Here $P_{X}$ and $P_{Y}$ can
be thought of as the distributions of two random variables respectively
taking values in $\mathcal{X}$ and $\mathcal{Y}$. 
 We use $P_{X}\otimes P_{Y}$ to denote the product of $P_{X}$ and
$P_{Y}$, and $P_{X}^{\otimes n}$ (resp. $P_{Y}^{\otimes n}$) to
denote the $n$-fold product of $P_{X}$ (resp. $P_{Y}$). Throughout
this paper, for a topological space $(\mathcal{Z},\tau)$, we use
$\Sigma(\mathcal{Z},\tau)$ or simply $\Sigma(\mathcal{Z})$ to denote
the Borel $\sigma$-algebra on $\mathcal{Z}$ generated by the topology
$\tau$. Hence $(\mathcal{Z},\Sigma(\mathcal{Z}))$ forms a measurable
space. For this measurable space, we denote the set of probability
measures on $(\mathcal{Z},\Sigma(\mathcal{Z}))$ as $\mathcal{P}(\mathcal{Z},\Sigma(\mathcal{Z}))$
or simply $\mathcal{P}(\mathcal{Z})$. If we equip $\mathcal{P}(\mathcal{Z})$
with the weak topology, then the resultant space is a Polish space
as well. For brevity, we also denote it as $(\mathcal{P}(\mathcal{Z}),\Sigma(\mathcal{P}(\mathcal{Z})))$.

We denote $x^{n}=(x_{1},x_{2},...,x_{n})\in\mathcal{X}^{n}$ as a
sequence in $\mathcal{X}^{n}$. We use $T_{X}$ and $T_{Y}$ to respectively
denote empirical measures of sequences in $\mathcal{X}^{n}$ and $\mathcal{Y}^{n}$,
and $T_{XY}$ to denote an empirical joint measure of a pair of sequences
in $\mathcal{X}^{n}\times\mathcal{Y}^{n}$. 
We denote $\ell_{1}:x^{n}\in\mathcal{X}^{n}\mapsto T_{x^{n}}$
and $\ell_{2}:y^{n}\in\mathcal{Y}^{n}\mapsto T_{y^{n}}$ as the empirical
measure  functions, and  denote $\ell :(x^{n},y^n)\in\mathcal{X}^{n} \times \mathcal{Y}^{n} \mapsto T_{x^{n},y^{n}}$
 as the joint empirical
measure  function. 
For $P_{X}\in\mathcal{P}(\mathcal{X})$, denote $\mu_{n}$
as the law of the empirical measure $\ell_{1}(X^n)$ of $X^{n}\sim P_{X}^{\otimes n}$,
which means that $\mu_{n}$ is the push-forward measure 
$
\mu_{n} =P_{X}^{\otimes n} \circ \ell_{1}^{-1}
$. 
Obviously, $\mu_{n}$ is concentrated on $\mathcal{P}_{n}(\mathcal{X})$.
Similarly, for $P_{Y}\in\mathcal{P}(\mathcal{Y})$, denote $\nu_{n}$
as the law of the empirical measure of $Y^{n}\sim P_{Y}^{\otimes n}$.

We use $B_{\delta}(z):=\{z'\in\mathcal{Z}:d(z,z')<\delta\}$ and $B_{\leq\delta}(z):=\{z'\in\mathcal{Z}:d(z,z')\leq\delta\}$
to respectively denote an open ball and a closed ball. We use $\overline{A}$,
$A^{o}$, and $A^{c}:=\mathcal{Z}\backslash A$ to respectively denote
the closure, interior, and complement of the set $A$. Denote the
sublevel set of the relative entropy (or the divergence ``ball'')
as $D_{\leq\epsilon}(P_{X}):=\{Q_{X}:D(Q_{X}\|P_{X})\le\epsilon\}$
for $\epsilon\ge0$. As defined above, the Lévy--Prokhorov metric
on $\mathcal{P}(\mathcal{X})$ is $\mathtt{L}_{1}(Q_{X}',Q_{X})=\inf\{\delta:Q_{X}'(A)\le Q_{X}(A_{\delta})+\delta,\forall\textrm{ closed }A\subseteq\mathcal{X}\}$,
which is compatible with the weak topology. This metric, the TV distance,
and the relative entropy admit the following relation: For any $Q_{X},P_{X}$,
\begin{equation}
\sqrt{2D(Q_{X}\|P_{X})}\ge\|Q_{X}-P_{X}\|\ge\mathtt{L}_{1}(Q_{X},P_{X}),\label{eq:-35}
\end{equation}
which implies for $\epsilon\ge0$, 
\begin{equation}
D_{\leq\sqrt{2\epsilon}}(P_{X})\subseteq B_{\leq\epsilon}(P_{X}).\label{eq:D-B}
\end{equation}
The first inequality in \eqref{eq:-35} is known as Pinsker's inequality,
and the second inequality follows by definition.

We use $f(n,x)=o_{n|x}(1)$ to denote that given each $x$, $f(n,x)\to0$
pointwise as $n\to+\infty$. 
We denote $\inf\emptyset:=+\infty,\;\sup\emptyset:=-\infty$, and
$[k]:=\{1,2,...,k\}$. We denote $\|\cdot\|_{q}$ as the $\ell_{q}$-norm.

This paper is organized as follows. In Section 2, we first provide
the binary example to further illustrate our main results. In Section
3-6, we provide the proofs for the nested formula and the LDP, MDP,
and CLT results, respectively. Besides, some basic lemmas are provided
in Appendix \ref{sec:Basic-Lemmas}, and the proofs of some other
useful lemmas are provided in Appendices \ref{sec:Proof-of-Lemma-continuity}
and \ref{sec:Proof-of-Lemma-setA}.

\section{Binary Example}

\label{sec:int_exam}


To further illustrate our main results, we now focus on the binary
alphabet case, i.e., $\mathcal{X}=\mathcal{Y}=\{0,1\}$. We assume
$P_{X}=\mathrm{Bern}(a)$ and $P_{Y}=\mathrm{Bern}(b)$, where $0\le a\le b\le1$.
Consider the Hamming distance as the cost function, i.e., $c(x,y)=\bone_{x\neq y}$.
For this case, by \eqref{eqn:maxcoupling}, $\mathcal{E}(Q_{X},Q_{Y})$
coincides with the TV distance between $Q_{X},Q_{Y}$. For the case
$a=b$, $\mathcal{G}_{\alpha}^{(n)}(P_{X},P_{Y})$ is attained by
the identity coupling $P_{X}^{\otimes n}(x^{n})\bone_{y^{n}=x^{n}}$,
and for this case, 
$\mathcal{G}_{\alpha}^{(n)}(P_{X},P_{Y})=\bone_{\alpha<0}$ holds
for all $n\ge1$. In the following, we focus on the case $0\le a<b\le1$,
and apply Theorems \ref{thm:LDP}, \ref{thm:MDP}, and \ref{thm:CLT}
to this case. We obtain explicit expressions or bounds for the asymptotics
of $\mathcal{G}_{\alpha}^{(n)}(P_{X},P_{Y})$ in large deviations,
moderate deviations, and central limit regimes.

\subsection{Large Deviations Principle}

For distributions $Q_{X}=\mathrm{Bern}(a')$ and $Q_{Y}=\mathrm{Bern}(b')$,
we have 
\begin{equation}
\mathcal{E}(Q_{X},Q_{Y})=\min_{P_{XY}\in\Pi(Q_{X},Q_{Y})}\mathbb{E}[c(X,Y)]=|b'-a'|.\label{eq:-85}
\end{equation}
The minimum in $\mathcal{\mathcal{E}}(Q_{X},Q_{Y})$ is uniquely attained
by 
\[
Q_{XY}=\begin{cases}
\Big[\begin{array}{cc}
\overline{b'} & b'-a'\\
0 & a'
\end{array}\Big], & a'\leq b',\\
\Big[\begin{array}{cc}
\overline{a'} & 0\\
a'-b' & b'
\end{array}\Big], & a'>b'.
\end{cases}
\]
Here for a number $t\in[0,1]$, we define $\overline{t}:=1-t$.

Setting $a'\leftarrow a,b'\leftarrow b$ in \eqref{eq:-85}, we obtain
$\mathcal{\mathcal{E}}(P_{X},P_{Y})=b-a$ and the minimum in $\mathcal{E}(P_{X},P_{Y})$
is uniquely attained by $P_{XY}=\Big[\begin{array}{cc}
\overline{b} & b-a\\
0 & a
\end{array}\Big].$ For $a',a\in[0,1]$, denote $D(a'\|a):=D(\mathrm{Bern}(a')\|\mathrm{Bern}(a))$.
We have the following corollary to Theorem \ref{thm:LDP}. 
\begin{cor}[LDP for Binary OT]
\label{cor:binaryLDP} Given two Bernoulli distributions $P_{X}=\mathrm{Bern}(a)$
and $P_{Y}=\mathrm{Bern}(b)$ with $0\le a<b\le1$, we have: 
\begin{enumerate}
\item If $0<\alpha<b-a$, then 
\begin{align*}
\lim_{n\to\infty}-\frac{1}{n}\log(1-\mathcal{G}_{\alpha}^{(n)}(P_{X},P_{Y})) & =D(a^{*}\|a),
\end{align*}
where $a^{*}$ denotes the unique solution to the equation $D(a'+\alpha\|b)=D(a'\|a)$
in $[a,b-\alpha]$ with $a'$ unknown. 
\item If $b-a<\alpha<1$, then \eqref{eq:LD2} with $g(\alpha)=\min\{D(a^{*}\|a),D(b^{*}\|b)\}$
holds, 
where $a^{*}$ denotes the maximum of the two solutions to the equation
$D(a'+\alpha\|b)=D(a'\|a)$ (with $a'$ unknown) such that $0<a'\le b-\alpha$
(if there is only one or no such solution, then $D(a^{*}\|a):=+\infty$),
and $b^{*}$ denotes the minimum of the two solutions to the equation
$D(b'\|b)=D(b'-\alpha\|a)$ (with $b'$ unknown) such that $a+\alpha\le b'<1$
(if there is only one or no such solution, then $D(b^{*}\|b):=+\infty$). 
\end{enumerate}
\end{cor}
\begin{proof} The first statement of this corollary follows by observing
that 
\begin{align*}
f(\alpha) & =\min_{a',b':|b'-a'|\le\alpha}\max\{D(b'\|b),D(a'\|a)\}\\
 & =\min_{a'\in[a,b-\alpha]}\max\{D(a'+\alpha\|b),D(a'\|a)\}\\
 & =D(a^{*}\|a).
\end{align*}

We next prove the second statement. Observe that 
\begin{align}
g_{P_{X},P_{Y}}(\alpha) & =\inf_{a'}D(a'\|a)\label{eq:-31}
\end{align}
where the infimum is taken over all $a'$ such that $\min_{b':D(b'\|b)\leq D(a'\|a)}|b'-a'|>\alpha$,
or equivalently, $\min_{b':|b'-a'|\le\alpha}D(b'\|b)>D(a'\|a)$. It
is easily seen that the optimal $a'$ attaining the infimum in \eqref{eq:-31}
is no greater than $b-\alpha$, and for this case, $\min_{b':|b'-a'|\le\alpha}D(b'\|b)=D(a'+\alpha\|b)$.
Hence, 
\[
g_{P_{X},P_{Y}}(\alpha)=\inf_{a'\le b-\alpha:D(a'+\alpha\|b)>D(a'\|a)}D(a'\|a).
\]
However, if $\alpha>b$, then $g_{P_{X},P_{Y}}(\alpha)=+\infty$.
If there are two solutions to the equation $D(a'+\alpha\|b)=D(a'\|a)$
(with $a'$ unknown) such that $0<a'\le b-\alpha$, then $g_{P_{X},P_{Y}}(\alpha)=D(a^{*}\|a)$
where $a^{*}$ is the maximum among the two solutions. If there is
only one or no solution, then $g_{P_{X},P_{Y}}(\alpha)=+\infty$.

Similarly, 
\[
g_{P_{Y},P_{X}}(\alpha)=\inf_{b'\ge a+\alpha:D(b'-\alpha\|a)>D(b'\|b)}D(b'\|b).
\]
For the case $\alpha>1-a$, then $g_{P_{Y},P_{X}}(\alpha)=+\infty$.
If there are two solutions to the equation $D(b'\|b)=D(b'-\alpha\|a)$
(with $b'$ unknown) such that $b-\alpha\le b'<1$, then $g_{P_{Y},P_{X}}(\alpha)=D(b^{*}\|b)$
where $b^{*}$ is the minimum among the two solutions. If there is
only one or no solution, then $g_{P_{Y},P_{X}}(\alpha)=+\infty$.
 \end{proof} Corollary \ref{cor:binaryLDP} is illustrated in Fig.
\ref{fig:The-LDP-for}. 
\begin{figure}
\subfloat[\label{fig:The-LDP-for}]{\centering \includegraphics[width=0.5\columnwidth]{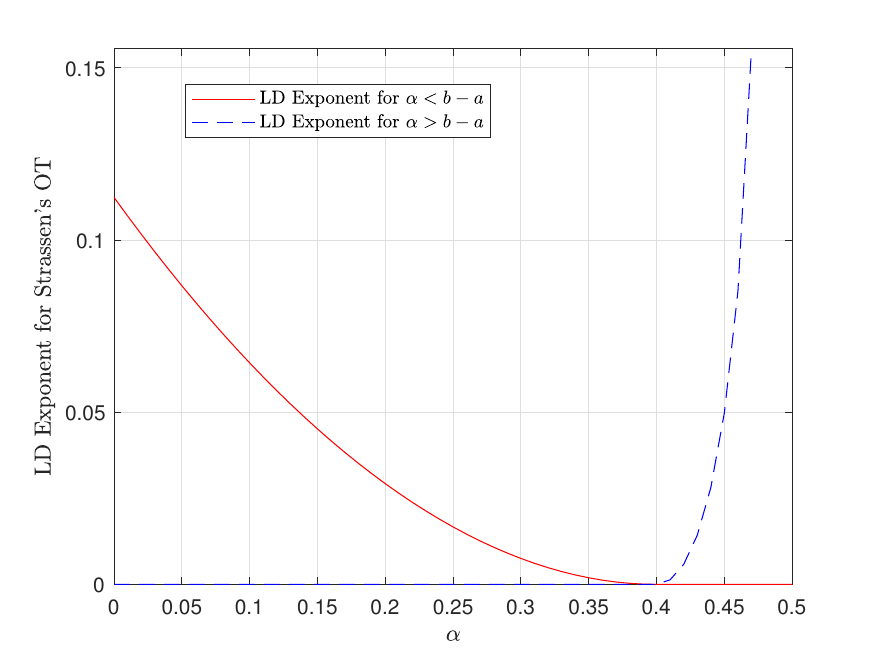}

}\subfloat[\label{fig:The-LDP-for-1}]{\centering \includegraphics[width=0.5\columnwidth]{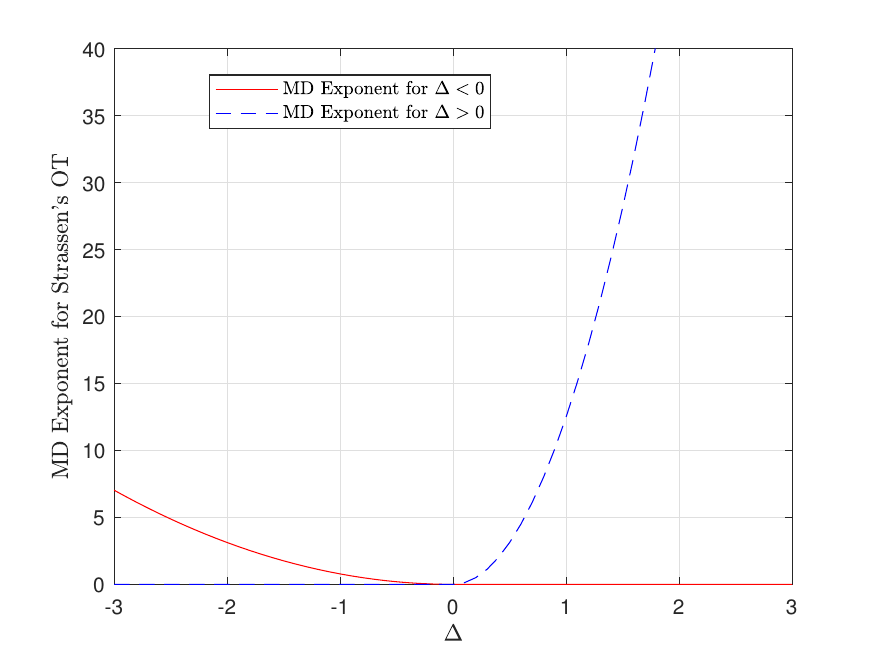}

}

\subfloat[\label{fig:The-LDP-for-1-1}]{\centering \includegraphics[width=0.5\columnwidth]{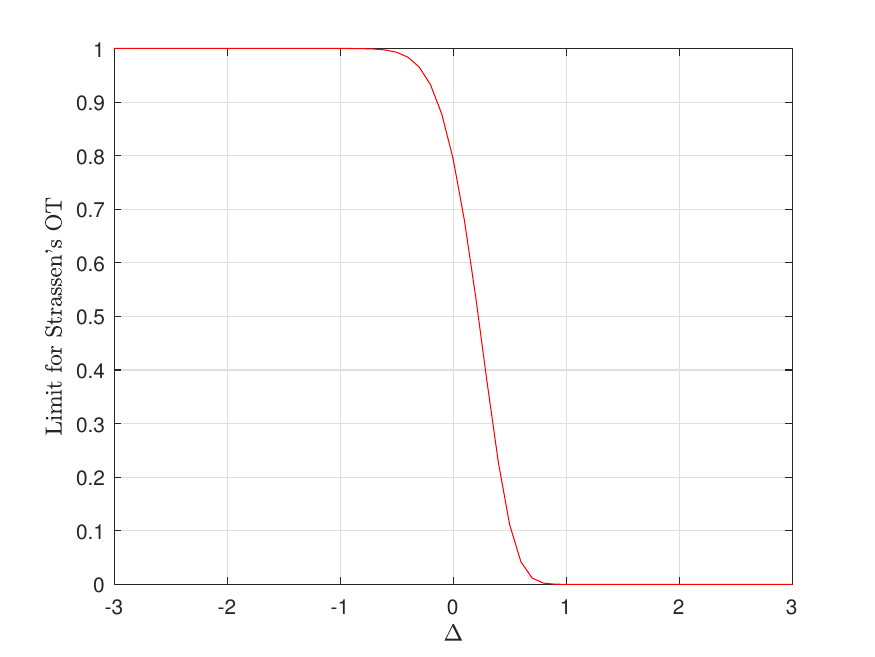}

}\caption{The LDP, MDP, and CLT for $P_{X}=\mathrm{Bern}(a)$ and $P_{Y}=\mathrm{Bern}(b)$
with $a=0.1$ and $b=0.5$.}
\end{figure}

\subsection{Moderate Deviations Principle}

We now focus on the moderate deviation regime. Let $P_{X}=\mathrm{Bern}(a)$
and $P_{Y}=\mathrm{Bern}(b)$ with $0<a<b<1$, and $\beta_{X}=(-a',a')$
and $\beta_{Y}=(-b',b')$. For this case, $\frac{1}{2}\sum_{x}\frac{\beta_{X}(x)^{2}}{P_{X}(x)}=\frac{a'^{2}}{2(a-a^{2})}$,
$\mathcal{S}=\{(0,0),(1,1),(0,1)\}$, and 
\begin{align*}
\theta(\beta_{X},\beta_{Y}) & =\min_{\begin{subarray}{c}
\beta_{XY}\in\overline{\Pi}(\beta_{X},\beta_{Y}),\\
\{(x,y):\beta_{XY}(x,y)<0\}\subseteq\mathcal{S}
\end{subarray}}\sum_{x,y}\beta_{XY}(x,y)c(x,y)=b'-a'.
\end{align*}

We have the following corollary to Theorem \ref{thm:MDP}. The proof
is similar to that of Corollary \ref{cor:binaryLDP}, and hence omitted
here. 
\begin{cor}[MDP for Binary OT]
\label{cor:binaryMDP} Let $P_{X}=\mathrm{Bern}(a)$ and $P_{Y}=\mathrm{Bern}(b)$
with $0<a<b<1$. Let $\alpha_{0}=b-a$. Let $\{a_{n}\}$ be a positive
sequence such that $a_{n}\to0$ and $na_{n}\to\infty$ as $n\to\infty$.
The following hold. 
\begin{enumerate}
\item If $\Delta<0$, then 
\begin{align*}
\lim_{n\to\infty}-a_{n}\log(1-\mathcal{G}_{\alpha_{0}+{\Delta}/{\sqrt{na_{n}}}}^{(n)}(P_{X},P_{Y})) & =\frac{1}{2}\Big(\frac{\Delta}{\sqrt{a-a^{2}}+\sqrt{b-b^{2}}}\Big)^{2}.
\end{align*}
\item If $\Delta>0$, then 
\[
\lim_{n\to\infty}-a_{n}\log\mathcal{G}_{\alpha_{0}+{\Delta}/{\sqrt{na_{n}}}}^{(n)}(P_{X},P_{Y})=\frac{1}{2}\Big(\frac{\Delta}{\sqrt{b-b^{2}}-\sqrt{a-a^{2}}}\Big)^{2}.
\]
\end{enumerate}
\end{cor}
Corollary \ref{cor:binaryMDP} is illustrated in Fig. \ref{fig:The-LDP-for-1}.

\subsection{Central Limit Theorem }

Recall the definition of $\mathbf{U}$ in Subsection \ref{subsec:Main-Result-3}.
For the binary case, $\mathbb{E}[\mathbf{U}]=(1-a,a)$ and $\mathrm{Cov}(\mathbf{U})=(a-a^{2})\begin{bmatrix}1 & -1\\
-1 & 1
\end{bmatrix}.$ Denote $\sigma_{X}^{2}:=\mathrm{Var}(U_{0})=a-a^{2}.$ The probability
density function of $U_{0}\sim\Phi_{U_{0}}:=\mathcal{N}(0,\sigma_{X}^{2})$
is 
\[
\phi_{X}(a')={\displaystyle \frac{1}{\sqrt{2\pi\sigma_{X}^{2}}}e^{-\frac{a'^{2}}{2\sigma_{X}^{2}}}},
\]
and the cumulative distribution function is 
\[
F_{X}(a')=\int_{-\infty}^{a'}\phi_{X}(t)\mathrm{d}t.
\]
For $P_{Y}$, define $\sigma_{Y}^{2}$, $\Phi_{V_{0}}$, $\phi_{Y}$,
and $F_{Y}$ similarly. Observe that $U_{0}+U_{1}=1$. Hence $\Phi_{P_{X}}$
is determined by $F_{X}$. Similarly, $\Phi_{P_{Y}}$ is determined
by $F_{Y}$. Hence by \eqref{eq:-20-6}, we have 
\begin{align}
\Lambda_{\Delta}(P_{X},P_{Y}) & =\sup_{\textrm{closed }A'\subseteq\mathbb{R}}\Phi_{U_{0}}(A')-\Phi_{V_{0}}(\bigcup_{a'\in A'}(-\infty,a'+\Delta])\nonumber \\
 & =\sup_{a'}F_{X}(a')-F_{Y}(a'+\Delta),\label{eq:-86}
\end{align}
where \eqref{eq:-86} follows since the difference $\Phi_{U_{0}}(A')-\Phi_{V_{0}}(\bigcup_{a'\in A'}(-\infty,a'+\Delta])$
is maximized only when $A'=(-\infty,a']$ for some $a'\in\mathbb{R}$.

To compute the optimal value of $a'$ in \eqref{eq:-86}, we need
the following lemma, which is derived by simple algebraic manipulations
and hence whose proof is omitted. 
\begin{lem}
If $\sigma_{X}^{2}=\sigma_{Y}^{2}$, then $a'=-\Delta/2$ is the unique
solution to the equation 
\begin{equation}
\phi_{X}(a')=\phi_{Y}(a'+\Delta).\label{eq:-87}
\end{equation}
If $\sigma_{X}^{2}\neq\sigma_{Y}^{2}$, then the equation \eqref{eq:-87}
has two solutions: 
\begin{align*}
a' & =\frac{-\sigma_{X}^{2}\Delta\pm\sigma_{X}\sigma_{Y}\sqrt{\Delta^{2}+2(\sigma_{X}^{2}-\sigma_{Y}^{2})\log\frac{\sigma_{X}}{\sigma_{Y}}}}{\sigma_{X}^{2}-\sigma_{Y}^{2}}.
\end{align*}
\end{lem}
If $a=b$ or $a+b=1$, then $\sigma_{X}^{2}=\sigma_{Y}^{2}$. By the
lemma above, for this case, $a'=-\Delta/2$ is the unique solution
of \eqref{eq:-87}. Hence for this case, 
\begin{align*}
\Lambda_{\Delta}(P_{X},P_{Y}) & =\begin{cases}
F_{X}(-\Delta/2)-F_{Y}(\Delta/2) & \Delta\leq0\\
0 & \Delta>0
\end{cases}.
\end{align*}

If $a\neq b$ and $a+b\neq1$, then the equation \eqref{eq:-87} has
two solutions. Denote them respectively as $a_{1}'(\Delta)$ and $a_{2}'(\Delta)$
such that $a_{1}'(\Delta)\le a_{2}'(\Delta)$. If additionally $a(1-a)\le b(1-b)$,
then $a_{2}'(\Delta)$ is the maximizer for the supremum in \eqref{eq:-86},
which implies that $\Lambda_{\Delta}(P_{X},P_{Y})=F_{X}(a_{2}'(\Delta))-F_{Y}(a_{2}'(\Delta)+\Delta).$
Similarly, if $a(1-a)>b(1-b)$, then $\Lambda_{\Delta}(P_{X},P_{Y})=F_{X}(a_{1}'(\Delta))-F_{Y}(a_{1}'(\Delta)+\Delta).$
Hence, we have the following corollary to Theorem \ref{thm:CLT}. 
\begin{cor}[CLT for Binary OT]
\label{cor:binaryCLT} Let $P_{X}=\mathrm{Bern}(a)$ and $P_{Y}=\mathrm{Bern}(b)$
with $0<a<b<1$. Let $\alpha_{0}=b-a$. Then, we have 
\begin{align*}
\lim_{n\to\infty}\mathcal{G}_{\alpha_{0}+{\Delta}/{\sqrt{n}}}^{(n)}(P_{X},P_{Y}) & =\begin{cases}
(F_{X}(-\Delta/2)-F_{Y}(\Delta/2))\bone_{\Delta\leq0} & a=1-b\\
F_{X}(a_{2}'(\Delta))-F_{Y}(a_{2}'(\Delta)+\Delta) & a<1-b\\
F_{X}(a_{1}'(\Delta))-F_{Y}(a_{1}'(\Delta)+\Delta) & a>1-b
\end{cases}.
\end{align*}
\end{cor}
Corollary \ref{cor:binaryCLT} is illustrated in Fig. \ref{fig:The-LDP-for-1-1}.

\section{Proof of Theorem \ref{thm:nestedOT}}

\label{sec:Proof-of-Lemma-Product} In this section, we prove Theorem
\ref{thm:nestedOT}. It suffices to prove that $\mathcal{G}_{\alpha}^{(n)}(P_{X},P_{Y})$
is equal to the expression in \eqref{eq:nesteddual} since the other
two expressions follow by Strassen's duality in Theorem \ref{thm:strassen}.
More specifically, note that since $\mathcal{X}$ is Polish, the space
$\mathcal{P}(\mathcal{X})$ with the weak topology is also Polish
\cite[Theorem 6.2 and Theorem 6.5]{parthasarathy2005probability}.
Similarly, $\mathcal{P}(\mathcal{Y})$ with the weak topology is Polish
as well. On the other hand, by Lemma \ref{lem:OTconvexitycontinuity}
(in Appendix \ref{sec:Basic-Lemmas}), $(Q_{X},Q_{Y})\mapsto\mathcal{E}(Q_{X},Q_{Y})$
is lower semi-continuous. Assumption \ref{ass:Nonempty-Condition-We}
implies that $\{(Q_{X},Q_{Y}):\mathcal{E}(Q_{X},Q_{Y})>\alpha\}$
is nonempty. 
Applying Strassen's duality in Theorem \ref{thm:strassen}, we obtain
\eqref{eq:nestedOT} and \eqref{eq:nesteddual2}.

We next that $\mathcal{G}_{\alpha}^{(n)}(P_{X},P_{Y})$ is equal to
the expression in \eqref{eq:nesteddual}. By the Kantorovich duality
in Theorem \ref{thm:Kantorovich} and the fact that Strassen's OT
problem is a special case of Monge--Kantorovich's OT problem, it
is not difficult to show that Strassen's OT problem also admits the
following duality; see \cite[Proof of Theorem 1.27]{villani2003topics}.
\begin{equation}
\mathcal{G}_{\alpha}^{(n)}(P_{X},P_{Y})=\sup_{(\phi,\psi)\in\Psi}\int_{\mathcal{X}^{n}}\phi\ \mathrm{d}P_{X}^{\otimes n}+\int_{\mathcal{Y}^{n}}\psi\ \mathrm{d}P_{Y}^{\otimes n}\label{eq:-1}
\end{equation}
where $\Psi$ is the set of all pairs $(\phi,\psi)\in L^{1}(\mathcal{X}^{n})\times L^{1}(\mathcal{Y}^{n})$
such that 
\[
\begin{cases}
\phi(x^{n})+\psi(y^{n})\leq\bone_{c_{n}>\alpha}(x^{n},y^{n}),\;\forall(x^{n},y^{n}),\\
0\le\phi\le1,\qquad-1\le\psi\leq0,\\
\phi\textrm{ is upper semi-continuous}.
\end{cases}
\]
Note that $\Psi$ is convex. 

Observe that $P_{X}^{\otimes n}$, $P_{Y}^{\otimes n}$, and $c_{n}$
are permutation-invariant (or $n$-symmetric) in the sense that for
any permutation $\sigma$ of $[n]$, it holds that $P_{X}^{\otimes n}=P_{X}^{\otimes n}\circ\sigma^{-1},P_{Y}^{\otimes n}=P_{Y}^{\otimes n}\circ\sigma^{-1}$,
and $c_{n}=c_{n}\circ(\sigma^{-1},\sigma^{-1})$. Hence, we can additionally
assume $(\phi,\psi)$ is also permutation-invariant, since otherwise,
we can take average of $(\phi,\psi)\circ(\sigma^{-1},\sigma^{-1})$
over all permutation $\sigma$ of $[n]$. We denote $\overline{\Psi}$
as the set of $(\phi,\psi)\in\Psi$ such that $(\phi,\psi)$ is permutation-invariant
(i.e., $(\phi,\psi)=(\phi,\psi)\circ(\sigma^{-1},\sigma^{-1})$ for
any $\sigma$). Note that $\overline{\Psi}$ is still convex. Moreover,
$\overline{\Psi}$ can be represented as a convex combination of pairs
of indicators of measurable subsets. Here a set $A$ is said to be
\emph{permutation-invariant} if $x^{n}\in A$ if and only if its arbitrary
permutations belong to $A$ as well. 
\begin{lem}
\cite[Proof of Theorem 1.27]{villani2003topics} $\overline{\Psi}$
can be represented as a convex combination of pairs of the form $(\bone_{A},-\bone_{B})$
for permutation-invariant and measurable $A\subseteq\mathcal{X}^{n},B\subseteq\mathcal{Y}^{n}$
such that $A$ is closed and $\bone_{A}(x^{n})+\bone_{B}(y^{n})\leq\bone_{c_{n}>\alpha}(x^{n},y^{n}),\;\forall(x^{n},y^{n})$. 
\end{lem}
An original version of this lemma without the ``permutation-invariant''
condition was proven in \cite[Proof of Theorem 1.27]{villani2003topics}
by using the ``layer cake representation''. It is easy to check
that the proof still works when we impose the ``permutation-invariant''
condition. 

By the lemma above and observing that the objective function in \eqref{eq:-1}
is linear in $(\phi,\psi)$, we can rewrite \eqref{eq:-1} as 
\[
\mathcal{G}_{\alpha}^{(n)}(P_{X},P_{Y})=\sup_{A,B}P_{X}^{\otimes n}(A)+P_{Y}^{\otimes n}(B),
\]
where the supremization is taken over all pairs of permutation-invariant
$A\in\Sigma(\mathcal{X}^{n}),B\in\Sigma(\mathcal{Y}^{n})$ such that
\[
\begin{cases}
\bone_{A}(x^{n})+\bone_{B}(y^{n})\leq\bone_{c_{n}>\alpha}(x^{n},y^{n}),\;\forall(x^{n},y^{n})\\
A\textrm{ is closed}.
\end{cases}
\]
Recall that    $\ell_{1}:x^{n}\in\mathcal{X}^{n}\mapsto T_{x^{n}}$
and $\ell_{2}:y^{n}\in\mathcal{Y}^{n}\mapsto T_{y^{n}}$ denote the empirical
measure  functions. Then, any permutation-invariant sets $A,B$
can be written as $A=\ell_{1}^{-1}(E),B=\ell_{2}^{-1}(F)$ for some
$E\subseteq\mathcal{P}(\mathcal{X}),F\subseteq\mathcal{P}(\mathcal{Y})$.
Note that $E,F$ are not necessarily measurable with respect to the
$\sigma$-algebras induced by weak topologies. We can rewrite \eqref{eq:-1}
as 
\begin{equation}
\mathcal{G}_{\alpha}^{(n)}(P_{X},P_{Y})=\sup_{E,F}P_{X}^{\otimes n}(\ell_{1}^{-1}(E))+P_{Y}^{\otimes n}(\ell_{2}^{-1}(F)),\label{eq:-10}
\end{equation}
where the supremization is taken over all pairs of $E\subseteq\mathcal{P}(\mathcal{X}),F\subseteq\mathcal{P}(\mathcal{Y})$
such that 
\[
\begin{cases}
\bone_{E}(T_{X})+\bone_{F}(T_{Y})\leq\bone_{\mathcal{E}_{n}>\alpha}(T_{X},T_{Y}),\;\forall(T_{X},T_{Y})\\
\ell_{1}^{-1}(E)\textrm{ is closed},\ell_{2}^{-1}(F)\in\Sigma(\mathcal{Y}^{n}).
\end{cases}
\]

The function $\ell_{1},\ell_{2}$ are continuous. For any $E,F$ such
that $\ell_{1}^{-1}(E),\ell_{2}^{-1}(F)$ are Borel subsets of Polish
spaces $\mathcal{X}^{n},\mathcal{Y}^{n}$, we have that $E=\ell_{1}(\ell_{1}^{-1}(E)),F=\ell_{2}(\ell_{2}^{-1}(F))$
are analytic subsets of $\mathcal{P}(\mathcal{X}),\mathcal{P}(\mathcal{Y})$
respectively.  Since given a probability measure, analytic sets are universally measurable and
every measurable set in the completion of this probability measure space is the union of a Borel set
and a subset of a null set (of this probability measure), there exists a Borel set $E'\subseteq E$
such that $\mu_{n}(E')=P_{X}^{\otimes n}(\ell_{1}^{-1}(E))$. Similarly,
there exists a Borel set $F'\subseteq F$ such that $\nu_{n}(F')=P_{Y}^{\otimes n}(\ell_{2}^{-1}(F))$.

We  claim that $E$ is closed in $\mathcal{P}(\mathcal{X})$ if and
only if $\ell_{1}^{-1}(E)$ is closed in $\mathcal{X}^{n}$. We now
prove it. On one hand, since $\ell_{1}$ is continuous, for any closed
$E$ in $\mathcal{P}(\mathcal{X})$, $\ell_{1}^{-1}(E)$ is closed
in $\mathcal{X}^{n}$. On the other hand, we next show that for any
closed $A$ in $\mathcal{X}^{n}$, $\ell_{1}(A)$ is closed in $\mathcal{P}(\mathcal{X})$.
Let $\{T_{X}^{(k)}\}_{k\in\mathbb{N}}$ be a sequence of empirical
measures that belongs to $\ell_{1}(A)$ and converges to some $T_{X}$
(under the weak topology).  Any empirical measure $T_{X}$ can be
written as $T_{X}=\frac{1}{n}\sum_{i=1}^{n}\delta_{x_{i}}$ for some
sequence $x^{n}$. So does $T_{X}^{(k)}$, i.e., for each $k$, $T_{X}^{(k)}=\frac{1}{n}\sum_{i=1}^{n}\delta_{x_{i}^{(k)}}$
for some sequence $x^{n,(k)}$. Let $f:\mathcal{X}\to\mathbb{R}$
be a continuous bounded function given by $f(x)=\sum_{i\in[n]}[\delta-d(x_{i},x)]^{+}$
where $\delta>0$, $[t]^{+}:=\max\{t,0\}$, and $d$ is the metric
on $\mathcal{X}$.  By definition of the weak topology, $T_{X}^{(k)}\to T_{X}$
implies $\varint f\mathrm{d}T_{X}^{(k)}\to\varint f\mathrm{d}T_{X}$,
i.e., $\sum_{i=1}^{n}f(x_{i}^{(k)})\to n\delta$. This further implies
that  there exists a sequence of permutations $\left\{ \sigma_{k}\right\} $
of $[n]$ such that $\hat{x}_{i}^{(k)}\to x_{i}$ as $k\to\infty$
uniformly for all $i\in[n]$, where $\hat{x}^{n,(k)}:=x_{\sigma_{k}}^{n,(k)}$
denotes the rearrangement of $x_{\sigma_{k}}^{n,(k)}$ via $\sigma_{k}$.
Equivalently, in the product space, $\hat{x}^{n,(k)}\to x^{n}$ as
$k\to\infty$ under the product topology. Since $A$ is closed, we
know that $x^{n}\in A$. Hence, $T_{X}=\ell_{1}(x^{n})\in\ell_{1}(A)$.
That is, $\ell_{1}(A)$ is closed. 

By the claim above, 
\begin{equation}
\mathcal{G}_{\alpha}^{(n)}(P_{X},P_{Y})=\sup_{E,F}\mu_{n}(E)+\nu_{n}(F),\label{eq:-10-1}
\end{equation}
where the supremization is taken over all pairs of measurable $E\subseteq\mathcal{P}(\mathcal{X}),F\subseteq\mathcal{P}(\mathcal{Y})$
such that 
\begin{equation}
\begin{cases}
\bone_{E}(T_{X})+\bone_{F}(T_{Y})\leq\bone_{\mathcal{E}_{n}>\alpha}(T_{X},T_{Y}),\;\forall(T_{X},T_{Y})\\
E\textrm{ is closed}.
\end{cases}\label{eq:-21}
\end{equation}
By the inner regularity of the probability measures, without changing
the value of the supremization above, we can require both $E,F$ to
be closed. Moreover, the first condition in \eqref{eq:-21} is equivalent
to $\mathcal{E}(Q_{X},Q_{Y})>\alpha,\forall Q_{X}\in E,Q_{Y}\in F$.
Hence, we have \eqref{eq:nesteddual}. 
\begin{rem}
Theorem \ref{thm:nestedOT} can be also proven from the primal formulation
in \eqref{eq:nestedOT}. See the intuition given below Theorem \ref{thm:nestedOT}. 
\end{rem}

\section{Proof of Theorem \protect\ref{thm:LDP}}

\label{sec:Proof-of-Theorem-LDP} In this section, we prove Theorem
\protect\ref{thm:LDP}.

\subsection{Case of $\alpha<\mathcal{E}(P_{X},P_{Y})$}

In this subsection, we prove \eqref{eq:LD1}. To this end, it suffices
to prove the following result. Without the assumption of UCOTF, it
holds that for $\alpha<\mathcal{E}(P_{X},P_{Y})$, 
\begin{align}
f^{-}(\alpha) & \leq\liminf_{n\to\infty}-\frac{1}{n}\log(1-\mathcal{G}_{\alpha}^{(n)}(P_{X},P_{Y}))\label{eqn:-1}\\
 & \leq\limsup_{n\to\infty}-\frac{1}{n}\log(1-\mathcal{G}_{\alpha}^{(n)}(P_{X},P_{Y}))\leq f^{+}(\alpha),\label{eqn:-2}
\end{align}
where 
\begin{align}
f^{+}(\alpha) & :=\lim_{\epsilon\downarrow0}\inf_{\substack{Q_{X},Q_{Y}:\mathcal{E}(Q_{X}',Q_{Y}')\le\alpha,\\
\forall Q_{X}'\in B_{\epsilon}(Q_{X}),Q_{Y}'\in B_{\epsilon}(Q_{Y})
}
}\max\{D(Q_{Y}\|P_{Y}),D(Q_{X}\|P_{X})\}\label{eq:-99}\\
f^{-}(\alpha) & :=\lim_{\epsilon\downarrow0}\inf_{\substack{Q_{X},Q_{Y}:\mathcal{E}(Q_{X}',Q_{Y}')\le\alpha,\\
\exists Q_{X}'\in B_{\epsilon}(Q_{X}),Q_{Y}'\in B_{\epsilon}(Q_{Y})
}
}\max\{D(Q_{Y}\|P_{Y}),D(Q_{X}\|P_{X})\}\label{eq:-100}
\end{align}
with $B_{\epsilon}(\cdot)$ denoting a ball of radius $\epsilon$
under the Lévy--Prokhorov metric.

Equation \eqref{eq:LD1} is a consequence of \eqref{eqn:-1} and \eqref{eqn:-2}
as shown in the following. By the UCOTF assumption, $f^{-}(\alpha)\ge\lim_{\alpha'\downarrow\alpha}f(\alpha')$
and $f^{+}(\alpha)\le\lim_{\alpha'\uparrow\alpha}f(\alpha')$. On
the other hand, by the convexity of the relative entropy and Lemmas
\ref{lem:OTconvexitycontinuity} and \ref{lem:continuityofinf2} in
Appendix \ref{sec:Basic-Lemmas}, $f$ is continuous on $(c_{\min},+\infty)$.
Hence, \eqref{eq:LD1} holds. 


\subsubsection{\label{subsec:Lower-Bound:}Lower Bound}

By Lemma \ref{thm:nestedOT}, 
\begin{align}
1-\mathcal{G}_{\alpha}^{(n)}(P_{X},P_{Y}) & =\inf_{\textrm{closed }A,B:\mathcal{E}(Q_{X},Q_{Y})>\alpha,\forall Q_{X}\in A,Q_{Y}\in B}\mu_{n}(A^{c})+\nu_{n}(B^{c}).\label{eqn:1_G}
\end{align}
Let $E=D_{\le r}(P_{X})$ and $F=D_{\le r}(P_{Y})$ be two sublevel
sets of the relative entropies for $r>0$. Then by \cite[Theorem 20]{Erven},
$E$ and $F$ are compact. By the definition of compactness, for any
$\epsilon>0$, there exists a \emph{cover} $\{B_{\epsilon}(Q_{X,i})\}_{i=1}^{k_{1}}$
with a finite size $k_{1}$ for $E$. That is, there exists a positive
integer $k_{1}$ and a collection $\{B_{\epsilon}(Q_{X,i})\}_{i=1}^{k_{1}}$
of $k_{1}$ open balls in $\mathcal{P}(\mathcal{X})$ such that $E\subseteq E_{\epsilon}:=\bigcup_{i=1}^{k_{1}}B_{\epsilon}(Q_{X,i})$.
Similarly, there also exists another cover $\{B_{\epsilon}(Q_{Y,i})\}_{i=1}^{k_{2}}$
with a finite size $k_{2}$ for $F_{2}$. Define $F_{\epsilon}:=\bigcup_{i=1}^{k_{2}}B_{\epsilon}(Q_{Y,i})$.
Define $E_{\le\epsilon}:=\bigcup_{i=1}^{k_{1}}B_{\le\epsilon}(Q_{X,i})$
and $F_{\le\epsilon}:=\bigcup_{i=1}^{k_{2}}B_{\le\epsilon}(Q_{Y,i})$,
which are closed. 

We choose $r,\epsilon>0$ such that $\mathcal{E}(Q_{X},Q_{Y})>\alpha,\forall Q_{X}\in E_{\le\epsilon},Q_{Y}\in F_{\le\epsilon}$.
Then, the set pair $(E_{\le\epsilon},F_{\le\epsilon})$ constructed
here satisfy the constraints in the optimization at the right-hand
side of \eqref{eqn:1_G}. Hence, the right-hand side of \eqref{eqn:1_G}
is upper bounded by $\mu_{n}(E_{\le\epsilon}^{c})+\nu_{n}(F_{\le\epsilon}^{c}).$
By Sanov's theorem \cite[Theorem 6.2.10]{Dembo}, for fixed $r,\epsilon>0$,
we have 
\begin{align*}
\liminf_{n\to\infty}-\frac{1}{n}\log\mu_{n}(E_{\le\epsilon}^{c}) & \ge\inf_{Q_{X}\in E_{\epsilon}^{c}}D(Q_{X}\|P_{X})\ge r,\\
\liminf_{n\to\infty}-\frac{1}{n}\log\nu_{n}(F_{\le\epsilon}^{c}) & \ge\inf_{Q_{Y}\in F_{\epsilon}^{c}}D(Q_{Y}\|P_{Y})\ge r.
\end{align*}
Therefore, 
\begin{align*}
\liminf_{n\to\infty}-\frac{1}{n}\log(1-\mathcal{G}_{\alpha}^{(n)}(P_{X},P_{Y})) & \ge\liminf_{n\to\infty}-\frac{1}{n}\log[\mu_{n}(E_{\le\epsilon}^{c})+\nu_{n}(F_{\le\epsilon}^{c})]\ge r.
\end{align*}

We can take infimum over all feasible $r,\epsilon>0$, and then obtain
that $\liminf_{n\to\infty}-\frac{1}{n}\log(1-\mathcal{G}_{\alpha}^{(n)}(P_{X},P_{Y}))$
is lower bounded by 
\begin{align}
\sup_{\substack{r,\epsilon>0:\mathcal{E}(Q_{X},Q_{Y})>\alpha,\forall Q_{X}\in E_{\le\epsilon},Q_{Y}\in F_{\le\epsilon}}
}r & \ge\sup_{\substack{r,\epsilon>0:\mathcal{E}(Q_{X},Q_{Y})>\alpha,\forall Q_{X}\in E^{2\epsilon},Q_{Y}\in F^{2\epsilon}}
}r\\
 & =\sup_{\epsilon>0}\inf_{r>0:\mathcal{E}(Q_{X},Q_{Y})\leq\alpha,\exists Q_{X}\in E^{2\epsilon},Q_{Y}\in F^{2\epsilon}}r\label{eq:-4}
\end{align}
where $E^{2\epsilon}:=\bigcup_{Q_{X}\in E}B_{2\epsilon}(Q_{X})$ and
$F^{2\epsilon}:=\bigcup_{Q_{Y}\in F}B_{2\epsilon}(Q_{Y})$, and the
equality above follows by the monotonicity of the sublevel sets $E,F$
and the continuity of real numbers. Note that for each $Q_{X}\in E^{2\epsilon}$,
there is $Q_{X}'\in E$ such that $Q_{X}\in B_{2\epsilon}(Q_{X}')$,
and for each $Q_{Y}\in F^{2\epsilon}$, there is $Q_{Y}'\in F$ such
that $Q_{Y}\in B_{2\epsilon}(Q_{Y}')$. By the definition of $E,F$,
we have $r\ge\max\{D(Q_{X}'\|P_{X}),D(Q_{Y}'\|P_{Y})\}$. Hence, \eqref{eq:-4}
is further lower bounded by $f^{-}(\alpha)$ given in \eqref{eq:-100}.
(Note that the notations $Q_{X}',Q_{Y}'$ and $Q_{X},Q_{Y}$ are exchanged
in the definition of $f^{-}(\alpha)$.)

\subsubsection{Upper Bound}

In the following, we use a splitting technique to design a desired
coupling $\pi$ of $\mu_{n}$ and $\nu_{n}$. Let $(Q_{X},Q_{Y})$
be a pair of distributions such that 
\begin{equation}
B_{\epsilon}(Q_{X})\times B_{\epsilon}(Q_{Y})\subseteq\{(Q_{X},Q_{Y}):\mathcal{E}(Q_{X},Q_{Y})\leq\alpha\}\label{eq:-120}
\end{equation}
for sufficiently small $\epsilon>0$. If there is no such pair, then
$f^{+}(\alpha)=+\infty$, and hence, the upper bound $f^{+}(\alpha)$
in \eqref{eqn:-2} holds trivially. Denote 
\[
p:=\min\{\mu_{n}(B_{\epsilon}(Q_{X})),\nu_{n}(B_{\epsilon}(Q_{Y}))\}.
\]
By large deviations theory, it is not difficult to see that $p>0$
for sufficiently large $n$; this point will be confirmed later. Denote
$\mu_{n|B_{\epsilon}(Q_{X})}$ as the conditional distribution induced
by $\mu_{n}$ given the event $B_{\epsilon}(Q_{X})$. The conditional
distribution $\nu_{n|B_{\epsilon}(Q_{Y})}$ is defined similarly.
Define two new distributions 
\begin{align*}
 & \mu_{n}':=\frac{\mu_{n}-p\cdot\mu_{n|B_{\epsilon}(Q_{X})}}{1-p},\qquad\nu_{n}':=\frac{\nu_{n}-p\cdot\nu_{n|B_{\epsilon}(Q_{Y})}}{1-p}.
\end{align*}
Then $\mu_{n}$ and $\nu_{n}$ can be written as the following mixtures:
\begin{align*}
\mu_{n} & =(1-p)\mu_{n}'+p\cdot\mu_{n|B_{\epsilon}(Q_{X})},\qquad\nu_{n}=(1-p)\nu_{n}'+p\cdot\nu_{n|B_{\epsilon}(Q_{Y})}.
\end{align*}
This is the so-called \emph{splitting technique}, which was previously
used to study limit theorems of recurrent Markov processes \cite{nummelin1978uniform,athreya1978new},
used to construct a coupling of the original Markov chain and the
target Markov chain in the study of the mixing rate of Markov Chain
Monte Carlo (MCMC) \cite{roberts2004general}, and also used to prove
the noncompact version of the Kantorovich duality given in Theorem
\ref{thm:Kantorovich} \cite{villani2003topics}.

We now define a new mixture distribution 
\begin{align}
 & \pi:=(1-p)\cdot\mu_{n}'\otimes\nu_{n}'+p\cdot\mu_{n|B_{\epsilon}(Q_{X})}\otimes\nu_{n|B_{\epsilon}(Q_{Y})}.\label{eq:-121}
\end{align}
Obviously, $\pi\in\Pi(\mu_{n},\nu_{n})$. Moreover, by \eqref{eq:-120}
and \eqref{eq:-121}, we have 
\[
\pi\{(Q_{X},Q_{Y}):\mathcal{E}(Q_{X},Q_{Y})\leq\alpha\}\geq\pi(B_{\epsilon}(Q_{X})\times B_{\epsilon}(Q_{Y}))\geq p.
\]
Combining this with the nested formula in Theorem \ref{thm:nestedOT}
yields that 
\begin{align*}
\limsup_{n\to\infty}-\frac{1}{n}\log(1-\mathcal{G}_{\alpha}^{(n)}(P_{X},P_{Y})) & \leq\limsup_{n\to\infty}-\frac{1}{n}\log p\\
 & \leq\max\big\{\inf_{Q_{X}'\in B_{\epsilon}(Q_{X})}D(Q_{X}'\|P_{X}),\inf_{Q_{Y}'\in B_{\epsilon}(Q_{Y})}D(Q_{Y}'\|P_{Y})\big\}\\
 & \leq\max\{D(Q_{X}\|P_{X}),D(Q_{Y}\|P_{Y})\},
\end{align*}
where the second inequality follows by Sanov's theorem. Since $(Q_{X},Q_{Y})$
is an arbitrary pair of distributions satisfying \eqref{eq:-120}
and $\epsilon>0$ in \eqref{eq:-120} is also arbitrary, we have 
\begin{align*}
 & \limsup_{n\to\infty}-\frac{1}{n}\log(1-\mathcal{G}_{\alpha}^{(n)}(P_{X},P_{Y}))\le f^{+}(\alpha).
\end{align*}

\subsection{Case of $\alpha>\mathcal{E}(P_{X},P_{Y})$}

\subsubsection{Lower Bound}

We now prove the direction of ``$\ge$'' in \eqref{eq:LD2}, i.e.,
$g(\alpha)$ is a lower bound on the left side of \eqref{eq:LD2}.
Compared to the case of $\alpha<\mathcal{E}(P_{X},P_{Y})$, our proof
for the case $\alpha>\mathcal{E}(P_{X},P_{Y})$ is more complicated,
especially for the lower bound case. This is because for this case,
it seems difficult to construct an explicit coupling that asymptotically
attains the lower bound on the exponent. So, instead, we utilize Strassen's
dual formula given in \eqref{eq:nesteddual} to derive the lower bound.

We first provide a heuristic proof idea for the lower bound case.
For the supremum in \eqref{eq:nesteddual}, it does not change if
we restrict $\nu_{n}(B^{c})\le\mu_{n}(A)$. That is, 
\begin{equation}
\mathcal{G}_{\alpha}^{(n)}(P_{X},P_{Y})=\sup_{\substack{\textrm{closed }A,B:\nu_{n}(B^{c})\le\mu_{n}(A),\\
\mathcal{E}(Q_{X},Q_{Y})>\alpha,\forall Q_{X}\in A,Q_{Y}\in B
}
}\mu_{n}(A)-\nu_{n}(B^{c}),.\label{eq:-13-3-3}
\end{equation}
which yields the following simpler upper bound by omitting the negative
term $-\nu_{n}(B^{c})$. 
\begin{equation}
\mathcal{G}_{\alpha}^{(n)}(P_{X},P_{Y})\leq\sup_{\substack{\textrm{closed }A,B:\nu_{n}(B^{c})\le\mu_{n}(A),\\
\mathcal{E}(Q_{X},Q_{Y})>\alpha,\forall Q_{X}\in A,Q_{Y}\in B
}
}\mu_{n}(A).\label{eq:-73}
\end{equation}
By Sanov's theorem \cite[Theorem 6.2.10]{Dembo}, roughly speaking,
$-\frac{1}{n}\log\mu_{n}(A)=D(Q_{X}\|P_{X})+o_{n|A}(1)$ for some
$Q_{X}$ and $-\frac{1}{n}\log\nu_{n}(B^{c})=D(Q_{Y}\|P_{Y})+o_{n|B}(1)$
for some $Q_{Y}$. The latter means $B^{c}\subseteq\{Q_{Y}':D(Q_{Y}'\|P_{Y})\geq D(Q_{Y}\|P_{Y})+o_{n|B}(1)\}$.
Furthermore, to approach the exponent of the supremum in the right-hand
side of \eqref{eq:-73}, the sets $A$ and $B$ should be chosen as
small as possible under the conditions that $\nu_{n}(B^{c})\le\mu_{n}(A)$
and that the exponent of $\mu_{n}(A)$ remains unchanged. Hence we
should choose $A=\{Q_{X}\}$ and $B=\{Q_{Y}':D(Q_{Y}'\|P_{Y})\le D(Q_{Y}\|P_{Y})\}$.
Substituting these into \eqref{eq:-73}, we obtain 
\begin{align}
 & \liminf_{n\to\infty}-\frac{1}{n}\log\mathcal{G}_{\alpha}^{(n)}(P_{X},P_{Y})\nonumber \\
 & \geq\inf_{\substack{Q_{X},Q_{Y}:D(Q_{Y}\|P_{Y})\geq D(Q_{X}\|P_{X}),\\
\mathcal{E}(Q_{X},Q_{Y}')>\alpha\textrm{ for all }Q_{Y}'\textrm{ s.t. }D(Q_{Y}'\|P_{Y})\leq D(Q_{Y}\|P_{Y})
}
}D(Q_{X}\|P_{X})\label{eq:-73-2}\\
 & =\inf_{Q_{X}:\mathcal{E}(Q_{X},Q_{Y}')>\alpha\textrm{ for all }Q_{Y}'\textrm{ s.t. }D(Q_{Y}'\|P_{Y})\leq D(Q_{X}\|P_{X})}D(Q_{X}\|P_{X})\label{eqn:-113}\\
 & =g(\alpha),\label{eq:-113}
\end{align}
where \eqref{eq:-113} follows since, roughly speaking, given $Q_{X}$,
the optimal $Q_{Y}$ in \eqref{eq:-73-2} satisfies $D(Q_{Y}\|P_{Y})=D(Q_{X}\|P_{X})$.
We should note that there are two obstacles in this heuristic proof. 
\begin{enumerate}
\item \label{enu:Although-the-expression} Although we claim that the optimal
$B$ should be $\{Q_{Y}':D(Q_{Y}'\|P_{Y})\le D(Q_{Y}\|P_{Y})\}$ in
the proof idea above, this is not necessarily true since we only know
that $B\supseteq\{Q_{Y}':D(Q_{Y}'\|P_{Y})<D(Q_{Y}\|P_{Y})\}$. This
implies that to obtain a lower bound, we only can relax $B$ to $\{Q_{Y}':D(Q_{Y}'\|P_{Y})<D(Q_{Y}\|P_{Y})\}$.
This difference is subtle but crucial, since if we replace ``$\le$''
in \eqref{eqn:-113} with ``$<$'', then \eqref{eqn:-113} becomes
zero (by observing that $Q_{X}=P_{X}$ is feasible in the infimization
of \eqref{eqn:-113}). However, $g(\alpha)$ is bounded away from
zero. Hence, \eqref{eq:-113} is ``discontinuous'' in the feasible
region in the sense that whether excluding the point $Q_{X}=P_{X}$
from the feasible region in \eqref{eq:-113} will result in different
values. In the following, we provide a formal proof, in which we clear
away this obstacle by excluding $P_{X}$ 
from $A$. That is, we add the constraint $Q_{X}\neq P_{X}$ into
the infimum in \eqref{eqn:-113}, which makes the value of \eqref{eqn:-113}
do not change, and more precisely, remain to be equal to $g(\alpha)$
no matter whether we replace ``$\le$'' in \eqref{eqn:-113} with
``$<$''. 
\item \label{enu:Another-difficulty-is}Another obstacle is that in order
to show the inequality in \eqref{eq:-73-2}, we need to swap $\liminf_{n\to\infty}$
and the infimization operation in \eqref{eq:-73-2}. However, this
is not feasible in general. In the formal proof, we use a covering
technique (or compactness technique) to address this obstacle. 
\end{enumerate}
We next provide a formal proof for the lower bound $g(\alpha)$.

We denote $P_{XY}\in\Pi(P_{X},P_{Y})$ as a coupling such that $\mathbb{E}_{P_{XY}}[c(X,Y)]<\alpha$.
This is feasible since $\mathcal{E}(P_{X},P_{Y})<\alpha$. Denote
$P_{XY}^{\otimes n}$ as the $n$-fold product of $P_{XY}$. Then
by the definition of $\mathcal{G}_{\alpha}^{(n)}(P_{X},P_{Y})$ and
by weak law of large number, 
\begin{align*}
\mathcal{G}_{\alpha}^{(n)}(P_{X},P_{Y}) & \leq P_{XY}^{\otimes n}\Big\{(x^{n},y^{n}):\frac{1}{n}\sum_{i=1}^{n}c(x_{i},y_{i})>\alpha\Big\}\to0\mbox{ as }n\to\infty.
\end{align*}
Hence in Strassen's dual formula in \eqref{eq:nesteddual}, $\mu_{n}(A)+\nu_{n}(B)-1$
converges to zero. That is, given any $\delta>0$ and for sufficiently
large $n$, it suffices to restrict $A,B$ in the constraints in \eqref{eq:nesteddual}
to satisfy that $\mu_{n}(A)\le\frac{1}{2}+\delta$ or $\nu_{n}(B)\le\frac{1}{2}+\delta$.
Therefore, for any $\delta\in(0,\frac{1}{2})$, we have 
\begin{align}
\mathcal{G}_{\alpha}^{(n)}(P_{X},P_{Y})=\max\{\Upsilon_{1},\Upsilon_{2}\},\label{eqn:maxUp}
\end{align}
where 
\begin{align}
\Upsilon_{1} & :=\sup_{\substack{\textrm{closed }A,B:\mu_{n}(A)\le\frac{1}{2}+\delta,\\
\mathcal{E}(Q_{X},Q_{Y})>\alpha,\forall Q_{X}\in A,Q_{Y}\in B
}
}\mu_{n}(A)-\nu_{n}(B^{c}),\label{eqn:-Up_1}\\
\Upsilon_{2} & :=\sup_{\substack{\textrm{closed }A,B:\nu_{n}(B)\le\frac{1}{2}+\delta,\\
\mathcal{E}(Q_{X},Q_{Y})>\alpha,\forall Q_{X}\in A,Q_{Y}\in B
}
}\nu_{n}(B)-\mu_{n}(A^{c}).\label{eqn:-Up_2}
\end{align}

First consider the term \eqref{eqn:-Up_1}. For the optimization problem
in \eqref{eqn:-Up_1}, it suffices to consider $A,B$ such that 
\begin{equation}
\nu_{n}(B^{c})\le\mu_{n}(A)\le\frac{1}{2}+\delta.\label{eq:-3}
\end{equation}
Now we exclude a neighborhood of $P_{X}$ from $A$. We show that
the condition ``$\mu_{n}(A)\le\frac{1}{2}+\delta$'' in the constraints
under the supremum in \eqref{eqn:-Up_1} can be replaced by ``$A\subseteq B_{\epsilon}(P_{X})^{c}$''
for sufficiently small $\epsilon$. 
\begin{lem}
\label{lem:For-the-optimization} Assume UCOTF. Then for $\delta\in(0,\frac{1}{2})$
and for $\epsilon>0$ such that 
\begin{equation}
\epsilon+\sup_{Q_{X}\in B_{\epsilon}(P_{X})}\mathcal{E}(Q_{X},P_{Y})<\alpha,\label{eq:-11}
\end{equation}
there exists an $N_{\delta,\epsilon}\in\mathbb{N}$ such that for
all $n\ge N_{\delta,\epsilon}$, 
\begin{equation}
\sup_{\substack{\textrm{closed }A,B:\nu_{n}(B^{c})\le\mu_{n}(A)\le\frac{1}{2}+\delta,\\
\mathcal{E}(Q_{X},Q_{Y})>\alpha,\forall Q_{X}\in A,Q_{Y}\in B
}
}\mu_{n}(A)-\nu_{n}(B^{c})\leq\sup_{\substack{\textrm{closed }A,B:A\subseteq B_{\epsilon}(P_{X})^{c},\\
\mathcal{E}(Q_{X},Q_{Y})>\alpha,\forall Q_{X}\in A,Q_{Y}\in B
}
}\mu_{n}(A)-\nu_{n}(B^{c}).\label{eq:-94}
\end{equation}
\end{lem}
\begin{proof}[Proof of Lemma \ref{lem:For-the-optimization}] Define
$A_{\epsilon}:=\{Q_{X}:\mathcal{E}(Q_{X},P_{Y})\geq\alpha-\epsilon\}.$
Denote $(A_{n},B_{n})$ as any pair of closed sets satisfying the
constraints in the left side of \eqref{eq:-94}, i.e., they satisfy
that 
\begin{align}
\nu_{n}(B_{n}^{c}) & \le\mu_{n}(A_{n})\le\frac{1}{2}+\delta,\label{eq:-32}\\
\mathcal{E}(Q_{X},Q_{Y}) & >\alpha,\;\forall Q_{X}\in A_{n},Q_{Y}\in B_{n}.\label{eq:-28}
\end{align}
In order to show $A_{n}\subseteq B_{\epsilon}(P_{X})^{c}$ for sufficiently
large $n$, we first prove that for any $\epsilon>0$, $A_{n}\subseteq A_{\epsilon}$
holds for sufficiently large $n$, and then prove that $A_{\epsilon}\subseteq B_{\epsilon}(P_{X})^{c}$
holds for all $\epsilon>0$ satisfying \eqref{eq:-11}.

We now prove $A_{n}\subseteq A_{\epsilon}$ by contradiction. Suppose
that for infinitely many $n$, there is $Q_{X}^{(n)}\in A_{n}$ such
that 
\begin{equation}
\mathcal{E}(Q_{X}^{(n)},P_{Y})<\alpha-\epsilon.\label{eq:-27}
\end{equation}
By the assumption of UCOTF, for $\epsilon'>0$, 
\begin{equation}
\sup_{Q_{Y}\in B_{\epsilon'}(P_{Y})}\mathcal{E}(Q_{X}^{(n)},Q_{Y})\leq\mathcal{E}(Q_{X}^{(n)},P_{Y})+o_{\epsilon'}(1).\label{eq:-26}
\end{equation}
Let $\epsilon'$ be small enough such that $o_{\epsilon'}(1)<\epsilon$.
Then combining \eqref{eq:-27} and \eqref{eq:-26}, we have 
\begin{equation}
\sup_{Q_{Y}\in B_{\epsilon'}(P_{Y})}\mathcal{E}(Q_{X}^{(n)},Q_{Y})\leq\alpha.\label{eq:-30}
\end{equation}
Combining \eqref{eq:-28} and \eqref{eq:-30} yields that $B_{\epsilon'}(P_{Y})\subseteq B_{n}^{c}$.
On the other hand, by Sanov's theorem \cite[Theorem 6.2.10]{Dembo},
for fixed $\epsilon'>0$, $\nu_{n}(D_{\le\epsilon'^{2}/8}(P_{Y}))\to1$
as $n\to\infty$, which, combined with \eqref{eq:D-B}, implies that
$\nu_{n}(B_{\epsilon'}(P_{Y}))\to1$ as $n\to\infty$. This contradicts
with the condition $\nu_{n}(B_{n}^{c})\le\frac{1}{2}+\delta$ (see
\eqref{eq:-32}). Hence for sufficiently large $n$, $A_{n}\subseteq A_{\epsilon}.$

By the condition in \eqref{eq:-11} and the definition of $A_{\epsilon}$,
we have $B_{\epsilon}(P_{X})\subseteq A_{\epsilon}^{c}$, i.e., $A_{\epsilon}\subseteq B_{\epsilon}(P_{X})^{c}.$
This completes the proof of Lemma \ref{lem:For-the-optimization}.
\end{proof} By the assumption of UCOTF, $\epsilon+\sup_{Q_{X}\in B_{\epsilon}(P_{X})}\mathcal{E}(Q_{X},P_{Y})\to\mathcal{E}(P_{X},P_{Y})$
as $\epsilon\downarrow0$. On the other hand, $\mathcal{E}(P_{X},P_{Y})<\alpha$.
Hence \eqref{eq:-11} holds for all sufficiently small $\epsilon>0$.
Hence, given a sufficiently small $\epsilon>0$, for all sufficiently
large $n$, it suffices to consider $A$ such that $A\subseteq B_{\epsilon}(P_{X})^{c}$
in \eqref{eqn:-Up_1}.

In fact, Obstacle \ref{enu:Although-the-expression} has been addressed
now since we have already added the condition $A\subseteq B_{\epsilon}(P_{X})^{c}$
into the constraints. Such a condition will exclude the distribution
$P_{X}$ from the feasible region in the final expression $g(\alpha)$.
We will discuss this near the end of this proof. We next address Obstacle
\ref{enu:Another-difficulty-is} by using a covering technique.

Denote $F_{1}:=D_{\le s}(P_{X})$ and $F_{2}:=D_{\le s}(P_{Y})$.
Then by \cite[Theorem 20]{Erven}, $F_{1}$ and $F_{2}$ are compact.
By compactness, for any $\delta>0$, there exists a cover $\{B_{\delta}(Q_{X,i})\}_{i=1}^{k_{1}}$
(consisting of $k_{1}$ equal balls) for $F_{1}$. 
Similarly, there also exists another cover $\{B_{\delta}(Q_{Y,i})\}_{i=1}^{k_{2}}$
with a finite size $k_{2}$ for $F_{2}$. Define $G_{1}:=\bigcup_{i=1}^{k_{1}}B_{\le\delta}(Q_{X,i})$
and $G_{2}:=\bigcup_{i=1}^{k_{2}}B_{\le\delta}(Q_{Y,i})$. Obviously,
$G_{1}$ and $G_{2}$ are closed.

Now we continue \eqref{eq:-94}: The right-hand side of \eqref{eq:-94}
is further upper bounded by 
\begin{align}
 & \sup_{\substack{\textrm{closed }A,B:A\subseteq B_{\epsilon}(P_{X})^{c},\\
\mathcal{E}(Q_{X},Q_{Y})>\alpha,\forall Q_{X}\in A,Q_{Y}\in B
}
}\mu_{n}(A\cap{G_{1}})+\nu_{n}(B\cap{G_{2}})-1+\mu_{n}(G_{1}^{c})+\nu_{n}(G_{2}^{c})\nonumber \\
 & =\mu_{n}(G_{1}^{c})+\nu_{n}(G_{2}^{c})+\sup_{\substack{\textrm{closed }A\subseteq{G_{1}},B\subseteq{G_{2}}:A\subseteq B_{\epsilon}(P_{X})^{c},\\
\mathcal{E}(Q_{X},Q_{Y})>\alpha,\forall Q_{X}\in A,Q_{Y}\in B
}
}\mu_{n}(A)+\nu_{n}(B)-1.\label{eq:-95}
\end{align}

By Sanov's theorem \cite[Theorem 6.2.10]{Dembo}, for $i=1,2$, $\liminf_{n\to\infty}-\frac{1}{n}\log\mu_{n}(G_{i}^{c})\ge\inf_{Q_{X}\in G_{i}^{c}}D(Q_{X}\|P_{X})\ge s.$
Hence if we choose $s>g(\alpha)$, then the exponent of $\mu_{n}(G_{1}^{c})+\nu_{n}(G_{2}^{c})$
would be larger than $g(\alpha)$. Hence, for this case, to show the
lower bound $g(\alpha)$ on the exponent of the optimal ECP, we only
need to prove the exponent of the supremum term in \eqref{eq:-95}
is also larger than or equal to $g(\alpha)$.

Let $A\subseteq{G_{1}},B\subseteq{G_{2}}$ be two sets satisfying
the constraints under the supremum in \eqref{eq:-95}, i.e., they
are closed and satisfy that $A\subseteq B_{\epsilon}(P_{X})^{c}$
and $\mathcal{E}(Q_{X},Q_{Y})>\alpha,\forall Q_{X}\in A,Q_{Y}\in B$.
We denote $L_{1}:=\{i\in[k_{1}]:{B}_{\le\delta}(Q_{X,i})\cap A\neq\emptyset\}$
and $L_{2}:=\{i\in[k_{2}]:{B}_{\le\delta}(Q_{Y,i})\cap B\neq\emptyset\}$.
By definition, obviously the following property holds.

\begin{property} \label{property:-For-every}For every $i\in L_{1},j\in L_{2}$,
there exist $Q_{X}\in{B}_{\le\delta}(Q_{X,i}),Q_{Y}\in{B}_{\le\delta}(Q_{Y,j})$
such that $Q_{X}\in B_{\epsilon}(P_{X})^{c},\mathcal{E}(Q_{X},Q_{Y})>\alpha$.
\end{property}

Now we set $\delta=\epsilon/4$. By the triangle inequality and the
assumption of UCOTF, Property \ref{property:-For-every} implies Property
\ref{property:-For-every2}.

\begin{property} \label{property:-For-every2}For every $i\in L_{1},j\in L_{2}$,
all $Q_{X}\in{B}_{\le\epsilon/4}(Q_{X,i}),Q_{Y}\in{B}_{\le\epsilon/4}(Q_{Y,j})$
satisfy $Q_{X}\in B_{\epsilon/2}(P_{X})^{c},\mathcal{E}(Q_{X},Q_{Y})>\alpha-\kappa(\epsilon)$,
where $\kappa:(0,+\infty)\to(0,+\infty)$ is some positive and increasing
function such that $\lim_{\epsilon\downarrow0}\kappa(\epsilon)=0$.
\end{property} 

We now upper bound the supremum term in \eqref{eq:-95} as follows.
First, observe that the objective function $\mu_{n}(A)+\nu_{n}(B)-1$
is upper bounded by $\mu_{n}\Big(\bigcup_{i\in L_{1}}{B}_{\le\epsilon/4}(Q_{X,i})\Big)+\nu_{n}\Big(\bigcup_{i\in L_{2}}{B}_{\le\epsilon/4}(Q_{Y,i})\Big)-1$.
By Sanov's theorem \cite[Theorem 6.2.10]{Dembo}, it is further upper
bounded by $e^{-n(E_{1}+o_{n|\epsilon}(1))}-e^{-n(E_{2}+o_{n|\epsilon}(1))}$,
where 
\[
E_{1}:=\inf_{Q_{X}\in\bigcup_{i\in L_{1}}{B}_{\le\epsilon/4}(Q_{X,i})}D(Q_{X}\|P_{X}),\qquad E_{2}:=\inf_{Q_{Y}\in(\bigcup_{i\in L_{2}}{B}_{\le\epsilon/4}(Q_{Y,i}))^{c}}D(Q_{Y}\|P_{Y})
\]
Rigorously speaking, the terms $o_{n|\epsilon}(1)$ in the exponents
above depend on the union sets, or equivalently, depend on the sets
$L_{1},L_{2}$. However, such dependence can be removed, since $k_{1}$
and $k_{2}$ are finite and fixed. That is, given $\epsilon$, the
terms $o_{n|\epsilon}(1)$ in the exponents above can be made to converge
to zero uniformly for all $L_{1}\subseteq[k_{1}],L_{2}\subseteq[k_{2}]$,
and hence $o_{n|\epsilon}(1)$ can be assumed to be independent of
$L_{1},L_{2}$. Combining all above, the supremum term in \eqref{eq:-95}
can be upper bounded by 
\begin{align}
 & e^{-no_{n|\epsilon}(1)}\sup_{L_{1}\subseteq[k_{1}],L_{2}\subseteq[k_{2}]:\textrm{ Property }\ref{property:-For-every2}}e^{-n(E_{1}+o_{n|\epsilon}(1))}-e^{-nE_{2}}.\label{eq:-13}
\end{align}
If we relax the union sets $\bigcup_{i\in L_{1}}{B}_{\le\epsilon/4}(Q_{X,i}),\bigcup_{i\in L_{2}}{B}_{\le\epsilon/4}(Q_{Y,i})$
to any closed sets $A,B$ such that all $Q_{X}\in A,Q_{Y}\in B$ satisfy
$Q_{X}\in B_{\epsilon/2}(P_{X})^{c},\mathcal{E}(Q_{X},Q_{Y})>\alpha-\kappa(\epsilon)$,
then the supremum term in \eqref{eq:-13} is further upper bounded
by 
\begin{align}
\sup_{\substack{\textrm{closed }A,B:Q_{X}\in B_{\epsilon/2}(P_{X})^{c},\\
\mathcal{E}(Q_{X},Q_{Y})>\alpha-\kappa(\epsilon),\forall Q_{X}\in A,Q_{Y}\in B
}
}e^{-n(\inf_{Q_{X}\in A}D(Q_{X}\|P_{X})+o_{n|\epsilon}(1))}-e^{-n\inf_{Q_{Y}\in B^{c}}D(Q_{Y}\|P_{Y})}.\label{eq:-96}
\end{align}
Note that the term $o_{n|\epsilon}(1)$ is independent of $A,B$.
Until now, Obstacle \ref{enu:Another-difficulty-is} has been addressed.

Furthermore, \eqref{eq:-96} can be rewritten as 
\begin{align}
\sup_{\substack{\textrm{closed }A,B:Q_{X}\in B_{\epsilon/2}(P_{X})^{c},\\
\mathcal{E}(Q_{X},Q_{Y})>\alpha-\kappa(\epsilon),\forall Q_{X}\in A,Q_{Y}\in B
}
}\sup_{Q_{X}\in A} & e^{-n(D(Q_{X}\|P_{X})+o_{n|\epsilon}(1))}-e^{-n\inf_{Q_{Y}\in B^{c}}D(Q_{Y}\|P_{Y})}.\label{eq:-115}
\end{align}
Obviously, to approach the supremum in \eqref{eq:-115}, the set $A$
should be as small as possible. Hence without loss of optimality,
we can restrict that $A=\{Q_{X}\}$. That is, \eqref{eq:-115} can
be further rewritten as 
\begin{align}
 & \sup_{\substack{Q_{X},\textrm{closed }B:Q_{X}\in B_{\epsilon/2}(P_{X})^{c},\\
\mathcal{E}(Q_{X},Q_{Y})>\alpha-\kappa(\epsilon),\forall Q_{Y}\in B
}
}e^{-n(D(Q_{X}\|P_{X})+o_{n|\epsilon}(1))}-e^{-n\inf_{Q_{Y}\in B^{c}}D(Q_{Y}\|P_{Y})}.\label{eq:-33}
\end{align}

For set $B$, define $r:=\inf_{Q_{Y}\in B^{c}}D(Q_{Y}\|P_{Y})$. Then
$B^{c}\subseteq\{Q_{Y}:D(Q_{Y}\|P_{Y})\geq r\}$, i.e., 
\begin{equation}
B\supseteq\{Q_{Y}:D(Q_{Y}\|P_{Y})<r\}.\label{eq:-18}
\end{equation}
Given any sufficiently small $\epsilon>0$ and any $\epsilon'>0$,
for all sufficiently large $n$, 
\eqref{eq:-33} is upper bounded by 
\begin{align}
\sup_{\substack{r,Q_{X}:Q_{X}\in B_{\epsilon/2}(P_{X})^{c},\\
\mathcal{E}(Q_{X},Q_{Y})>\alpha-\kappa(\epsilon)\textrm{ for all }Q_{Y}\textrm{ s.t. }D(Q_{Y}\|P_{Y})<r
}
}e^{-n(D(Q_{X}\|P_{X})-\epsilon')}-e^{-nr},\label{eq:-17}
\end{align}
which follows since $o_{n|\epsilon}(1)\ge-\epsilon'$ for sufficiently
large $n$ given $\epsilon$, and moreover, by \eqref{eq:-18}, the
closed set $B$ is relaxed to $\{Q_{Y}:D(Q_{Y}\|P_{Y})<r\}$. 



By the equivalence of a statement and its contrapositive, we have
\begin{align}
 & \mathcal{E}(Q_{X},Q_{Y})>\alpha-\kappa(\epsilon)\textrm{ for all }Q_{Y}\textrm{ s.t. }D(Q_{Y}\|P_{Y})<r\\
\Longleftrightarrow\qquad & D(Q_{Y}\|P_{Y})\geq r\textrm{ for all }Q_{Y}\textrm{ s.t. }\mathcal{E}(Q_{X},Q_{Y})\leq\alpha-\kappa(\epsilon).\label{eqn:equicont}
\end{align}
Hence, \eqref{eq:-17} is further upper bounded by 
\begin{align}
\sup_{\substack{r,Q_{X}:Q_{X}\in B_{\epsilon/2}(P_{X})^{c},\\
\phi_{Q_{X}}(\alpha-\kappa(\epsilon))\geq r
}
}e^{-n(D(Q_{X}\|P_{X})-\epsilon')}-e^{-nr} & =\sup_{Q_{X}\in B_{\epsilon/2}(P_{X})^{c}}e^{-n(D(Q_{X}\|P_{X})-\epsilon')}-e^{-n\varphi_{Q_{X}}(\alpha-\kappa(\epsilon))},\label{eq:-54-1}
\end{align}
where $\phi_{Q_{X}}(t):=\inf_{Q_{Y}:\mathcal{E}(Q_{X},Q_{Y})\leq t}D(Q_{Y}\|P_{Y})$
for $t\ge0$. Since \eqref{eq:-54-1} is an upper bound on a nonnegative
quantity, it is nonnegative as well. Hence, without loss of optimality,
one can add the condition $\phi_{Q_{X}}(\alpha-\kappa(\epsilon))\geq D(Q_{X}\|P_{X})-\epsilon'$
into the constraints under the supremization in the right side of
\eqref{eq:-54-1}. Moreover, the second term (i.e., the negative one)
can be removed, in order to obtain a further upper bound.

Combining all points above (from \eqref{eq:-3} to the current point)
and taking $\liminf_{n\to\infty}-\frac{1}{n}\log$, we have that for
all sufficiently small $\epsilon,\epsilon'>0$, 
\begin{align}
\underline{\mathrm{E}}_{X}(\alpha) & :=\liminf_{n\to\infty}-\frac{1}{n}\log\Upsilon_{1}\geq\inf_{Q_{X}\in\mathcal{Q}_{\epsilon,\epsilon'}}D(Q_{X}\|P_{X})-\epsilon',\label{eq:-47-1}
\end{align}
where $\Upsilon_{1}$ is defined in \eqref{eqn:-Up_1}, and 
\begin{align}
\mathcal{Q}_{\epsilon,\epsilon'} & :=\{Q_{X}\in B_{\epsilon/2}(P_{X})^{c}:\,\phi_{Q_{X}}(\alpha-\kappa(\epsilon))\geq D(Q_{X}\|P_{X})-\epsilon'\}.\label{eq:-22}
\end{align}
From \eqref{eq:-47-1}, we obtain that 
\begin{align}
\underline{\mathrm{E}}_{X}(\alpha) & \geq\lim_{\epsilon\downarrow0}\lim_{\epsilon'\downarrow0}\inf_{Q_{X}\in\mathcal{Q}_{\epsilon,\epsilon'}}D(Q_{X}\|P_{X}).\label{eq:-37}
\end{align}
We next remove the limits in the lower bound above. 

Define $\phi_{Q_{X}}^{-}(t):=\lim_{s\uparrow t}\phi_{Q_{X}}(s)$.
Obviously, $\phi_{Q_{X}}^{-}(t)\ge\varphi_{Q_{X}}(t)$, and $\phi_{Q_{X}}^{-}$
is nonincreasing (since $\phi_{Q_{X}}$ is nonincreasing) and left-continuous.
Hence, 
\begin{align}
\mathcal{Q}_{\epsilon,\epsilon'} & \subseteq\mathcal{Q}_{\epsilon,\epsilon'}^{-}:=\{Q_{X}\in B_{\epsilon/2}(P_{X})^{c}:\,\phi_{Q_{X}}^{-}(\alpha-\kappa(\epsilon))\geq D(Q_{X}\|P_{X})-\epsilon'\}.\label{eq:-21-1}
\end{align}

We now claim that given $\epsilon,\epsilon'$, $\mathcal{Q}_{\epsilon,\epsilon'}^{-}$
is closed. To show this, it suffices to show that given $t>0$, $Q_{X}\mapsto\varphi_{Q_{X}}^{-}(t)$
is upper semi-continuous (under the weak topology). This follows since,
on one hand, for any sequence $\{Q_{X}^{(k)}\}$ such that $Q_{X}^{(k)}\rightarrow Q_{X}$
as $k\to\infty$, by the assumption of UCOTF, we have 
\begin{equation}
\limsup_{k\to\infty}\phi_{Q_{X}^{(k)}}^{-}(t)\le\lim_{s\uparrow t}\phi_{Q_{X}}^{-}(s)=\phi_{Q_{X}}^{-}(t).\label{eq:-23}
\end{equation}
Hence, $\mathcal{Q}_{\epsilon}^{-}$ is closed.

Note that $\mathcal{Q}_{\epsilon,\epsilon'}^{-}$ is non-increasing
in $\epsilon'$, and hence, the operation $\lim_{\epsilon'\downarrow0}$
in \eqref{eq:-37} can be replaced by $\sup_{\epsilon'>0}$. Applying
Lemma \ref{lem:continuityofinf}, we obtain that 
\begin{equation}
\sup_{\epsilon'>0}\inf_{Q_{X}\in\mathcal{Q}_{\epsilon,\epsilon'}^{-}}D(Q_{X}\|P_{X})=\inf_{Q_{X}\in\mathcal{Q}_{\epsilon}^{-}}D(Q_{X}\|P_{X}).\label{eq:-53-1}
\end{equation}
where $\mathcal{Q}_{\epsilon}^{-}:=\bigcap_{\epsilon'>0}\mathcal{Q}_{\epsilon,\epsilon'}^{-}$.
It is easily seen that 
\begin{align*}
\mathcal{Q}_{\epsilon}^{-} & =\{Q_{X}\in B_{\epsilon/2}(P_{X})^{c}:\,\phi_{Q_{X}}^{-}(\alpha-\kappa(\epsilon))\geq D(Q_{X}\|P_{X})\}\;\subseteq\;\tilde{\mathcal{Q}}_{\epsilon}^{-}\backslash\{P_{X}\},
\end{align*}
where $\tilde{\mathcal{Q}}_{\epsilon}^{-}:=\{Q_{X}:\,\phi_{Q_{X}}^{-}(\alpha-\kappa(\epsilon))\geq D(Q_{X}\|P_{X})\}$
is closed and non-decreasing in $\epsilon>0$. By Lemma \ref{lem:continuityofinf}
and \eqref{eq:-37}, 
\begin{align}
\underline{\mathrm{E}}_{X}(\alpha) & \geq\sup_{\epsilon>0}\inf_{Q_{X}\in\tilde{\mathcal{Q}}_{\epsilon}^{-}\backslash\{P_{X}\}}D(Q_{X}\|P_{X})=\inf_{Q_{X}\in\mathcal{Q}^{-}}D(Q_{X}\|P_{X}),\label{eq:-37-2}
\end{align}
where $\mathcal{Q}^{-}:=\bigcap_{\epsilon>0}\overline{\tilde{\mathcal{Q}}_{\epsilon}^{-}\backslash\{P_{X}\}}$.

Let $\epsilon$ be small enough such that $\mathcal{E}(P_{X},P_{Y})<\alpha-\kappa(\epsilon)$.
We now claim that for such $\epsilon$, $\tilde{\mathcal{Q}}_{\epsilon}^{-}\backslash\{P_{X}\}$
is closed. We next prove this claim. First observe that for any sequence
$\{Q_{X}^{(k)}\}\subseteq\tilde{\mathcal{Q}}_{\epsilon}^{-}\backslash\{P_{X}\}$
such that $Q_{X}^{(k)}\to Q_{X}$ as $k\to\infty$, we have $Q_{X}\in\tilde{\mathcal{Q}}_{\epsilon}^{-}$
since $\tilde{\mathcal{Q}}_{\epsilon}^{-}$ is closed. Hence, it suffices
to prove $Q_{X}\neq P_{X}$, which will be proven by contradiction
in the following. Suppose $Q_{X}=P_{X}$. Then, by the assumption
of UCOTF, $\mathcal{E}(Q_{X}^{(k)},P_{Y})\to\mathcal{E}(P_{X},P_{Y})<\alpha-\kappa(\epsilon)$.
Hence, for all sufficiently large $k$, it always holds that $\phi_{Q_{X}^{(k)}}^{-}(\alpha-\kappa(\epsilon))=0$,
i.e., for this case, $Q_{Y}=P_{Y}$ is a feasible (and also optimal)
solution to the infimization in the definition of $\phi_{Q_{X}^{(k)}}^{-}(\alpha-\kappa(\epsilon))$.
However, since $Q_{X}^{(k)}\neq P_{X}$, we have $D(Q_{X}^{(k)}\|P_{X})>0$
for all $k$. Hence, for sufficiently large $k$, $Q_{X}^{(k)}\notin\tilde{\mathcal{Q}}_{\epsilon}^{-}$.
This contradict with the choice of the sequence $\{Q_{X}^{(k)}\}$.
Hence, $Q_{X}\neq P_{X}$, which in turn implies that $Q_{X}\in\tilde{\mathcal{Q}}_{\epsilon}^{-}\backslash\{P_{X}\}$.
Since the convergent sequence $\{Q_{X}^{(k)}\}$ is arbitrarily chosen,
we have that $\tilde{\mathcal{Q}}_{\epsilon}^{-}\backslash\{P_{X}\}$
is closed, completing the proof of the claim.

In fact, the set $\tilde{\mathcal{Q}}_{\epsilon}^{-}$ consists of
two disjoint closed subsets $\tilde{\mathcal{Q}}_{\epsilon}^{-}\backslash\{P_{X}\}$
and $\{P_{X}\}$. The subset $\tilde{\mathcal{Q}}_{\epsilon}^{-}\backslash\{P_{X}\}$
is the ``reasonable'' feasible region for the infimization in \eqref{eq:-37-2};
see Obstacle \ref{enu:Although-the-expression}. Here we address Obstacle
\ref{enu:Although-the-expression} by excluding $P_{X}$ from $\tilde{\mathcal{Q}}_{\epsilon}^{-}$.
In the following, we show that by doing this, the resultant lower
bound turns into $g(\alpha)$.

By the claim above, we can write $\mathcal{Q}^{-}=\bigcap_{\epsilon>0}\big(\tilde{\mathcal{Q}}_{\epsilon}^{-}\backslash\{P_{X}\}\big)$.
It is easily seen that 
\begin{align*}
\tilde{\mathcal{Q}}_{\epsilon}^{-} & \subseteq\{Q_{X}:\,\phi_{Q_{X}}(\alpha-2\kappa(\epsilon))\geq D(Q_{X}\|P_{X})\},
\end{align*}
which follows since $\phi_{Q_{X}}^{-}(t)\le\varphi_{Q_{X}}(t-\delta)$
for any $\delta>0$, and here we set $\delta=\kappa(\epsilon)$.

By an equivalence similar to \eqref{eqn:equicont}, we have 
\begin{align}
\tilde{\mathcal{Q}}_{\epsilon}^{-} & \subseteq\{Q_{X}:\inf_{Q_{Y}:D(Q_{Y}\|P_{Y})<D(Q_{X}\|P_{X})}\mathcal{E}(Q_{X},Q_{Y})\geq\alpha-2\kappa(\epsilon)\}.\label{eq:-6}
\end{align}
Define $\psi_{Q_{X}}(t):=\inf_{Q_{Y}:D(Q_{Y}\|P_{Y})\leq t}\mathcal{E}(Q_{X},Q_{Y})$
for $t\ge0$. Define $\psi_{Q_{X}}^{-}(t):=\lim_{s\uparrow t}\psi_{Q_{X}}(s)$
for $t>0$. Then, $\tilde{\mathcal{Q}}_{\epsilon}^{-}\backslash\{P_{X}\}\subseteq\{Q_{X}:\psi_{Q_{X}}^{-}(D(Q_{X}\|P_{X}))\geq\alpha-2\kappa(\epsilon)\},$
which implies that $\mathcal{Q}^{-}\subseteq\{Q_{X}:\psi_{Q_{X}}^{-}(D(Q_{X}\|P_{X}))\geq\alpha\}.$
Therefore, 
\begin{align}
\underline{\mathrm{E}}_{X}(\alpha) & \geq\inf_{Q_{X}\neq P_{X}:\psi_{Q_{X}}^{-}(D(Q_{X}\|P_{X}))\geq\alpha}D(Q_{X}\|P_{X}).\label{eq:-37-2-2}
\end{align}

By the assumption of UCOTF and applying the triangle inequality, we
have that $\mathcal{E}(Q_{X},Q_{Y})$ is finite for all $Q_{X},Q_{Y}$.
Hence, $\psi_{Q_{X}}$ is finite on $[0,+\infty)$. Since the sublevel
sets of the relative entropy are compact, by Lemma \ref{lem:continuityofinf2},
$\psi_{Q_{X}}$ is continuous on $[0,+\infty)$, which implies that
$\psi_{Q_{X}}^{-}=\psi_{Q_{X}}$ on $[0,+\infty)$. Hence, we have
\begin{align}
\underline{\mathrm{E}}_{X}(\alpha) & \geq\inf_{Q_{X}:\psi_{Q_{X}}(D(Q_{X}\|P_{X}))\geq\alpha}D(Q_{X}\|P_{X})\ge\lim_{\alpha'\uparrow\alpha}g_{P_{X},P_{Y}}(\alpha').\label{eq:-37-2-2-1}
\end{align}
By symmetry, we obtain 
\[
\underline{\mathrm{E}}_{Y}(\alpha):=\liminf_{n\to\infty}-\frac{1}{n}\log\Upsilon_{2}\geq\lim_{\alpha'\uparrow\alpha}g_{P_{Y},P_{X}}(\alpha').
\]
Therefore, by \eqref{eqn:maxUp}, 
\[
\liminf_{n\to\infty}-\frac{1}{n}\log\mathcal{G}_{\alpha}^{(n)}(P_{X},P_{Y})\ge\lim_{\alpha'\uparrow\alpha}g(\alpha').
\]

\subsubsection{Upper Bound}

We next prove the direction of ``$\le$'' in \eqref{eq:LD2}, i.e.,
$g(\alpha)$ is an upper bound on the left side of \eqref{eq:LD2}.
For this case, \eqref{eqn:maxUp} still holds. Setting $A=B_{\le\epsilon}(Q_{X}):=\{Q_{X}':\mathtt{L}_{1}(Q_{X}',Q_{X})\le\epsilon\}$
for some fixed $Q_{X}$ and some $\epsilon>0$ and then applying Sanov's
theorem \cite[Theorem 6.2.10]{Dembo} to $\mu_{n}(A)$ and $\nu_{n}(\Gamma_{\mathcal{E}\le d}(A))$,
we obtain that 
\begin{align}
\Upsilon_{1}= & \sup_{\textrm{compact }A\subseteq\mathcal{P}(\mathcal{X}):\mu_{n}(A)\le\frac{1}{2}+\delta}\mu_{n}(A)-\nu_{n}(\Gamma_{\mathcal{E}\le\alpha}(A))\nonumber \\
 & \geq e^{-n(\inf_{Q_{X}'\in B_{\le\epsilon}(Q_{X})^{o}}D(Q_{X}'\|P_{X})+o_{n|Q_{X},\epsilon}(1))}-e^{-n(\inf_{Q_{Y}\in\overline{\Gamma_{\mathcal{E}\le\alpha}(B_{\le\epsilon}(Q_{X}))}}D(Q_{Y}\|P_{Y})+o_{n|Q_{X},\epsilon}(1))}.\label{eq:-58-2}
\end{align}
Since $Q_{X}\in B_{\le\epsilon}(Q_{X})^{o}$, we have 
\begin{equation}
\inf_{Q_{X}'\in B_{\le\epsilon}(Q_{X})^{o}}D(Q_{X}'\|P_{X})\leq D(Q_{X}\|P_{X}).\label{eq:-104}
\end{equation}
On the other hand, by the assumption of UCOTF, 
\begin{align}
\Gamma_{\mathcal{E}\le\alpha}(B_{\le\epsilon}(Q_{X})) & =\{Q_{Y}:\exists Q_{X}'\in B_{\le\epsilon}(Q_{X}),\mathcal{E}(Q_{X}',Q_{Y})\leq\alpha\}\subseteq\{Q_{Y}:\mathcal{E}(Q_{X},Q_{Y})\leq\alpha+\kappa(\epsilon)\},\label{eq:-105}
\end{align}
where $\kappa:(0,+\infty)\to(0,+\infty)$ is some positive and increasing
function such that $\lim_{\epsilon\downarrow0}\kappa(\epsilon)=0$.
By Lemma \ref{lem:OTconvexitycontinuity} (or the assumption of UCOTF),
the set at the most right-hand side above is closed. Therefore, 
\begin{equation}
\overline{\Gamma_{\mathcal{E}\le\alpha}(B_{\le\epsilon}(Q_{X}))}\subseteq\{Q_{Y}:\mathcal{E}(Q_{X},Q_{Y})\leq\alpha+\kappa(\epsilon)\}.\label{eq:-109}
\end{equation}
Substituting \eqref{eq:-104} and \eqref{eq:-109} into \eqref{eq:-58-2}
yields that 
\begin{align}
 & \Upsilon_{1}\geq e^{-n(D(Q_{X}\|P_{X})+o_{n|Q_{X},\epsilon}(1))}-e^{-n(\phi_{Q_{X}}(\alpha+\kappa(\epsilon))+o_{n|Q_{X},\epsilon}(1))}.\label{eq:-104-1}
\end{align}
where $\phi_{Q_{X}}(t):=\inf_{Q_{Y}:\mathcal{E}(Q_{X},Q_{Y})\leq t}D(Q_{Y}\|P_{Y})$
for $t\ge0$.

Now we choose 
\begin{equation}
Q_{X}\in\hat{\mathcal{Q}}:=\big\{ Q_{X}:\phi_{Q_{X}}(\alpha)>D(Q_{X}\|P_{X})\big\},\label{eqn:Q}
\end{equation}
which means that $D(Q_{X}\|P_{X})$ is finite for all $Q_{X}\in\hat{\mathcal{Q}}$.
Then given each $Q_{X}\in\hat{\mathcal{Q}}$, for all sufficiently
small $\epsilon>0$, 
\begin{equation}
\phi_{Q_{X}}(\alpha+\kappa(\epsilon))\geq D(Q_{X}\|P_{X})+\epsilon,\label{eq:-117}
\end{equation}
which follows by the right-continuity of $\phi_{Q_{X}}$; see Lemma
\ref{lem:continuityofinf2}. Fixing a pair $(Q_{X},\epsilon)$ satisfying
\eqref{eqn:Q} and \eqref{eq:-117}, and letting $n\to\infty$ in
\eqref{eq:-104-1}, we obtain 
\begin{align}
\overline{\mathrm{E}}_{X}(\alpha) & :=\limsup_{n\to\infty}-\frac{1}{n}\log\Upsilon_{1}\leq D(Q_{X}\|P_{X}).\label{eq:-58-1-1-1}
\end{align}
Since $Q_{X}\in\mathcal{Q}$ is arbitrary, we take infimum over all
$Q_{X}\in\hat{\mathcal{Q}}$. Then we obtain 
\begin{align}
\overline{\mathrm{E}}_{X}(\alpha) & \leq\inf_{Q_{X}\in\hat{\mathcal{Q}}}D(Q_{X}\|P_{X}).\label{eq:-57}
\end{align}

We now bound $\hat{\mathcal{Q}}$. By an equivalence similar to \eqref{eqn:equicont},
we have 
\begin{align}
\hat{\mathcal{Q}} & \supseteq\hat{\mathcal{Q}}_{1}:=\bigcup_{\epsilon>0}\big\{ Q_{X}:\psi_{Q_{X}}(D(Q_{X}\|P_{X})+\epsilon)>\alpha\big\},\label{eq:-72-2}
\end{align}
where $\psi_{Q_{X}}(t):=\inf_{Q_{Y}:D(Q_{Y}\|P_{Y})\leq t}\mathcal{E}(Q_{X},Q_{Y})$
for $t\ge0$. Indeed the union operation in \eqref{eq:-72-2} can
be replaced by $\lim_{\epsilon\downarrow0}$, since the set inside
the union operation is nonincreasing in $\epsilon>0$. We now remove
the union operation in \eqref{eq:-72-2}. 
We claim that 
\begin{equation}
\hat{\mathcal{Q}}_{1}=\hat{\mathcal{Q}}_{2}:=\big\{ Q_{X}:\psi_{Q_{X}}(D(Q_{X}\|P_{X}))>\alpha\big\}.\label{eq:-61}
\end{equation}
We next prove this claim.

By the monotonicity of $\psi_{Q_{X}}$, $\hat{\mathcal{Q}}_{1}\subseteq\hat{\mathcal{Q}}_{2}$.
We next prove the other direction. By Lemma \ref{lem:continuityofinf2},
given $Q_{X}$, $\psi_{Q_{X}}$ is right-continuous. Hence, for $Q_{X}\in\hat{\mathcal{Q}}_{2}$,
\begin{align*}
 & \lim_{\epsilon\downarrow0}\psi_{Q_{X}}(D(Q_{X}\|P_{X})+\epsilon)=\psi_{Q_{X}}(D(Q_{X}\|P_{X}))>\alpha.
\end{align*}
Hence, given $Q_{X}\in\hat{\mathcal{Q}}_{2}$, for all sufficiently
small $\epsilon>0$, 
\begin{align*}
\psi_{Q_{X}}(D(Q_{X}\|P_{X})+\epsilon) & >\alpha.
\end{align*}
That is, $Q_{X}\in\hat{\mathcal{Q}}_{1}$, which implies \eqref{eq:-61},
i.e., the claim above.

Combining \eqref{eq:-72-2} and \eqref{eq:-61} yields $\hat{\mathcal{Q}}_{2}\subseteq\hat{\mathcal{Q}}.$
Then, combining this with \eqref{eq:-57}, we have $\overline{\mathrm{E}}_{X}(\alpha)\leq g_{P_{X},P_{Y}}(\alpha).$
By symmetry, we obtain $\overline{\mathrm{E}}_{Y}(\alpha):=\limsup_{n\to\infty}-\frac{1}{n}\log\Upsilon_{2}\leq g_{P_{Y},P_{X}}(\alpha).$
Therefore, by \eqref{eqn:maxUp}, 
\[
\limsup_{n\to\infty}-\frac{1}{n}\log\mathcal{G}_{\alpha}^{(n)}(P_{X},P_{Y})\leq g(\alpha).
\]

\section{Proof of Theorem \protect\ref{thm:MDP}}

Before proving Theorem \protect\ref{thm:MDP}, we need introduce
two lemmas on properties of $\theta(\beta_{X},\beta_{Y})$. The first
lemma is the following which shows that $\theta(\beta_{X},\beta_{Y})$
is Lipschitz continuous on $\mathbb{S}_{X}\times\mathbb{S}_{Y}$.
The proof of this lemma is provided in Appendix \ref{sec:Proof-of-Lemma-continuity}. 
\begin{lem}
\label{LEM:CONTINUITY} The function $\theta(\beta_{X},\beta_{Y})$
is uniformly continuous on $\mathbb{S}_{X}\times\mathbb{S}_{Y}$.
More precisely, 
\begin{align}
|\theta(\beta_{X},\beta_{Y})-\theta(\beta_{X}',\beta_{Y}')| & \leq C\max\{\Vert\beta_{X}-\beta_{X}'\Vert_{\infty},\Vert\beta_{Y}-\beta_{Y}'\Vert_{\infty}\},\label{eqn:Lipcont}
\end{align}
where $C>0$ is a constant only depending on $P_{X},P_{Y},$ and $c$. 
\end{lem}
Based on Lemma \ref{LEM:CONTINUITY}, in the following lemma, we show
that the functional $(\beta_{X},\beta_{Y})\mapsto\theta(\beta_{X},\beta_{Y})$
corresponds to the directional derivative of $(Q_{X},Q_{Y})\mapsto\mathcal{E}(Q_{X},Q_{Y})$.
The proof is provided in Appendix \ref{sec:Proof-of-Lemma-setA}. 
\begin{lem}
\label{LEM:SETA} Denote $\alpha_{0}=\mathcal{E}(P_{X},P_{Y})$. Then
the following hold. 
\begin{enumerate}
\item For a pair of distributions $(Q_{X},Q_{Y})$ and a number $t>0$,
we have 
\begin{align}
\frac{\mathcal{E}(Q_{X},Q_{Y})-\alpha_{0}}{t} & \geq\theta(\beta_{X},\beta_{Y}),\label{eq:-88}
\end{align}
where $\beta_{X}:=\frac{Q_{X}-P_{X}}{t}$ and $\beta_{Y}:=\frac{Q_{Y}-P_{Y}}{t}.$ 
\item For any $(\beta_{X},\beta_{Y})$, we have 
\begin{equation}
\limsup_{t\downarrow0}\frac{\mathcal{E}(P_{X}+t\beta_{X},P_{Y}+t\beta_{Y})-\alpha_{0}}{t}\leq\theta(\beta_{X},\beta_{Y}).\label{eq:-89}
\end{equation}
Moreover, for a pair of bounded subsets $A\subseteq\mathbb{S}_{X},B\subseteq\mathbb{S}_{Y}$
(under the relative topologies), we have 
\begin{equation}
\limsup_{t\downarrow0}\sup_{\beta_{X}\in A,\beta_{Y}\in B}\Big(\frac{\mathcal{E}(P_{X}+t\beta_{X},P_{Y}+t\beta_{Y})-\alpha_{0}}{t}-\theta(\beta_{X},\beta_{Y})\Big)\leq0.\label{eq:-89-1}
\end{equation}
\end{enumerate}
\end{lem}
Note that the differentiability of $t\mapsto\mathcal{E}(P_{X}+t\beta_{X},P_{Y}+t\beta_{Y})$
at $t=0$ can be also proven by the theorems on Hadamard directional
differentiability in \cite{bonnans2013perturbation}, as done in \cite{tameling2019empirical}.
However, here we require a stronger condition, the ``uniform differentiability''
given in \eqref{eq:-88} and \eqref{eq:-89-1}. This restricts our
attention on finite alphabets. However, it is interesting to investigate
how to extend our proof to infinite alphabets, which remains to be
done in the future.

The proof of Theorem \ref{thm:MDP} is in fact almost the same as
the proof of Theorem \ref{thm:LDP}, except that the quantities $\mathcal{E}(P_{X},P_{Y})$
and $D(Q_{X}\|P_{X})$ are respectively replaced by $\theta(\beta_{X},\beta_{Y})$
and $\frac{1}{2}\sum_{x}\frac{\beta_{X}(x)^{2}}{P_{X}(x)}$. The feasibility
of the first replacement follows by Lemma \ref{LEM:SETA} and the
feasibility of the second one follows by the moderate deviation theorem
in \cite{wu1994large} or \cite[Theorem 3.7.1]{Dembo}. We omit the
detailed proof here.

\section{Proof of Theorem \protect\ref{thm:CLT}}

\label{sec:Proof-of-Theorem-CLT}

For a discrete distribution, it is uniquely determined by its probability
mass function (pmf). Moreover, a pmf can be thought of as a vector
$(P_{X}(i))_{i\in[M]}$ where $\mathcal{X}=[M]$. Hence, the empirical
measure $T_{X^{n}}$ with $X^{n}\sim P_{X}^{\otimes n}$ corresponds
a random vector $(T_{X^{n}}(i))_{i\in[M]}$. Denote $\hat{\mu}_{n}$
as the law of $\sqrt{n}\cdot(T_{X^{n}}(i)-P_{X}(i))_{i\in[M]}$. We
extend the law $\hat{\mu}_{n}$ to the space $\mathcal{M}_{1}(\mathcal{X})$
of signed measures with total measure $1$ by taking $\hat{\mu}_{n}(A)=\hat{\mu}_{n}(A\cap\mathcal{M}_{1}(\mathcal{X}))$
for measurable $A\subseteq\mathcal{M}_{1}(\mathcal{X})$. By the multivariate
central limits theorem, the distribution $\hat{\mu}_{n}$ converges
weakly to the Gaussian distribution $\Phi_{P_{X}}$ given in Section
\ref{subsec:Main-Result-3}. Similarly, denote $\hat{\nu}_{n}$ as
the law of $\sqrt{n}\cdot(T_{Y^{n}}(i)-P_{Y}(i))_{i\in[N]}$ with
$Y^{n}\sim P_{Y}^{\otimes n}$, and extend the law $\hat{\nu}_{n}$
to the space $\mathcal{M}_{1}(\mathcal{Y})$. Then, $\hat{\nu}_{n}$
converges weakly to the Gaussian distribution $\Phi_{P_{Y}}$.

\subsection{Lower bound}

Choose $A'\subseteq\mathbb{S}_{X}$ and $B'\subseteq\mathbb{S}_{Y}$
as closed sets such that $\theta(\beta_{X},\beta_{Y})>\Delta$ for
all $\beta_{X}\in A',\beta_{Y}\in B'$, and 
\begin{equation}
\Phi_{P_{X}}(A')+\Phi_{P_{Y}}(B')-1\ge\Lambda_{\Delta}(P_{X},P_{Y})-\epsilon.\label{eq:-29}
\end{equation}
We obtain that 
\begin{align}
\liminf_{n\to\infty}\mathcal{G}_{\alpha_{0}+{\Delta}/{\sqrt{n}}}^{(n)}(P_{X},P_{Y}) & =\liminf_{n\to\infty}\sup_{\substack{\textrm{closed }A\subseteq\mathcal{P}(\mathcal{X}),B\subseteq\mathcal{P}(\mathcal{Y}):\\
\mathcal{E}(Q_{X},Q_{Y})>\alpha_{0}+\frac{\Delta}{\sqrt{n}},\forall Q_{X}\in A,Q_{Y}\in B
}
}\mu_{n}(A)+\nu_{n}(B)-1\label{eqn:CLstrassen}\\
 & \geq\liminf_{n\to\infty}\sup_{\substack{\textrm{closed }A,B:\theta(\sqrt{n}(Q_{X}-P_{X}),\sqrt{n}(Q_{Y}-P_{Y}))>\Delta,\\
\forall Q_{X}\in A,Q_{Y}\in B
}
}\mu_{n}(A)+\nu_{n}(B)-1\label{eq:-106}\\
 & \geq\liminf_{n\to\infty}\hat{\mu}_{n}(A')+\hat{\nu}_{n}(B')-1\label{eqn:Aprime}\\
 & \geq\Phi_{P_{X}}(A')+\Phi_{P_{Y}}(B')-1\label{eq:-14}\\
 & \geq\Lambda_{\Delta}(P_{X},P_{Y})-\epsilon,\label{eq:-15}
\end{align}
where \eqref{eqn:CLstrassen} follows by Strassen's duality, \eqref{eq:-106}
follows by Lemma \ref{LEM:SETA}, in \eqref{eqn:Aprime} we choose
$A=P_{X}+A'/\sqrt{n}$ and $B=P_{Y}+B'/\sqrt{n}$, and \eqref{eq:-14}
follows by the multivariate central limit theorem. Since \eqref{eq:-15}
holds for any $\epsilon>0$, we have 
\begin{align*}
\liminf_{n\to\infty}\mathcal{G}_{\alpha_{0}+{\Delta}/{\sqrt{n}}}^{(n)}(P_{X},P_{Y}) & \geq\Lambda_{\Delta}(P_{X},P_{Y}).
\end{align*}

\subsection{Upper bound}

The proof of the upper bound follows steps similar to those from \eqref{eq:-95}-\eqref{eq:-96}.
Note that both $\hat{\mu}_{n}$ and $\Phi_{P_{X}}$ are concentrated
on the hyperplane $\mathbb{S}_{X}$. 
 Obviously, $\Phi_{P_{X}}$ is \emph{tight} on $\mathbb{S}_{X}$ equipped
with relative topology, i.e., for any $\epsilon>0$, there is a compact
$K_{X}\subseteq\mathbb{S}_{X}$ such that $\Phi_{P_{X}}(K_{X}^{c})\le\epsilon$.
Since $K_{X}$ is compact, for any $\delta>0$, it has a finite cover
which consists of finitely many $\delta$-radius balls $\{B_{\delta}(\beta_{X,i})\}_{i=1}^{k_{1}}$.
Similarly, for $\hat{\nu}_{n}$, there is a compact $K_{Y}\subseteq\mathbb{S}_{Y}$
such that $\Phi_{P_{Y}}(K_{Y}^{c})\le\epsilon$, which has a finite
cover consisting of finitely many $\delta$-radius balls $\{B_{\delta}(\beta_{Y,i})\}_{i=1}^{k_{2}}$.

For any measurable $A\subseteq\mathbb{S}_{X},B\subseteq\mathbb{S}_{Y}$,
\begin{align*}
\hat{\mu}_{n}(A) & \le\hat{\mu}_{n}(A\cap K_{X})+\hat{\mu}_{n}(K_{X}^{c})\\
 & \le\hat{\mu}_{n}\Big(A\cap\bigcup_{i\in[k_{1}]}B_{\le\delta}(\beta_{X,i})\Big)+\hat{\mu}_{n}(K_{X}^{c}).
\end{align*}
For the second term in the last line, $\hat{\mu}_{n}(K_{X}^{c})\to\Phi_{P_{X}}(K_{X}^{c})\le\epsilon$
as $n\to\infty$. Define 
\[
\theta_{n}(\beta_{X},\beta_{Y}):=\sqrt{n}\left(\mathcal{E}(P_{X}+\beta_{X}/\sqrt{n},P_{Y}+\beta_{Y}/\sqrt{n})-\alpha_{0}\right).
\]
Then, by Strassen's duality in \eqref{eq:nesteddual}, 
\begin{align}
 & \mathcal{G}_{\alpha_{0}+{\Delta}/{\sqrt{n}}}^{(n)}(P_{X},P_{Y})+1\nonumber \\
 & =\sup_{\substack{\textrm{closed }A\subseteq\mathcal{P}(\mathcal{X}),B\subseteq\mathcal{P}(\mathcal{Y}):\\
\mathcal{E}(Q_{X},Q_{Y})>\alpha_{0}+\Delta/\sqrt{n},\forall Q_{X}\in A,Q_{Y}\in B
}
}\mu_{n}(A)+\nu_{n}(B)\nonumber \\
 & =\sup_{\substack{\textrm{closed }A\subseteq\mathbb{S}_{X},B\subseteq\mathbb{S}_{Y}:\\
\theta_{n}(\beta_{X},\beta_{Y})>\Delta,\forall\beta_{X}\in A,\beta_{Y}\in B
}
}\hat{\mu}_{n}(A)+\hat{\nu}_{n}(B)\nonumber \\
 & \leq\sup_{\substack{\textrm{closed }A\subseteq\mathbb{S}_{X},B\subseteq\mathbb{S}_{Y}:\\
\theta_{n}(\beta_{X},\beta_{Y})>\Delta,\forall\beta_{X}\in A,\beta_{Y}\in B
}
}\hat{\mu}_{n}\Big(A\cap\bigcup_{i\in[k_{1}]}B_{\le\delta}(\beta_{X,i})\Big)+\hat{\nu}_{n}\Big(B\cap\bigcup_{i\in[k_{2}]}B_{\le\delta}(\beta_{Y,i})\Big)+\hat{\mu}_{n}(K_{X}^{c})+\hat{\nu}_{n}(K_{Y}^{c})\nonumber \\
 & =\hat{\mu}_{n}(K_{X}^{c})+\hat{\nu}_{n}(K_{Y}^{c})+\sup_{\substack{\textrm{closed }A\subseteq\bigcup_{i\in[k_{1}]}B_{\le\delta}(\beta_{X,i}),B\subseteq\bigcup_{i\in[k_{2}]}B_{\le\delta}(\beta_{Y,i}):\\
\theta_{n}(\beta_{X},\beta_{Y})>\Delta,\forall\beta_{X}\in A,\beta_{Y}\in B
}
}\hat{\mu}_{n}(A)+\hat{\nu}_{n}(B).\label{eq:-5}
\end{align}

By Lemma \ref{LEM:SETA}, for bounded subsets $A\subseteq\mathbb{S}_{X},B\subseteq\mathbb{S}_{Y}$,
$\theta_{n}(\beta_{X},\beta_{Y})\to\theta(\beta_{X},\beta_{Y})$ as
$n\to\infty$ uniformly for all $\beta_{X}\in A,\beta_{Y}\in B$.
Hence, given $\epsilon'>0$, for any sufficiently large $n$, the
supremum term in \eqref{eq:-5} is upper bounded by a variant of this
supremum term in which ``$\theta_{n}(\beta_{X},\beta_{Y})>\Delta$''
is replaced by ``$\theta(\beta_{X},\beta_{Y})>\Delta-\epsilon'$''.

For sets $A,B$, we denote $L_{1}:=\{i\in[k_{1}]:B_{\le\delta}(\beta_{X,i})\cap A\neq\emptyset\}$
and $L_{2}:=\{i\in[k_{2}]:B_{\le\delta}(\beta_{Y,i})\cap B\neq\emptyset\}$.
Then, $A\subseteq\bigcup_{i\in L_{1}}B_{\le\delta}(\beta_{X,i})$
and $B\subseteq\bigcup_{i\in L_{2}}B_{\le\delta}(\beta_{Y,i})$. Moreover,
by the uniform continuity of $\theta$ (Lemma \ref{LEM:CONTINUITY}),
$|\theta(\beta_{X},\beta_{Y})-\theta(\beta_{X}',\beta_{Y}')|<o_{\delta}(1)$,
$\forall\beta_{X}'\in B_{\le\delta}(\beta_{X}),\beta_{Y}'\in B_{\le\delta}(\beta_{Y})$.
Hence, given $\epsilon',\delta>0$, the supremum term in \eqref{eq:-5}
is further upper bounded by 
\begin{align}
 & \max_{\substack{L_{1}\subseteq[k_{1}],L_{2}\subseteq[k_{2}]:\theta(\beta_{X},\beta_{Y})>\Delta-\epsilon'-o_{\delta}(1),\\
\forall\beta_{X}\in\bigcup_{i\in L_{1}}B_{\le\delta}(\beta_{X,i}),\beta_{Y}\in\bigcup_{i\in L_{2}}B_{\le\delta}(\beta_{Y,i})
}
}\hat{\mu}_{n}\Big(\bigcup_{i\in L_{1}}B_{\le\delta}(\beta_{X,i})\Big)+\hat{\nu}_{n}\Big(\bigcup_{i\in L_{2}}B_{\le\delta}(\beta_{Y,i})\Big)\nonumber \\
 & \le\sup_{\substack{\textrm{closed }A\subseteq\mathbb{S}_{X},B\subseteq\mathbb{S}_{Y}:\\
\theta_{n}(\beta_{X},\beta_{Y})>\Delta-\epsilon'-o_{\delta}(1),\forall\beta_{X}\in A,\beta_{Y}\in B
}
}\hat{\mu}_{n}(A)+\hat{\nu}_{n}(B).\label{eq:-12}
\end{align}

Therefore, substituting the upper bound in \eqref{eq:-12} into \eqref{eq:-5},
and taking limits, we have that given $\epsilon,\epsilon',\delta>0$,
\begin{align*}
 & \limsup_{n\to\infty}\mathcal{G}_{\alpha_{0}+{\Delta}/{\sqrt{n}}}^{(n)}(P_{X},P_{Y})+1\le2\epsilon+\sup_{\substack{\textrm{closed }A\subseteq\mathbb{S}_{X},B\subseteq\mathbb{S}_{Y}:\\
\theta_{n}(\beta_{X},\beta_{Y})>\Delta-\epsilon'-o_{\delta}(1),\forall\beta_{X}\in A,\beta_{Y}\in B
}
}\Phi_{P_{X}}(A)+\Phi_{P_{Y}}(B).
\end{align*}
Letting $\epsilon,\epsilon',\delta\downarrow0$, we obtain 
\begin{align*}
\limsup_{n\to\infty}\mathcal{G}_{\alpha_{0}+{\Delta}/{\sqrt{n}}}^{(n)}(P_{X},P_{Y}) & \leq\lim_{\Delta'\uparrow\Delta}\Lambda_{\Delta'}(P_{X},P_{Y}).
\end{align*}

\appendix

\section{Basic Lemmas}

\label{sec:Basic-Lemmas}

In this section, we prove several basic lemmas for the OT problem.
These lemmas will be used to prove our main results in Sections \ref{sec:Proof-of-Lemma-Product}-\ref{sec:Proof-of-Theorem-CLT}. 
\begin{lem}
\label{lem:OTconvexitycontinuity} Let $\mathcal{X}$ and $\mathcal{Y}$
be Polish spaces. Assume that the cost function $c$ satisfies the
lower semi-continuity assumption. Then for $(P_{X},P_{Y})\in\mathcal{P}(\mathcal{X})\times\mathcal{P}(\mathcal{Y})$,
we have that $\mathcal{E}(P_{X},P_{Y})$ is convex in $(P_{X},P_{Y})$
and lower semi-continuous in $(P_{X},P_{Y})$ in the weak topology. 
\end{lem}
\begin{proof} By definition, it is easy to verify that $\mathcal{E}(P_{X},P_{Y})$
is convex in $(P_{X},P_{Y})$; see \cite[Theorem 4.8]{villani2008optimal}.
We next prove the lower semi-continuity. 
For any sequence of $\{(P_{X}^{(n)},P_{Y}^{(n)})\}$ such that $(P_{X}^{(n)},P_{Y}^{(n)})\to(P_{X},P_{Y})$
in the weak topology, we have 
\begin{align*}
\liminf_{n\to\infty}\mathcal{E}(P_{X}^{(n)},P_{Y}^{(n)}) & =\liminf_{n\to\infty}\sup_{(\phi,\psi)\in C_{\mathrm{b}}(\mathcal{X})\times C_{\mathrm{b}}(\mathcal{Y}):\phi+\psi\leq c}\int_{\mathcal{X}}\phi\mathrm{d}P_{X}^{(n)}+\int_{\mathcal{Y}}\psi\mathrm{d}P_{Y}^{(n)}\\
 & \geq\sup_{(\phi,\psi)\in C_{\mathrm{b}}(\mathcal{X})\times C_{\mathrm{b}}(\mathcal{Y}):\phi+\psi\leq c}\liminf_{n\to\infty}\int_{\mathcal{X}}\phi\mathrm{d}P_{X}^{(n)}+\int_{\mathcal{Y}}\psi\mathrm{d}P_{Y}^{(n)}\\
 & =\sup_{(\phi,\psi)\in C_{\mathrm{b}}(\mathcal{X})\times C_{\mathrm{b}}(\mathcal{Y}):\phi+\psi\leq c}\int_{\mathcal{X}}\phi\mathrm{d}P_{X}+\int_{\mathcal{Y}}\psi\mathrm{d}P_{Y}\\
 & =\mathcal{E}(P_{X},P_{Y}).
\end{align*}
\end{proof} 
\begin{lem}
\label{lem:continuityofinf} Let $\mathcal{Z}$ be a compact set in
a topological space. Let $\epsilon\in(0,+\infty)\mapsto A_{\epsilon}\subseteq\mathcal{Z}$
be a set-valued function. Assume $A_{\epsilon}$ is closed for every
$\epsilon>0$, and non-decreasing in $\epsilon$ (i.e., $A_{\epsilon}\subseteq A_{\epsilon'}$
for all $\epsilon<\epsilon'$). Let $f:\mathcal{Z}\to[0,+\infty]$
be a lower semi-continuous function. Then 
\begin{equation}
\sup_{\epsilon>0}\inf_{z\in A_{\epsilon}}f(z)=\inf_{z\in\bigcap_{\epsilon>0}A_{\epsilon}}f(z).\label{eq:-91}
\end{equation}
\end{lem}
\begin{proof} Obviously, 
\begin{equation}
\sup_{\epsilon>0}\inf_{z\in A_{\epsilon}}f(z)\leq\inf_{z\in\bigcap_{\epsilon>0}A_{\epsilon}}f(z).\label{eq:-122}
\end{equation}
Hence we only need to prove 
\begin{equation}
\sup_{\epsilon>0}\inf_{z\in A_{\epsilon}}f(z)\geq\inf_{z\in\bigcap_{\epsilon>0}A_{\epsilon}}f(z).\label{eqn:ge}
\end{equation}

By definition, both the operations ``$\sup_{\epsilon>0}$'' and
``$\bigcap_{\epsilon>0}$'' in \eqref{eq:-91} can be replaced by
``$\lim_{\epsilon\downarrow0}$''. In particular, 
\begin{equation}
\sup_{\epsilon>0}\inf_{z\in A_{\epsilon}}f(z)=\lim_{\epsilon\downarrow0}\inf_{z\in A_{\epsilon}}f(z),\label{eq:-48}
\end{equation}
since $A_{\epsilon}$ is non-decreasing in $\epsilon$. Let $\{\epsilon_{n}\}$
be a decreasing positive sequence such that $\lim_{n\to\infty}\epsilon_{n}=0$
and 
\begin{equation}
\lim_{n\to\infty}\inf_{z\in A_{\epsilon_{n}}}f(z)=\lim_{\epsilon\downarrow0}\inf_{z\in A_{\epsilon}}f(z).\label{eq:-43}
\end{equation}
Let $\delta>0$ be a positive number. We denote $\{z_{n}\in A_{\epsilon_{n}}:n\in\mathbb{N}\}$
as a sequence such that for each $n$, 
\begin{equation}
f(z_{n})\leq\inf_{z\in A_{\epsilon_{n}}}f(z)+\delta.\label{eq:-42}
\end{equation}
Since $\mathcal{Z}$ is compact, we can pass the sequence $\{z_{n}:n\in\mathbb{N}\}$
into a convergent subsequence, and assume the limit of this subsequence
is $\hat{z}\in\mathcal{Z}$. By the monotonicity and closedness of
$A_{\epsilon}$ in $\epsilon$, we have $\hat{z}\in A_{\epsilon}$
for any $\epsilon>0$, which further implies 
\begin{equation}
\hat{z}\in\bigcap_{\epsilon>0}A_{\epsilon}.\label{eq:-46}
\end{equation}
Therefore, 
\begin{align}
\sup_{\epsilon>0}\inf_{z\in A_{\epsilon}}f(z) & =\lim_{n\to\infty}\inf_{z\in A_{\epsilon_{n}}}f(z)\geq\liminf_{n\to\infty}f(z_{n})-\delta\ge f(\hat{z})-\delta\geq\inf_{z\in\bigcap_{\epsilon>0}A_{\epsilon}}f(z)-\delta,\label{eq:-47}
\end{align}
where the equality follows from \eqref{eq:-48} and \eqref{eq:-43},
the first inequality follows from \eqref{eq:-42}, the second inequality
follows by the lower semi-continuity of $f$, and the last inequality
follows from \eqref{eq:-46}. Since $\delta>0$ is arbitrary, we obtain
\eqref{eqn:ge}. \end{proof} 
\begin{lem}
\label{lem:continuityofinf2} Let $\mathcal{Z}$ be a convex set.
Let $f,g:\mathcal{Z}\to[0,+\infty]$ be convex functions. Define 
\begin{equation}
F:t\in[0,+\infty)\mapsto\inf_{z\in\mathcal{Z}:g(z)\le t}f(z).\label{eq:-25}
\end{equation}
Denote $t_{\inf}:=\inf\{t\in[0,+\infty):F(t)<+\infty\}$. Then, the
following three statements hold. 
\begin{enumerate}
\item $F$ is non-increasing and convex on $[0,+\infty)$, and continuous
on $(t_{\inf},+\infty)$. 
\item If additionally, $\mathcal{Z}$ is a compact topological space and
$f,g$ are lower semi-continuous, then $F$ is continuous on $[t_{\inf},+\infty)$. 
\item If additionally, $\mathcal{Z}$ is a topological space, $f,g$ are
lower semi-continuous, and any sublevel set of $f$ or $g$ is a compact
subset of $\mathcal{Z}$, then $F$ is continuous on $[t_{\inf},+\infty)$. 
\end{enumerate}
\end{lem}
\begin{rem}
In the second and third statements, $F$ is in fact right-continuous
on $[0,+\infty)$. 
\end{rem}
\begin{proof} We first prove the first statement. By definition,
it is easy to verify that $F$ is nonincreasing and convex on $[0,+\infty)$.
Furthermore, any convex function is continuous on any open interval
on which it is finite. Hence $F$ is continuous on $(t_{\inf},+\infty)$.

We next prove the second statement. To this end, we only need to show
that $F$ is right-continuous at $t=t_{\inf}$. For a sequence $\{t_{k}\}$
such that $t_{k}\downarrow t_{\inf}$ as $k\to\infty$ and for any
given $\delta>0$, one can find a sequence $\{z_{k}\}$ such that
$g(z_{k})\le t_{k}$ and $f(z_{k})\le F(t_{k})+\delta$. If additionally,
$\mathcal{Z}$ is compact, then we can pass $\{z_{k}\}$ to its a
convergent subsequence with the limit denoted by $\hat{z}$. For this
limit $\hat{z}$, by the lower semi-continuity of $f,g$, we have
$g(\hat{z})\le\liminf_{k\to\infty}g(z_{k})\le t_{\inf}$ and $f(\hat{z})\le\liminf_{k\to\infty}f(z_{k})\le\lim_{t\downarrow t_{\inf}}F(t)+\delta$,
which imply that $\hat{z}$ is a feasible solution for the case $t=t_{\inf}$.
Hence, $F(t_{\inf})\le\lim_{t\downarrow t_{\inf}}F(t)+\delta$. Since
$\delta>0$ is arbitrary, we have $F(t_{\inf})\le\lim_{t\downarrow t_{\inf}}F(t)$.
On the other hand, by the monotonicity of $F$, $F(t_{\inf})\ge\lim_{t\downarrow t_{\inf}}F(t)$.
Therefore, $F(t_{\inf})=\lim_{t\downarrow t_{\inf}}F(t)$, i.e., $F$
is right-continuous at $t=t_{\inf}$.

We now prove the last statement. We first assume that any sublevel
set of $f$ is a compact subset of $\mathcal{Z}$. If $\lim_{t\downarrow t_{\inf}}F(t)=+\infty$,
then by monotonicity of $F$, $F(t_{\inf})=+\infty$ and hence, $F$
is right-continuous at $t=t_{\inf}$. If $\lim_{t\downarrow t_{\inf}}F(t)<+\infty$,
then without loss of optimality, one can replace the constraint $z\in\mathcal{Z}$
with $z\in\mathcal{A}_{r}:=\{z:f(z)\le r\}$ for $r>\lim_{t\downarrow t_{\inf}}F(t)$
in the constraints in the infimization in \eqref{eq:-25}. In other
words, for any $t>t_{\inf}$, 
\[
F(t)=\inf_{z\in\mathcal{A}_{r}:g(z)\le t}f(z).
\]
Since $\mathcal{A}_{r}$ is compact, applying the second statement,
we have that $F$ is right-continuous at $t=t_{\inf}$.

We next assume that any sublevel set of $g$ is a compact subset of
$\mathcal{Z}$. Similarly to the above, we only need to consider the
case $\lim_{t\downarrow t_{\inf}}F(t)<+\infty$. For this case, without
loss of optimality, for any $t\le r$ with some $r>t_{\inf}$, one
can replace the constraint $z\in\mathcal{Z}$ with $z\in\mathcal{B}_{r}:=\{z:g(z)\le r\}$
in the constraints in the infimization in \eqref{eq:-25}. In other
words, for any $t\le r$, 
\[
F(t)=\inf_{z\in\mathcal{B}_{r}:g(z)\le t}f(z).
\]
Since $\mathcal{B}_{r}$ is compact, applying the second statement,
we have that $F$ is right-continuous at $t=t_{\inf}$. \end{proof} 
\begin{lem}
\label{lem:marginalbound} For Polish spaces $\mathcal{X}$ and $\mathcal{Y}$,
let $P_{X},Q_{X}$ be two distributions on $\mathcal{X}$, and $P_{Y},Q_{Y}$
two distributions on $\mathcal{Y}$. Then for any $Q_{XY}\in\Pi(Q_{X},Q_{Y})$,
there exists $P_{XY}\in\Pi(P_{X},P_{Y})$ such that 
\begin{equation}
\Vert P_{XY}-Q_{XY}\Vert_{\mathrm{TV}}\leq\Vert P_{X}-Q_{X}\Vert_{\mathrm{TV}}+\Vert P_{Y}-Q_{Y}\Vert_{\mathrm{TV}}.\label{eq:-81}
\end{equation}
\end{lem}
\begin{proof} Let $Q_{X'X}\in\Pi(P_{X},Q_{X})$ and $Q_{Y'Y}\in\Pi(P_{Y},Q_{Y})$
be two couplings. Define $Q_{X'XYY'}=Q_{X'|X}Q_{XY}Q_{Y'|Y}$. (Such
a joint distribution is well-defined, since for Polish $\mathcal{X}$
and $\mathcal{Y}$, the regular conditional distributions $Q_{X'|X}$
and $Q_{Y'|Y}$ exist.) Hence $Q_{X'Y'}\in\Pi(P_{X},P_{Y}).$

On the other hand, the joint distribution $Q_{X'XYY'}$ constructed
above satisfies 
\begin{align}
 & Q_{X'XYY'}\{(x',x,y,y'):(x,y)\neq(x',y')\}\nonumber \\
 & \leq Q_{XX'}\{(x,x'):x\neq x'\}+Q_{YY'}\{(y,y'):y\neq y'\}.\label{eq:-90}
\end{align}
Taking infimum over all $Q_{X'X}\in\Pi(P_{X},Q_{X}),Q_{Y'Y}\in\Pi(P_{Y},Q_{Y})$
for both sides of \eqref{eq:-90}, we have 
\begin{align*}
 & \Vert Q_{X'Y'}-Q_{XY}\Vert_{\mathrm{TV}}\\
 & =\inf_{P_{X'Y'XY}\in\Pi(Q_{X'Y'},Q_{XY})}P_{X'XYY'}\{(x',x,y,y'):(x,y)\neq(x',y')\}\\
 & \leq\inf_{Q_{X'X}\in\Pi(P_{X},Q_{X}),Q_{Y'Y}\in\Pi(P_{Y},Q_{Y})}Q_{X'XYY'}\{(x',x,y,y'):(x,y)\neq(x',y')\}\\
 & \leq\inf_{Q_{X'X}\in\Pi(P_{X},Q_{X})}Q_{XX'}\{(x,x'):x\neq x'\}+\inf_{Q_{Y'Y}\in\Pi(P_{Y},Q_{Y})}Q_{YY'}\{(y,y'):y\neq y'\}\\
 & =\Vert P_{X}-Q_{X}\Vert_{\mathrm{TV}}+\Vert P_{Y}-Q_{Y}\Vert_{\mathrm{TV}},
\end{align*}
where the two equalities above follow by the maximal coupling equality
given in \eqref{eqn:maxcoupling}. Hence $Q_{X'Y'}$ is a desired
distribution.

\end{proof}

\section{Proof of Lemma \protect\ref{LEM:CONTINUITY}}

\label{sec:Proof-of-Lemma-continuity}

Obviously, for $(\beta_{X},\beta_{Y}),(\beta_{X}',\beta_{Y}')\in\mathbb{S}_{X}\times\mathbb{S}_{Y}$,
we have 
\begin{align}
 & \min_{\begin{subarray}{c}
\beta_{XY}\in\overline{\Pi}(\beta_{X},\beta_{Y}),\\
\{(x,y):\beta_{XY}(x,y)<0\}\subseteq\mathcal{S}
\end{subarray}}\sum_{x,y}\beta_{XY}(x,y)c(x,y)\nonumber \\
 & \leq\min_{\begin{subarray}{c}
\beta_{XY}\in\overline{\Pi}(\beta_{X}',\beta_{Y}'),\\
\{(x,y):\beta_{XY}(x,y)<0\}\subseteq\mathcal{S}
\end{subarray}}\sum_{x,y}\beta_{XY}(x,y)c(x,y)+\min_{\begin{subarray}{c}
\beta_{XY}\in\overline{\Pi}(\beta_{X}-\beta_{X}',\beta_{Y}-\beta_{Y}'),\\
\{(x,y):\beta_{XY}(x,y)<0\}\subseteq\mathcal{S}
\end{subarray}}\sum_{x,y}\beta_{XY}(x,y)c(x,y).\label{eq:-118}
\end{align}
Observe that 
\[
C:=\sup_{\Vert\hat{\beta}_{X}\Vert_{\infty},\Vert\hat{\beta}_{Y}\Vert_{\infty}\le1}\min_{\begin{subarray}{c}
\beta_{XY}\in\overline{\Pi}(\hat{\beta}_{X},\hat{\beta}_{Y}),\\
\{(x,y):\beta_{XY}(x,y)<0\}\subseteq\mathcal{S}
\end{subarray}}\sum_{x,y}\beta_{XY}(x,y)c(x,y)
\]
satisfies that $C<+\infty$. Otherwise, $\mathcal{E}(P_{X}+\epsilon\hat{\beta}_{X},P_{Y}+\epsilon\hat{\beta}_{Y})=+\infty$
holds for any $\epsilon>0$, which is impossible since $\mathcal{E}(P_{X},P_{Y})\leq c_{\sup}=\max_{x,y}c(x,y)<+\infty$
for any $(P_{X},P_{Y})$. By the upper boundness of $C$, we have
\begin{align}
 & \min_{\begin{subarray}{c}
\beta_{XY}\in\overline{\Pi}(\beta_{X}-\beta_{X}',\beta_{Y}-\beta_{Y}'),\\
\{(x,y):\beta_{XY}(x,y)<0\}\subseteq\mathcal{S}
\end{subarray}}\sum_{x,y}\beta_{XY}(x,y)c(x,y)\leq C\max\{\Vert\beta_{X}-\beta_{X}'\Vert_{\infty},\Vert\beta_{Y}-\beta_{Y}'\Vert_{\infty}\}.\label{eq:-119}
\end{align}

Combining \eqref{eq:-118} with \eqref{eq:-119} yields 
\begin{align*}
\theta(\beta_{X},\beta_{Y}) & \leq\theta(\beta_{X}',\beta_{Y}')+C\max\{\Vert\beta_{X}-\beta_{X}'\Vert_{\infty},\Vert\beta_{Y}-\beta_{Y}'\Vert_{\infty}\}.
\end{align*}
By symmetry, we can obtain 
\begin{align*}
\theta(\beta_{X}',\beta_{Y}') & \leq\theta(\beta_{X},\beta_{Y})+C\max\{\Vert\beta_{X}-\beta_{X}'\Vert_{\infty},\Vert\beta_{Y}-\beta_{Y}'\Vert_{\infty}\}.
\end{align*}
Therefore, \eqref{eqn:Lipcont} holds. 

\section{Proof of Lemma \protect\ref{LEM:SETA}}

\label{sec:Proof-of-Lemma-setA}

We first prove \eqref{eq:-88}. Denote $Q_{XY}^{*}$ as an optimal
distribution attaining $\mathcal{E}(Q_{X},Q_{Y})$. Recall that $P_{XY}^{*}$
is an optimal distribution attaining $\mathcal{E}(P_{X},P_{Y})$ with
support $\mathcal{S}$. For such $P_{XY}^{*}$ and $Q_{XY}^{*}$,
we can write $Q_{XY}^{*}=P_{XY}^{*}+t\beta_{XY}^{*},$ where $\beta_{XY}^{*}:=\frac{Q_{XY}^{*}-P_{XY}^{*}}{t}.$
Obviously, $\beta_{XY}^{*}\in\overline{\Pi}(\beta_{X},\beta_{Y})$
and $\{(x,y):\beta_{XY}^{*}(x,y)<0\}\subseteq\mathcal{S}.$ Therefore,
\begin{align*}
\mathcal{E}(Q_{X},Q_{Y}) & =\sum_{x,y}(P_{XY}^{*}(x,y)+t\beta_{XY}^{*}(x,y))c(x,y)\,=\,\alpha_{0}+t\sum_{x,y}\beta_{XY}^{*}(x,y)c(x,y)\,\geq\,\alpha_{0}+t\theta(\beta_{X},\beta_{Y}),
\end{align*}
where the inequality above follows by the definition of the function
$\theta$ in \eqref{eqn:theta}.

Next we prove \eqref{eq:-89}. Since $\mathcal{X}$ and $\mathcal{Y}$
are respectively the supports of $P_{X}$ and $P_{Y}$, given $\beta_{X},\beta_{Y}$
and for sufficiently small $t$, the measures $P_{X}+t\beta_{X},P_{Y}+t\beta_{Y}$
are two distributions. Hence, for sufficiently small $t$, by definition,
\begin{equation}
\mathcal{E}(P_{X}+t\beta_{X},P_{Y}+t\beta_{Y})=\min_{P_{XY}\in\Pi(P_{X}+t\beta_{X},P_{Y}+t\beta_{Y})}\mathbb{E}[c(X,Y)].\label{eq:-64}
\end{equation}
For $\epsilon>0$, denote $\beta_{XY}^{*}\in\overline{\Pi}(\beta_{X},\beta_{Y})$
as a bivariate function which $\epsilon$-approximately attains $\theta(\beta_{X},\beta_{Y})$
in the sense that $\{(x,y):\beta_{XY}^{*}(x,y)<0\}\subseteq\mathcal{S}$
and $\sum_{x,y}\beta_{XY}^{*}(x,y)c(x,y)\leq\theta(\beta_{X},\beta_{Y})+\epsilon.$
Now we set $P_{XY}^{(t)}=P_{XY}^{*}+t\beta_{XY}^{*}$. Then, for sufficiently
small $t$, $P_{XY}^{(t)}$ is a distribution, and moreover, $P_{XY}^{(t)}\in\Pi(P_{X}+t\beta_{X},P_{Y}+t\beta_{Y}).$
Hence for sufficiently large $n$, 
\begin{align}
\mathcal{E}(P_{X}+t\beta_{X},P_{Y}+t\beta_{Y}) & \leq\mathbb{E}_{P_{XY}^{(t)}}[c(X,Y)]=\alpha_{0}+t\sum_{x,y}\beta_{XY}^{*}(x,y)c(x,y)\leq\alpha_{0}+t(\theta(\beta_{X},\beta_{Y})+\epsilon).\label{eqn:alpha0}
\end{align}
Since $\epsilon>0$ is arbitrary, we obtain \eqref{eq:-89}.

Furthermore, the sets $A,B$ in \eqref{eq:-89-1} are assumed to be
bounded, which means that $\|\beta_{X}\|_{\infty},\|\beta_{Y}\|_{\infty}$
are bounded on $A,B$. Let $t$ be small enough and choose $\epsilon$
fixed for all $\beta_{X}\in A,\beta_{Y}\in B$, then the proof for
\eqref{eq:-89} still works, i.e., \eqref{eqn:alpha0} holds for all
$\beta_{X}\in A,\beta_{Y}\in B$. Letting $\epsilon\downarrow0$,
we obtain \eqref{eq:-89-1}.

\section*{Acknowledgments}

The author would like to thank Vincent Y. F. Tan for his comments
on the preliminary version of the manuscript, and also thank Yonglong
Li and Bo Wei for useful discussions. The author  would like to
thank the reviewers for their suggestions which greatly have improved the readability of this paper, and also thank one of them for pointing out the references \cite{fournier2015rate,weed2019sharp}.

\bibliographystyle{plain}
\bibliography{ref}

\end{document}